\documentclass[10pt]{amsart}

\usepackage[text={420pt,660pt},centering]{geometry}

\usepackage{esint,amssymb} 
\usepackage{graphicx}
\usepackage{MnSymbol}
\usepackage{mathtools} 
\usepackage[colorlinks=true, pdfstartview=FitV, linkcolor=blue, citecolor=blue, urlcolor=blue,pagebackref=false]{hyperref}
\usepackage{microtype}
\usepackage{bm}
\usepackage{dsfont}
\usepackage{enumitem}

\newtheorem{proposition}{Proposition}
\newtheorem{theorem}[proposition]{Theorem}
\newtheorem{lemma}[proposition]{Lemma}
\newtheorem{corollary}[proposition]{Corollary}

\newtheorem{remark}[proposition]{Remark}
\newtheorem{example}[proposition]{Example}
\theoremstyle{definition}
\newtheorem{definition}[proposition]{Definition}

\newtheorem*{example*}{Example}

\numberwithin{equation}{section}
\numberwithin{proposition}{section}
\numberwithin{figure}{section}
\numberwithin{table}{section}

\newcommand{\E}{\mathbb{E} }
\renewcommand{\P}{\mathbb{P}}

\newcommand{\R}{\mathbb{R}}
\newcommand{\N}{\mathbb{N}}

\renewcommand{\d}{\mathrm{d}}
\newcommand{\la}{\left\langle}\newcommand{\ra}{\right\rangle}
\newcommand{\eps}{\epsilon}
\newcommand{\cX}{\mathcal{X}}

\renewcommand{\H}{\mathsf{H}}

\newcommand{\cH}{\mathcal{H}}

\renewcommand{\S}{\mathbf{S}}

\newcommand{\cK}{\mathcal{K}}

\newcommand{\cL}{\mathcal{L}}
\newcommand{\cD}{\mathcal{D}}

\newcommand{\tr}{\mathsf{tr}}
\newcommand{\nn}{\mathbf{n}}
\newcommand{\cP}{\mathcal{P}}

\newcommand{\cM}{\mathcal{C}}
\newcommand{\cE}{\mathcal{E}}

\newcommand{\metric}{\mathbf{d}}
\renewcommand{\j}{{(j)}}
\renewcommand{\bar}{\overline}
\renewcommand{\tilde}{\widetilde}
\newcommand{\pj}{\mathrm{p}}
\newcommand{\lf}{\mathrm{l}}
\newcommand{\dom}{\mathsf{dom}\,}

\newcommand{\sF}{\mathsf{F}}
\DeclareMathOperator*{\esssup}{ess\,sup}
\newcommand{\bxi}{\boldsymbol{\xi}}

\newcommand{\itr}{\mathsf{int}\,}
\newcommand{\cl}{\mathsf{cl}\,}
\newcommand{\bd}{\mathsf{bd}\,}

\newcommand{\rank}{\mathsf{rank}}

\newcommand{\diag}{\mathsf{diag}}

\newcommand{\C}{\mathcal C}
\newcommand{\F}{\mathcal F}

\usepackage{ulem}
\usepackage{cancel}

\newcommand{\wbd}{\partial_{\mathsf{w}}}

\newcommand{\HJ}{\mathrm{HJ}}
\newcommand{\J}{\mathfrak{J}}
\newcommand{\Jgen}{\J_\mathsf{gen}}
\newcommand{\Junif}{\J_\mathsf{unif}}
\newcommand{\Jgood}{\J_\mathsf{good}}
\newcommand{\LHS}{\mathsf{LHS}}
\newcommand{\RHS}{\mathsf{RHS}}
\renewcommand{\hat}{\widehat}
\newcommand{\Law}{\mathsf{Law}}

\newcommand{\rv}[1]{\textcolor{black}{#1}}
\newcommand{\rrv}[1]{\textcolor{black}{#1}}

\author[Hong-Bin Chen]{Hong-Bin Chen}
\address[Hong-Bin Chen]{Courant Institute of Mathematical Sciences, New York University, New York, New York, USA}
\email{hbchen@cims.nyu.edu}

\author[Jiaming Xia]{Jiaming Xia}
\address[Jiaming Xia]{Department of Mathematics, University of Pennsylvania, Philadelphia, Pennsylvania, USA}
\email{xiajiam@sas.upenn.edu}
\begin{document}

\title{Hamilton--Jacobi equations from mean-field spin glasses}

\begin{abstract}
We give a meaning to the Hamilton--Jacobi equation arising from mean-field spin glass models in the viscosity sense, and establish the corresponding well-posedness. Originally defined on the set of monotone probability measures, these equations can be interpreted, via an isometry, to be defined on an infinite-dimensional closed convex cone with an empty interior in a Hilbert space. We prove the comparison principle, and the convergence of finite-dimensional approximations furnishing the existence of solutions. Under additional convexity conditions, we show that the solution can be represented by a version of the Hopf--Lax formula, or the Hopf formula on cones. Previously, two notions of solutions were considered, one defined directly as the Hopf--Lax formula, and another as limits of finite-dimensional approximations. They have been proven to describe the limit free energy in a wide class of mean-field spin glass models. This work shows that these two kinds of solutions are viscosity solutions.
\end{abstract}

\maketitle

\section{Introduction}

\rrv{The interplay between statistical mechanics and PDEs has been extensively studied (see, e.g., \cite{bogolyubov1984some, brankov1974model, choquard2004mean}). In this work, we focus on Hamilton--Jacobi equations arising in spin glass models. This connection was first explored by Guerra~\cite{guerra2001sum} in the replica-symmetric regime and was later extended to various settings~\cite{barra1,barra2,abarra,barramulti,barra2014quantum}. Our motivation stems from the recent works of Mourrat~\cite{mourrat2019parisi,mourrat2020extending,mourrat2020nonconvex,mourrat2023free}.}
By interpreting the inverse temperature as the temporal variable and enriching the model with a random magnetic field, whose parameter serves as the spatial variable, one can compare the resulting free energy with solutions to a certain Cauchy problem for a Hamilton--Jacobi equation.

Let us give an overview of these equations.
The spatial variable, denoted by $\varrho$, lives in $\cP^\uparrow$ the set of monotone probability measures on $\S^D_+$, the cone of $D\times D$ positive semi-definite matrices (see Section~\ref{s.monotone_measure} for definitions and properties of monotone probability measures) for some fixed $D\in\N$. Formally, the equation is of the following form:
\begin{align}\label{e.HJ_spin_glass}
    \partial_t f - \int \xi(\partial_\varrho f)\d \varrho =0,\quad\text{on }\R_+\times\cP^\uparrow,
\end{align}
where $\xi$ is a real-valued function on $\R^{D\times D}$ and
we set $\R_+ = [0,\infty)
$ throughout.

\rrv{We briefly explain the connection to a spin glass model. In the above, $D$ denotes the dimension of a single spin, $\xi$ represents the covariance function of the spin-glass Hamiltonian, $t$ is related to the inverse temperature $\beta$ via $\beta =\sqrt{2t}$, and $\varrho$ is the parameter for the Ruelle Probability Cascade~\cite{ruelle1987mathematical}, introduced as an external field. See~\eqref{e.F_N} for the expression of the finite-size free energy.  
When $\varrho=\delta_0$ is a Dirac measure at $0$, the corresponding spin glass model has no external field. For instance, the classical Sherrington--Kirkpatrick model~\cite{sherrington1975solvable} corresponds to the case where $D=1$ and $\xi(r)=r^2$ for $r\in\R$. With an appropriately chosen initial condition $f(0,\cdot)$, we obtain that $-f(t,\delta_0) + t$ (setting $\varrho=\delta_0$) represents the limit free energy of this model without an external field at inverse temperature $\beta = \sqrt{2t}$. See Example~\ref{e.SK} for further details.
}

Two notions of solutions have been considered. In \cite{mourrat2019parisi,mourrat2020extending} where $\xi$ is convex, the solution is defined by a version of the Hopf--Lax formula, which has been proven there to be equivalent to the celebrated Parisi formula first proposed in \cite{parisi1980sequence} and rigorously verified in \cite{gue03,talagrand2006parisi}. In \cite{mourrat2020nonconvex,mourrat2023free}, the solution, defined as limits of finite-dimensional approximations, was shown to be an upper bound for the limiting free energy in a wide class of models. 

The \textit{ad hoc} and \textit{extrinsic} nature of these two notions motivate us to seek an \textit{intrinsic} definition of solutions. We want to define solutions in the viscosity sense and establish the well-posedness of the equation, by which we mean the validity of a comparison principle and the existence of solutions. Moreover, we verify that the solution is the limit of finite-dimensional approximations, and, under certain convexity conditions, the solution admits a representation by a variational formula. In particular, we want to ensure that solutions understood in the aforementioned two notions are in fact viscosity solutions. Therefore, the framework of viscosity solutions is compatible with the existing theory.

The key difficulty is to find a natural definition of solutions in the viscosity sense so that all goals announced above are achievable. The surprising observation is that it is sufficient to simply require the solution to satisfy the equation in the viscosity sense everywhere, including the boundary \rv{of the natural domain} without prescribing any additional condition (e.g.\ Neumann or Dirichlet) on the boundary. Let us expand the discussion below.

We start with some basics.
To make sense of the differential $\partial_\varrho f$, we restrict $\cP^\uparrow$ to $\cP^\uparrow_2$, the set of monotone measures with finite second moments, and equip $\cP^\uparrow_2$ with the $2$-Wasserstein metric. Heuristically, the derivative $\partial_\varrho f(t,\varrho)$ describes the asymptotic behavior of $ f(t,\vartheta)-f(t,\varrho)$ as $\vartheta$ tends to $\varrho$ in the transport sense, namely, in the Wasserstein metric. 
Fortunately, $\cP^\uparrow_2$ can be isometrically \rv{(see~\eqref{e.isometry})} embedded onto a closed convex cone in an $L^2$-space. 
\rv{The isometry, defined in \eqref{e.diff_wass}, is simply mapping a probability measure to its quantile function (generalized from one dimension to $D$-dimension).}
This cone has an empty interior but generates the $L^2$-space. So, we cannot restrict to a subspace to ensure that the cone has a nonempty interior. Via this isometry, $\partial_\varrho f$ can be understood in the sense of the Fr\'echet derivative. 

Therefore, we can interpret \eqref{e.HJ_spin_glass} as the Hamilton--Jacobi equation
\begin{align}\label{e.HJ}
    \partial_t f - \H(\nabla f) =0,\quad  \text{on $\R_+\times \C$},
\end{align}
where $\C$ is a closed convex cone in a separable Hilbert space $\cH$, and $\H$ is a general nonlinearity.
\rv{Under the spin glass setting, $\C$ (see~\eqref{e.cone_M}) is the set of increasing functions from $[0,1)$ to $\S^D$ representing (generalized) quantile functions of probability measures and $\cH$ (see~\eqref{e.H_inft_d}) is the obvious $L^2$-space spanned by $\C$.}
\rv{As aforementioned, in the spin glass setting, $\C$ has an empty interior and thus the usual sense of boundary of $\C$ is not useful.}
Aside from the lack of local compactness in infinite dimensions,
\rv{to make sense of~\eqref{e.HJ},} 
one important issue is to figure out a suitable boundary condition \rv{on some suitable notion of boundary of $\C$}. The spin glass setting does not provide a direct hint, except for invalidating the Neumann boundary condition. Moreover, as aforementioned, the solution to \eqref{e.HJ_spin_glass} is expected to satisfy the Hopf--Lax formula under some convexity condition and to be the limit of finite-dimensional approximations. These can be hard to verify if the boundary condition is not easy to work with. The fact that the cone in the spin glass setting has an empty interior adds more difficulty. 

To bypass these obstacles, we exploit the assumption that $\H$ is increasing along the direction given by the dual cone of $\C$ \rv{(see~\eqref{e.C^*-increasing} for this notion of monotonicity)}, which holds in the spin glass setting. Under this assumption, as aforementioned, we do not need to impose any additional condition on \rv{some notion of the boundary of $\C$} and only need the equation to be satisfied in the viscosity sense (see Definition~\ref{d.vs}). This greatly simplifies our analysis and allows us to pass to the limit. Surprisingly, well-posedness holds under this simple definition because usually some boundary condition is needed.

In Section~\ref{s.hj_finite_d}, we study \eqref{e.HJ} on general finite-dimensional cones. Under the monotonicity assumption on $\H$, we recall from \cite{chen2022hamilton} the comparison principle (implying the uniqueness of solutions), the existence of solutions, and, under extra convexity conditions, the representation of the solution as either the Hopf--Lax formula, or the Hopf formula. We also prove a quantified version of the comparison principle, which is needed for passing to the limit.

In Section~\ref{s.HJ_inft_d}, we consider \eqref{e.HJ} on the infinite-dimensional cone relevant to the spin glass models. After establishing the comparison principle, we show that the limit of solutions to finite-dimensional approximations of \eqref{e.HJ} is a viscosity solution of \eqref{e.HJ}. Here, the construction of finite-dimensional approximations has the flavor of projective limits. We also verify that the Hopf--Lax formula and the Hopf formula are stable when passed to the limit.

In Section~\ref{s.spin-glass}, we start with a brief description of mean-field spin glass models. We present more definitions, basic results, and constructions, leading to an interpretation of viscosity solutions of \eqref{e.HJ_spin_glass} in Definition~\ref{d.vis_sol_spin_glass}. Then, we derive the basic properties of the nonlinear term in the equation, which allow us to combine results from other sections to prove the main result, Theorem~\ref{t.main}. Below is a formal restatement of our main result.
\begin{theorem}\label{t.informal}
Under certain assumptions on $\xi$ and on the initial condition $\psi$, which are admissible in mean-field spin glass models, there is a unique viscosity solution $f$ of the Cauchy problem of \eqref{e.HJ_spin_glass}. Moreover, 
\begin{enumerate}
    \item $f$ is the limit of viscosity solutions of finite-dimensional approximations of \eqref{e.HJ_spin_glass};
    \item $f$ is given by the Hopf--Lax formula \eqref{e.Hopf-Lax_spin_glass} if $\xi$ is convex on $\S^D_+$;
    \item $f$ is given by the Hopf formula \eqref{e.Hopf_spin_glass} if $\psi$ is convex.
\end{enumerate}
\end{theorem}
\noindent Accompanying this, a version of the comparison principle holds. In Section~\ref{s.def_vs_equiv}, we verify that solutions considered in \cite{mourrat2019parisi,mourrat2020extending,mourrat2020nonconvex,mourrat2023free} are viscosity solutions.

Lastly, in Section~\ref{s.fenchel-moreau}, we prove that on the cones underlying the finite-dimensional equations that approximate \eqref{e.HJ_spin_glass}, a version of the Fenchel--Moreau biconjugation identity holds, which is needed for the variational representation of a solution. We believe that this is also new.

\rv{Recall that our motivation is to find an intrinsic interpretation of~\eqref{e.HJ_spin_glass}. The approach taken here involves using an isometry to lift $\cP^\uparrow$ to a subset $\C$ within an $L^2$-space $\cH$, allowing us to define derivatives through the linear structure of $\cH$. However, this method disregards certain geometric subtleties of the space of probability measures—specifically, the existence of multiple geodesics between two measures beyond the one induced by the linear $L^2$-geometry.
As a result, while the framework developed here enables us to formulate the equation relevant to spin glasses in the form of~\eqref{e.HJ_spin_glass}, it is not intrinsically geometric in the strictest sense. We hope that future work will further address this limitation.
}

\rv{On the other hand, if we directly interpret the relevant spatial variables as increasing paths (quantile functions) rather than probability measures—namely, if we start with~\eqref{e.HJ} instead of~\eqref{e.HJ_spin_glass}—then our approach indeed provides an intrinsic interpretation. Here, the spatial variables should be of the same type as those over which the Parisi formula (analogous to the Hopf--Lax formula) optimizes. 
When $D=1$, many works in spin glass theory treat such variables as probability measures. However, for $D>1$, quantile functions appear to be more convenient to work with. This observation may provide motivation for favoring~\eqref{e.HJ} over~\eqref{e.HJ_spin_glass}.
}

We close this section with a discussion of related works and a description of the general setting for the equation~\eqref{e.HJ}.

\subsection{Related works}

First, we briefly review existing works on Hamilton--Jacobi equations in Hilbert spaces and Wasserstein spaces.

Equations on Banach spaces satisfying the Radon--Nikodym property (in particular, separable Hilbert spaces) were initially studied in \cite{crandall1985hamiltonI,crandall1986hamiltonII}, where the differential is understood in the Fr\'echet sense and the definition of viscosity solutions is a straightforward extension of definitions in finite dimensions. Comparison principles and existence results were established. Our interpretation of solutions is close in spirit to theirs. We use Stegall's variational principle (restated as Theorem~\ref{t.stegall}) as used in \cite{crandall1985hamiltonI} to compensate for the lack of local compactness, in order to prove the comparison principle (Proposition~\ref{p.comp}). Different from \cite{crandall1986hamiltonII}, we directly use finite-dimensional approximations to furnish the existence result. As demonstrated in \cite[Section~5]{crandall1986hamiltonII}, there are examples where finite-dimensional approximations converge to a solution of a different equation. Hence, the class of equations in this work provides an interesting example where finite-dimensional approximations work properly. Moreover, since the domain for \eqref{e.HJ} is a closed convex cone with an empty interior, simple modifications of methods for existence results in \cite{crandall1986hamiltonII} may not be viable. Works with modified definitions of viscosity solutions for equations in Hilbert spaces also include \cite{crandall1986hamiltonIII,crandall1990viscosityIV,crandall1991viscosityV,crandall1994hamiltonVI,tataru1992viscosity,feng2006large,feng2009comparison}.

Investigations of Hamilton--Jacobi equations on the Wasserstein space of probability measures include \cite{cardaliaguet2008deterministic,cardaliaguet2010notes,cardaliaguet2012differential,gangbo2008hamilton,ambrosio2014class,gangbo2014optimal,gangbo2019differentiability}. There are mainly three notions of differentiability considered in these works. Let $\cP_2(\R^d)$ be the $2$-Wasserstein space of probability measures on $\R^d$ for some $d\in\N$. The first way to make sense of differentiability is through defining the tangent space at each $\varrho\in\cP_2(\R^d)$ by analogy to differential manifolds.
The tangent space at $\varrho$ is the closure of $\{\nabla\phi:\phi\in C^\infty_\mathrm{c}(\R^d)\}$ in $L^2(\R^d,\varrho)$. We refer to \cite{ambrosio2005gradient} \rv{for this notion of differentiability; to~\cite{fenkur} for the first use of mass transport to study such equations; and to~\cite{feng2012hamilton} for the well-posedness of such equations along this consideration}. The second one, more extrinsic, starts by extending any function $g:\cP_2(\R^d)\to\R$ through defining $G: L^2(\Omega,\P) \to\R$ by $G[X]=g(\Law(X))$ for every $\R^d$-valued random variable $X\in L^2(\Omega,\P)$ on some nice probability space $(\Omega,\P)$. Then, one can make sense of the differentiability of $g$ via the Fr\'echet differentiability of $G$. 
One issue is that two different random variables may have the same law, which leads to the situation where $\varrho,\vartheta$ can be ``lifted'' to $X,Y$, respectively, while $X$ and $Y$ are not optimally coupled. Namely, the $L^2$ norm of $X-Y$ is not equal to the metric distance between $\varrho$ and~$\vartheta$.
Another issue is the lack of a canonical choice of $\Omega$.
For details, we refer to \cite{cardaliaguet2010notes,gangbo2019differentiability}. The third notion is based on viewing $\cP_2(\R^d)$ as a geodesic metric space. Denoting by $\metric_2$ the $2$-Wasserstein metric, for any $g:\cP_2(\R^d)\to\R$, one can define the slope of $g$ at $\varrho$ by $|\nabla g|=\limsup_{\vartheta\to\varrho}\frac{|g(\vartheta)-g(\varrho)|}{\metric_2(\vartheta,\varrho)}$. Then, one can study equations involving slopes. This notion was considered in \cite{ambrosio2014class}.

The notion of differentiability adopted in this work is close in spirit to the second one discussed above. But for us, there is an isometry (see \eqref{e.isometry}) between $\cP^\uparrow_2$ and a closed convex cone in an $L^2$-space. As a result of the monotonicity (see \eqref{e.monotone_measure}) of measures in $\cP^\uparrow_2$, the isometry is given by the right-continuous inverse of some analogue of the probability distribution function, which was observed in \cite[Section~2]{mourrat2023free}. Hence, in our case, we can identify
$[0,1)$ equipped with the Borel sigma-algebra and the Lebesgue measure as the canonical probability space $\Omega$, appearing in the discussion of the second notion. It is natural and convenient to use the Hilbert space structure of $L^2([0,1))$ to define differentiability.
We note that in the absence of convexity, the Hopf–Lax and Hopf formulas are not applicable. In this case, \cite{chen2024envelope} provides an envelope representation of the solution to~\eqref{e.HJ}, following the approach of~\cite{evans2014envelopes}.

To the best of our knowledge, there are no prior works on the well-posedness of Hamilton--Jacobi equations \rv{on a nontrivial subset of an infinite-dimensional linear space}, or on \rv{a nontrivial subset in} a Wasserstein space.

Secondly, we discuss works that apply PDE methods to statistical mechanics models.

\rrv{One of the simplest models to which the Hamilton--Jacobi approach is applicable is the Curie--Weiss model~\cite{barra2008mean,genovese2009mechanical}. The relevant equation has a form similar to~\eqref{e.HJ} but is finite-dimensional (one-dimensional for Ising spins). Moreover, if the interaction is quadratic, then at finite size $N$, the finite-size free energy satisfies the viscous version of the limit equation, with viscosity coefficient $\frac{1}{N}$ (see~\cite[(3.7)]{HJbook}). We refer to~\cite[Section~3.1]{HJbook} for a complete treatment, including more general types of interactions. However, in the spin glass setting, at any finite size, the free energy does not appear to satisfy a viscous version of the limit equation. More precisely, it is unclear how to identify the viscous term in the equation at finite size (see~\cite[(3.27)]{mourrat2023free}), which reflects a key difficulty in the spin glass setting.
}

As mentioned in the beginning, considerations of using Hamilton--Jacobi equations to study the free energy of mean-field disordered models first appeared in physics literature \cite{guerra2001sum,barra1,barra2,abarra,barramulti,barra2014quantum}.
\rrv{Also, see~\cite{agliari2022nonlinear,fachechi2021pde} for recent continuations.}
The approach was also explored in \cite{mourrat2018hamilton}, and used subsequently in \cite{mourrat2019hamilton,chen2020hamilton,chen2020hamiltonTensor,chen2021statistical,chen2021limiting} to treat statistical inference models. There, the equations also take the form \eqref{e.HJ} but are defined on finite-dimensional cones. Similar to the equation in spin glass models, the nonlinearity is monotone along the direction of the dual cone (which is the same cone as the cones in these models are self-dual). In these works, some additional Neumann-type condition is imposed on the boundary \rv{of the said finite-dimensional cone}. We remark that these conditions can be dropped and the results in \cite{mourrat2018hamilton,chen2020hamiltonTensor,chen2021statistical}, where solutions were defined in the viscosity sense, still hold with our simplified definition of viscosity solutions (Definition~\ref{d.vs}). Facts about viscosity solutions proven and used there can be replaced by those in Section~\ref{s.hj_finite_d}.
Related works on the well-posedness of these finite-dimensional equations also include~\cite{issa2024weaksol,issa2024weak}.
The Hamilton–Jacobi equation technique is also useful in analyzing spin glass models with additional conventional order parameters, such as self-overlap or mean magnetization~\cite{chen2023self,chen2024self,chen2024conventional}. Moreover, recent studies have demonstrated that the Parisi formula can be un-inverted, transforming the supremum into an infimum~\cite{mourrat2023inverting,issa2024hopf}. Notably, the latter work reveals that the neun-inverted formula closely resembles the Hopf formula.

\rrv{The solution of~\eqref{e.HJ_spin_glass} is believed to describe the limit free energy for all reasonable $\xi$ arising from the covariance of the spin-glass Hamiltonian (see~\eqref{e.Gaussian_covariance}) and typical choices of spins. 
Recall that $\xi$ is a real-valued function on $\R^{D\times D}$. If $\xi$ is convex on $\S^D_+$ and the spins are i.i.d.\ and bounded, then the limit free energy has been proven to satisfy~\eqref{e.HJ_spin_glass} in~\cite{HJ_critical_pts}, covering a broad class of vector spin models from~\cite{pan05,panchenko2014parisi,pan,pan.potts,pan.vec,mourrat2020extending}. 
Due to the absence of Talagrand's positivity principle (see~\cite[Section~3.3]{pan}) for $D>1$, the works~\cite{pan.potts,pan.vec} assume $\xi$ to be convex over $\R^{D\times D}$. Thus, the result in~\cite{HJ_critical_pts} is stronger. Similar results are established for multi-species models with a convex interaction and bounded spins in~\cite{chen2024on}, extending~\cite{pan.multi,barcon}. In these cases, the Hopf--Lax representations hold and are equivalent to the Parisi formulas.
The spherical spin glass (single-species)~\cite{tal.sph,chen2013aizenman} is already treated in~\cite{mourrat2019parisi}. We expect that multi-species spherical models with convex interaction, such as those in~\cite{bates2022crisanti, bates2022free, ko2020free}, can also be analyzed using the same method, albeit with additional technical challenges.
In the above works within the Hamilton--Jacobi equation framework, the boundedness of spins simplifies the analysis in many ways. As long as $\xi$ is convex, extending the results to unbounded spins with sufficiently fast decay appears feasible but requires additional effort.
}

\rrv{The most interesting case arises when the interaction $\xi$ is not convex. A key example is the bipartite spin glass, which is closely related to the Hopfield model and the restricted Boltzmann machine. Non-convex models serve as a primary motivation for the Hamilton--Jacobi equation approach, as existing methods based on the Parisi formula fail to identify the limit free energy. This problem remains unresolved, and we mention partial results obtained through this approach.
Consider a vector spin model with non-convex $\xi$ and bounded spins. In this case, Guerra's RSB bound~\cite{gue03} does not apply, and instead,~\cite{mourrat2020nonconvex,mourrat2023free} provides an upper bound on the free energy. Furthermore,~\cite{HJ_critical_pts} shows that any subsequential limit of the free energy must satisfy the equation on a dense set, which is a weaker condition than being a viscosity solution. Consequently, the uniqueness of viscosity solutions cannot be used to establish the limit.
Additionally,~\cite{HJ_critical_pts} demonstrates that if the limit free energy exists, then at every point $(t,\varrho)$, its value must be determined by a characteristic line associated with the equation. This property is stronger than that of a viscosity solution but still not able to identify the limit. In principle, it is possible that the appropriate notion of solutions in the spin glass setting differs from viscosity solutions. Hence, there is still much to be understood in the non-convex case.}

Lastly, we briefly mention relevant results on finite-dimensional equations.

In \cite{mourrat2018hamilton,chen2020hamiltonTensor,chen2021statistical}, the viscosity solution always admit an expression as the Hopf formula. To prove this, a version of the Fenchel--Moreau biconjugation identity on cones is needed, which has been proven for a large class of cones in \cite{chen2020fenchel}. However, the cones pertinent to spin glass models do not fall in that class. As aforementioned, we prove the identity on these cones in Section~\ref{s.fenchel-moreau}, following similar arguments as in \cite{chen2020fenchel}.

Using the monotonicity of the nonlinearity, \cite{crandall1985viscosity, souganidis1986remark} showed that the viscosity solution to a Hamilton--Jacobi equation on an open set $\Omega$ in finite dimensions can be extended to a viscosity solution on $\Omega\cup\{z\}$ for any regular point $z\in\partial \Omega$. Results in \cite{chen2022hamilton} to be recalled in Section~\ref{s.hj_finite_d} extend these to the cones.

\subsection{General setting and definitions}\label{s.general_setting}

\rv{In this subsection, $\C$ and $\cH$ are considered in a general context and are not specific to the spin glass setting.}
Let $\cH$ be a separable Hilbert space with inner product $\la\, \cdot\,,\,\cdot\,\ra_\cH$ and associated norm $|\cdot|_\cH$. Let $\C\subset\cH$ be a closed convex cone. In addition, we assume that $\C$ generates $\cH$, namely,
\begin{align}\label{e.cl(C-C)=H}
    \cl(\C-\C)=\cH,
\end{align}
where $\mathsf{cl}$ is the closure operator.
The dual cone of $\C$ is defined to be
\begin{align}\label{e.C^*}
    \C^*=\{x\in\cH: \la x,y\ra_\cH\geq 0,\ \forall y\in\C\}.
\end{align}
It is clear that $\C^*$ is a closed and convex cone.

\begin{definition}[Differentiability and smoothness]\label{d.differentiability}
Let $\cD$ be a subset of $\cH$.
\begin{enumerate}
\item A function $\phi:(0,\infty)\times\cD\to\R$ is said to be \textit{differentiable} at $(t,x)\in(0,\infty)\times\cD$, if there is a unique element in $\R\times\cH$, denoted by $(\partial_t \phi(t,x),\nabla \phi(t,x))$ and called the \textit{differential} of $\phi$ at $(t,x)$, such that
\begin{align*}
    \phi(s,y)-\phi(t,x) = \partial_t\phi(t,x)(s-t) + \la \nabla \phi(t,x), y - x\ra_\cH + o\left(|s-t|+|y - x|_\cH\right),
\end{align*}
as $(s,y) \in (0,\infty)\times\cD$ tends to $(t,x)$ in $\R\times \cH$.
\item \label{i.phi_smooth} A function $\phi:(0,\infty)\times\cD\to\R$ is said to be \textit{smooth} if
\begin{enumerate}
    \item $\phi$ is differentiable everywhere with differentials satisfying that, for every $(t,x)\in(0,\infty)\times\cD$,
\begin{align*}
    \phi(s,y)-\phi(t,x) = \partial_t\phi(t,x)(s-t) + \la \nabla \phi(t,x), y - x\ra_\cH + O\left(|s-t|^2+|y - x|^2_\cH\right),
\end{align*}
as $(s,y) \in (0,\infty)\times\cD$ tends to $(t,x)$ in $\R\times \cH$;

    \item the function $(t,x)\mapsto(\partial_t\phi(t,x),\nabla \phi(t,x))$ is continuous from $(0,\infty)\times \cD$ to $\R\times\cH$.
\end{enumerate}

\item \label{i.def_differentiable_C} A function $g:\cD\to\R$ is said to be \textit{differentiable} at $x\in\cD$, if there is a unique element in $\cH$, denoted by $\nabla g(x)$ and called the \textit{differential} of $g$ at $x$, such that
\begin{align*}
    g(y)-g(x) =  \la \nabla g(x), y - x\ra_\cH + o\left(|y - x|_\cH\right),
\end{align*}
as $y \in \cD$ tends to $x$ in $\cH$.

\end{enumerate}

\end{definition}

\begin{remark}\rm
We are mostly interested in the case $\cD=\C$.
Note that the differential is defined at every point of the closed cone $\C$, which is needed to make sense of differentials at boundary points. Also, in infinite dimensions, $\C$ can have an empty interior. Let us show that the differential is unique whenever it exists. Hence, the above is well-defined.

To see this, it suffices to show that, for any fixed $(t,x)\in(0,\infty)\times\C$, if $(r,h)\in\R\times\cH$ satisfies $r(s-t)+\la h, y-x\ra_\cH = o(|s-t|+|y - x|_\cH)$ for all $(s,y)\in(0,\infty)\times\cM$, then we must have $r=0$ and $h =0 $. It is easy to see that $r=0$. Replacing $y$ by $x+\eps z$ for $\eps>0$ and any fixed $z\in\C$, and sending $\eps\to 0$, we can deduce that $\la h,z\ra_\cH = 0$ for all $z\in\C$, which along with~\eqref{e.cl(C-C)=H} implies that~$h =0$.
\end{remark}

For a closed cone $\cK\subset\cH$, a function $g: \cD\to (-\infty,\infty]$ defined on a subset $\cD\subset \cH$ is said to be \textit{$\cK$-increasing} (over $\cD$) if $g$ satisfies that
\begin{align}\label{e.C^*-increasing}
    g(x)\geq g(x'),\qquad \text{for all $x, x'\in \cD$ satisfying $x-x'\in\cK$.}
\end{align}
Let $\H:\cH\to\R$ be a continuous function.
Since we work with equations defined on different sets, in different ambient Hilbert spaces, and with different nonlinearities, for convenience, we denote by $\HJ(\cH,\cD,\H)$ the equation $\partial_t f-\H(\nabla f)=0$ on $\R_+\times \cD$ for some $\cD\subset \cH$. The corresponding Cauchy problem with initial condition $\psi:\cD\to\R$ is denoted by $\HJ(\cH,\cD,\H;\psi)$.

\begin{definition}[Viscosity solutions]\label{d.vs} Let $\cD$ be a subset of $\cH$.
\leavevmode
\begin{enumerate}
    \item A continuous function $f:\R_+\times \cD\to \R$ is a \textit{viscosity subsolution} of~$\HJ(\cH,\cD,\H)$ if for every $(t,x) \in (0,\infty)\times \cD$ and every smooth $\phi:(0,\infty)\times \cD\to\R$ such that $f-\phi$ has a local maximum at $(t,x)$, we have
\begin{align*}
    \left(\partial_t \phi - \H(\nabla\phi)\right)(t,x)\leq 0.
\end{align*}

\item A continuous function $f:\R_+\times \cD\to \R$ is a \textit{viscosity supersolution} of~$\HJ(\cH,\cD,\H)$ if for every $(t,x) \in (0,\infty)\times \cD$ and every smooth $\phi:(0,\infty)\times \cD\to\R$ such that $f-\phi$ has a local minimum at $(t,x)$, we have
\begin{align*}
    \left(\partial_t \phi - \H(\nabla\phi)\right)(t,x)\geq 0. 
\end{align*}

\item A continuous function $f:\R_+\times \cD\to \R$ is a \textit{viscosity solution} of~$\HJ(\cH,\cD,\H)$ if $f$ is both a viscosity subsolution and supersolution.

\end{enumerate}

\end{definition}

Here, a local extremum at $(t,x)$ is understood to be an extremum over a metric ball of some positive radius centered at~$(t,x)$ intersected with $(0,\infty)\times \cD$.

For $\psi:\cD\to\R$, we call $f:\R_+\times\cD\to\R$ a viscosity solution of $\HJ(\cH,\cD,\H;\psi)$ if $f$ is a viscosity solution of $\HJ(\cH,\cD,\H)$ and satisfies $f(0,\cdot) = \psi$.

Throughout, Lipschitzness of any real-valued function on a subset of $\cH$ or $\R\times\cH$ is defined with respect to $|\cdot|_\cH$ or $|\cdot|_{\R\times\cH}$, respectively. A \textit{Lipschitz viscosity solution} is a viscosity solution that is Lipschitz.

For every $a,b\in\R$, we write $a\vee b= \max\{a,b\}$ and $a_+ = a\vee0$.

\subsection*{Acknowledgement}
We warmly thank Jean-Christophe Mourrat for many stimulating discussions and helpful comments.

\section{Equations on finite-dimensional cones}\label{s.hj_finite_d}

In this section, we assume that $\cH$ is finite-dimensional. We consider the setting given in Section~\ref{s.general_setting} and study the equation $\HJ(\cH,\C,\H)$.
Notice that, in finite dimensions, the assumption~\eqref{e.cl(C-C)=H} implies that $\C$ has a nonempty interior.
We denote by $\mathring \C$ the interior of $\C$ in $\cH$ and we also consider the equation $\HJ(\cH,\mathring\C,\H)$.

\subsection{Basic results}

Hamilton--Jacobi equations with monotone nonlinearities on convex cones have been studied in~\cite{chen2022hamilton}.
We consider viscosity solutions in the class of functions $f:\R_+\times \C\to\R$ satisfying
\begin{align}\label{e.class_sol}
    \sup_{t\in\R_+}\|f(t,\cdot)\|_\mathrm{Lip}<\infty;\quad \sup_{t>0,\,x\in\C}\frac{|f(t,x)-f(0,x)|}{t}<\infty;\quad  \text{$f(t,\cdot)$ is $\C^*$-increasing, $\forall t\in\R_+$}.
\end{align}
To find variational representations of the solution, we need an additional condition on $\C$.
For $\cE\supset\C$ and $g:\cE\to(-\infty,\infty]$, we define the \textit{monotone conjugate} (over $\C$) of $g$ by
\begin{align}\label{e.def_u*}
    g^*(y) = \sup_{x\in \C}\{\la x,y\ra_\cH-g(x)\},\quad\forall y \in \cH.
\end{align}
Let $g^{**} = (g^*)^*$ be the \textit{monotone biconjugate} of $g$ with expression
\begin{align*}g^{**}(x) = \sup_{y\in \C}\{\la y,x\ra_{\cH}-g^*(y)\},\quad\forall x\in \cH.
\end{align*}
\begin{definition}\label{d.fenchel_moreau_prop}
A nonempty closed convex cone $\C$ is said to have the \textit{Fenchel--Moreau property} if the following holds: for every $g:\C\to(-\infty,\infty]$ not identically equal to $\infty$, we have that $g^{**}=g$ on $\C$ if and only if $g$ is convex, lower semicontinuous, and $\C^*$-increasing.
\end{definition}

In Section~\ref{s.fenchel-moreau}, we show that the cones relevant to the spin glass models have the Fenchel--Moreau property. 

For $r>0$, we denote by $B(r)$ the centered closed ball in $\cH$ with radius $r$. Below is \cite[Theorem~1.2]{chen2022hamilton} slightly simplified.

\begin{theorem}[\cite{chen2022hamilton}]\label{t.hj_cone}
Let $\H:\cH\to\R$ satisfy that $\H\lfloor_\C$ is locally Lipschitz and $\C^*$-increasing. Then, the following holds:
\begin{enumerate}
  \setlength\itemsep{1em}
    \item (Comparison principle)
    
        \medskip
        
    \noindent If 
    $u,v:\R_+\times\C\to\R$ are respectively a subsolution and a supersolution of $\HJ(\cH,\mathring\C,\H)$ in the class~\eqref{e.class_sol}, then $\sup_{\R_+\times\C}(u-v)  = \sup_{\{0\}\times \C}(u-v)$.
    
    \item (Existence of solutions)
    
        \medskip
        
    \noindent For every Lipschitz and $\C^*$-increasing $\psi:\C\to\R$, there is a viscosity solution $f:\R_+\times \C\to\R$ of $\HJ(\cH,\mathring\C,\H;\psi)$ unique in the class~\eqref{e.class_sol}. In addition, $f$ satisfies the following:
    \bigskip
    \begin{enumerate}
    \setlength\itemsep{1em}
        \item  \label{i.main_lip} 
        (Lipschitzness)
        
        \medskip
        
        The solution $f$ is Lipschitz and satisfies
        \begin{gather*}
\sup_{t\in\R_+}\|f(t,\cdot)\|_\mathrm{Lip}=\|\psi\|_\mathrm{Lip},
            \\
            \sup_{x\in\C}\|f(\cdot,x)\|_\mathrm{Lip}\leq \sup_{\C\cap B(\|\psi\|_\mathrm{Lip})}|\H|.
        \end{gather*}
        
        \item
(Monotonicity in time)
        
        \medskip
        
        If $\H\lfloor_\C \geq 0$, then $f(t,x)\leq f(t',x)$ for all $t'\geq t\geq 0$ and $x\in\C$.
  
        \item \label{i.main_equiv} 
        (Solving modified equations)
        
        \medskip
        
        For every locally Lipschitz and $\C^*$-increasing $\sF:\cH\to\R$ satisfying $\sF\lfloor_{\C\cap B(\|\psi\|_\mathrm{Lip})} = \H\lfloor_{\C\cap B(\|\psi\|_\mathrm{Lip})}$, $f$ is the solution of $\HJ(\cH,\C,\sF;\psi)$ unique in the class~\eqref{e.class_sol}.
        
        \item
        \label{i.main_var_rep}
        (Variational representations)
        
        \medskip
        
        Under an additional assumption that $\C$ has the {\rm Fenchel--Moreau property}, if $\H\lfloor_\C$ is convex and bounded below, then
        \begin{align}\label{e.hopf_lax}
    f(t,x) = \sup_{y\in \C}\inf_{z\in\C}\left\{\psi(x+y)-\la y,z\ra_\cH+t\H(z)\right\},\quad\forall (t,x)\in\R_+\times\C,
\end{align}
        or if $\psi$ is convex, then
        \begin{align}\label{e.hopf}
    f(t,x) = \sup_{z\in\C}\inf_{y\in\C} \left\{\psi(y)+\la x-y,z\ra_\cH+t\H(z)\right\},\quad\forall (t,x)\in\R_+\times\C.
\end{align}
    \end{enumerate}
\end{enumerate}
\end{theorem}

In the statement, by that $f:\R_+\times \C\to\R$ is a solution of $\HJ(\cH,\mathring\C,\H)$, we mean that $f\lfloor_{\R_+\times \mathring\C}$ is a solution. Also, by that $f$ is a solution of $\HJ(\cH,\mathring\C,\H;\psi)$, we mean that, additionally, $f(0,\cdot)\lfloor_{\mathring\C} = \psi\lfloor_{\mathring\C}$ which actually implies $f(0,\cdot)=\psi$ on $\C$ since both are Lipschitz.

In~\eqref{i.main_var_rep}, \eqref{e.hopf_lax} is the Hopf--Lax formula on convex cones. For the standard version, we refer to \cite{evans2010partial,crandall1992user}. The second one~\eqref{e.hopf} is the Hopf formula on convex cones. Hopf originally proposed the standard version in \cite{hopf1965generalized}, which was later confirmed rigorously in \cite{bardi1984hopf,lions1986hopf}.

As a consequence of~\eqref{i.main_equiv}, the unique solution $f$ of $\HJ(\cH,\mathring\C,\sF)$ is also the unique solution of $\HJ(\cH,\C,\sF)$. Notice that the uniqueness of the solution of $\HJ(\cH,\C,\sF)$ usually requires imposing a boundary condition on $\partial \C$. Here, due to the monotonicity of $\sF$ on $\cH$, such a condition is not needed. More precisely, $\sF$ is increasing along $\C^*$ and outer normal vectors of $\C$ lies in $-\C^*$. This coincidence allows us to use ideas in \cite{crandall1985viscosity,souganidis1986remark} to deduce the irrelevance of the boundary condition. 

Let us briefly explain this.
We argue that if $f$ is a subsolution of $\HJ(\cH,\mathring\C,\sF)$, then $f$ should be a subsolution of $\HJ(\cH,\C,\sF)$. 
Let $d= \mathrm{dist}(\cdot,\partial\C)$ and consider $f_\eps$ given by $f_\eps(t,x) =f(t,x) - \frac{\eps}{d(x)}$. Heuristically, since the ``gradient'' of $d$ lies in $\C^*$ ($d$ may not be differentiable), we have $\nabla f_\eps -\nabla f\in \C^*$. Since $\sF$ is $\C^*$-increasing, we have $\partial_t f_\eps - \sF(\nabla f_\eps)\leq 0$ and thus $f_\eps$ should also be a subsolution of $\HJ(\cH,\mathring\C,\sF)$. 
Now, if $f-\phi$ achieves a local maximum at some $(t,x)$ for $x\in\C$, we can choose $\eps$ sufficiently small so that $f_\eps-\phi$ has a local maximum at some $(t_\eps, x_\eps)$ near $(t,x)$. The advantage now is that $x_\eps$ must be an interior point. Since $f_\eps$ is a subsolution, we have $\partial_t\phi-\sF(\nabla \phi)\leq 0$ at $(t_\eps,x_\eps)$. Sending $\eps\to0$, the relation holds at $(t,x)$. Hence, we can conclude that $f$ is a subsolution of $\HJ(\cH,\C,\sF)$. 
For the supersolution, we use $f^\eps$ given by $f^\eps(t,x) =f(t,x) + \frac{\eps}{d(x)}$.
We restate \cite[Proposition~2.1]{chen2022hamilton} below and refer to its proof for the rigorous argument.

\begin{proposition}[\cite{chen2022hamilton}]\label{p.drop_bdy}
Let $\sF:\cH\to\R$ be $\C^*$-increasing and continuous. If $f:\R_+\times\C\to\R$ is a viscosity solution of $\HJ(\cH,\mathring\C,\sF)$, then $f$ is a viscosity solution of $\HJ(\cH,\C,\sF)$.
\end{proposition}

\subsection{Lipschitzness in \texorpdfstring{$\ell^p$}{lp}-norms}

We consider the setting where $\cH$ is a product space with $\ell^p$ norms:
\begin{enumerate}[start=1,label={\rm{(P)}}]
    \item \label{i.H_prod} Let $\cH = \times_{i=1}^k\cH_i$ where each $\cH_i$ is a Hilbert space with inner product $\la\cdot,\cdot\ra_{\cH_i}$ and the induced norm $|\cdot|_{\cH_i}$. 
    Let $a_1,a_2,\ldots, a_k>0$ satisfy $\sum_{i=1}^ka_i=1$. We set $\la x,x'\ra_\cH= \sum_{i=1}^k a_i \la x_i,x'_i\ra_{\cH_i}$ for $x,x'\in\cH$.
    For every $x\in \cH$, we define
    \begin{align*}
        \|x\|_p = \left(\sum_{i=1}^ka_i|x_i|^p_{\cH_i}\right)^\frac{1}{p},
    \end{align*}
    for $p\in[1,\infty)$
    and $\|x\|_\infty = \sup_{i=1,2,\ldots,k}|x_i|_{\cH_i}$. As usual, we set $p^* = \frac{p}{p-1}$.
\end{enumerate}
Note that in this setting, the induced norm $|\cdot|_\cH$ by the inner product on $\cH$ is equal to $\|\cdot\|_2$.
For any $g:\C\to\R$, define
\begin{align*}
    \|g\|_{\mathrm{Lip}\|\cdot\|_p} = \sup_{\substack{y,y'\in\C\\ y\neq y'}}\frac{|g(y)-g(y')|}{\|y-y'\|_p}.
\end{align*}

The following is \cite[Corollary~5.4]{chen2022hamilton}, which is used later in Section~\ref{s.spin-glass}.
\begin{proposition}[\cite{chen2022hamilton}]\label{p.lip_l^p}
Under~\ref{i.H_prod}, if $f$ is a viscosity solution of $\HJ(\cH,\C,\sF;\psi)$ given by Theorem~\ref{t.hj_cone}~\eqref{i.main_equiv}, then, for all $p\in[1,\infty]$,
\begin{align*}
    \sup_{t\in\R_+}\|f(t,\cdot)\|_{\mathrm{Lip}\|\cdot\|_p} = \|\psi \|_{\mathrm{Lip}\|\cdot\|_p}\quad\text{and} \quad \sup_{x\in\C}\|f(\cdot,x)\|_\mathrm{Lip}\leq\sup_{v\in\C,\,\|v\|_{p^*}\leq \|\psi\|_{\mathrm{Lip}\|\cdot\|_p}}|F(v)|.
\end{align*}
\end{proposition}

\subsection{Quantified comparison principle}

Later, to show the convergence of solutions of finite dimensional equations to the solution in infinite dimensions, we need a more quantified version of the comparison principle than that in Theorem~\ref{t.hj_cone}.

\begin{proposition}[Quantified comparison principle]\label{p.comp_fin_d_strong}
Suppose that $\H:\cH\to\R$ is locally Lipschitz.
Let $u$ be a viscosity subsolution of $\HJ(\cH,\C,\H)$ and $v$ be a viscosity supersolution of $\HJ(\cH,\C',\H)$, with either $\C\subset\C'$ or $\C'\subset\C$. Suppose that
\begin{align*}
    L=\sup_{t\in\R_+}\|u(t,\cdot)\|_\mathrm{Lip}\vee\|v(t,\cdot)\|_\mathrm{Lip}
\end{align*}
is finite. Then, for every $R>0$ and every $M>2L$, the function
\begin{align}\label{e.u-v-M}
    \R_+\times (\C\cap\C')\ni(t,x)\longmapsto u(t,x) -v(t,x)-M(|x|_\cH+Vt-R)_+
\end{align}
achieves its global supremum on $\{0\}\times(\C\cap\C')$, where 
\begin{align*}
    V=\sup\left\{\frac{|\H(y)-\H(y')|}{|y-y'|_\cH}:|y|_\cH,\,|y'|_\cH\leq 2L+3M\right\}.
\end{align*}
\end{proposition}

The proof below is a modification of the proof of \cite[Proposition~3.2]{mourrat2020nonconvex}.

\begin{proof}[Proof of Proposition~\ref{p.comp_fin_d_strong}]
For $\delta\in(0,1)$ to be chosen, let $\theta:\R\to\R_+$ be a increasing smooth function satisfying
\begin{align*}
    |\theta'|\leq 1, \qquad\text{and}\qquad(r-\delta)_+ \leq \theta(r)\leq r_+,\quad\forall r\in\R,
\end{align*}
where $\theta'$ is the derivative of $\theta$.
We define
\begin{align*}
    \Phi(t,x) = M\theta\left(\left(\delta+|x|^2_\cH\right)^\frac{1}{2}+Vt-R\right),\quad\forall (t,x)\in \R_+\times \C.
\end{align*}
It is immediate that
\begin{gather}
    \sup_{(t,x)\in\R_+\times\C}|\nabla\Phi(t,x)|_\cH\leq M,\label{e.|nabla_Phi|}
    \\
    \partial_t\Phi\geq V|\nabla \Phi|_\cH,\label{e.d_tPhi>}
    \\
    \Phi(t,x)\geq M(|x|_\cH+Vt-R-1)_+, \quad\forall (t,x)\in\R_+\times\C. \label{e.Phi(t,x)>M|x|}
\end{gather}
We argue by contradiction and assume that the function in \eqref{e.u-v-M} does not achieve its supremum on $\{0\}\times(\C\cap\C')$. Then, we can fix $\delta\in(0,1)$ sufficiently small and $T>0$ sufficiently large so that
\begin{align*}
    \sup_{[0,T)\times(\C\cap\C')}(u-v-\Phi)>\sup_{\{0\}\times(\C\cap\C')}(u-v-\Phi).
\end{align*}
For $\eps>0$ to be determined, we define
\begin{align*}
    \chi(t,x)=\Phi(t,x) + \eps t+\frac{\eps}{T-t},\quad \forall (t,x) \in [0,T)\times\C.
\end{align*}
In view of the previous display, we can choose $\eps>0$ small and further enlarge $T$ so that
\begin{align}\label{e.u-v-chi}
    \sup_{[0,T)\times(\C\cap\C')}(u-v-\chi)>\sup_{\{0\}\times(\C\cap\C')}(u-v-\chi).
\end{align}
For each $\alpha>1$, we introduce
\begin{align*}
    \Psi_\alpha(t,x,t',x')= u(t,x)-v(t',x')-\frac{\alpha}{2}(|t-t'|^2+|x-x'|^2_\cH)-\chi(t,x),
    \\
    \quad\forall (t,x,t',x')\in[0,T)\times\C\times[0,T]\times\C'. 
\end{align*}
By the definition of $L$ and \eqref{e.Phi(t,x)>M|x|}, setting $C_1=\sup_{t\in[0,T]}(|u(t,0)|\vee|v(t,0)|)$, we can see that
\begin{align*}
    \Psi_\alpha(t,x,t',x')\leq C_1 +L(2|x|_\cH+|x-x'|_\cH)-\frac{1}{2}|x-x'|^2-M(|x|_\cH-R-1)_+.
\end{align*}
Hence, due to the requirement $M>2L$, $\Psi_\alpha$ is bounded from above uniformly in $\alpha>1$ and decays as $|x|_\cH,|x'|_\cH\to\infty$. Since $\cH$ is finite-dimensional, we can see that $\Psi_\alpha$ achieves its supremum at some $(t_\alpha,x_\alpha,t'_\alpha,x'_\alpha)$. The above display also implies that there is $C$ such that
\begin{align*}
    |x_\alpha|_\cH,\ |x'_\alpha|_\cH\leq C,\quad\forall \alpha>1.
\end{align*}
Setting $C_0=\Psi_\alpha(0,0,0,0)$ which is independent of $\alpha$, we have
\begin{align*}
    C_0\leq \Psi_\alpha(t_\alpha,x_\alpha,t'_\alpha,x_\alpha)\leq C_1+2LC-\frac{\alpha}{2}(|t_\alpha-t'_\alpha|^2+|x_\alpha-x'_\alpha|^2_\cH).
\end{align*}
From this, we can see that $\alpha(|t_\alpha-t'_\alpha|^2+|x_\alpha-x'_\alpha|^2_\cH)$ is bounded as $\alpha\to\infty$.
Hence, passing to a subsequence if necessary, we may assume $t_\alpha,t'_\alpha\to t_0$ and $x_\alpha,x'_\alpha\to x_0$ for some $(t_0,x_0)\in [0,T]\times( \C\cap\C')$.

Then, we show $t_0\in(0,T)$.
Since
\begin{align*}
    C_0\leq \Psi_\alpha(t_\alpha,x_\alpha,t'_\alpha,x_\alpha)\leq C_1+2LC-\frac{\eps}{T-t_\alpha},
\end{align*}
we must have that $t_\alpha$ is bounded away from $T$ uniformly in $\alpha$, which implies $t_0<T$. Since
\begin{align*}
    u(t_\alpha,x_\alpha) - v(t'_\alpha,x'_\alpha)-\chi(t_\alpha,x_\alpha)\geq \Psi_\alpha(t_\alpha,x_\alpha,t'_\alpha,x'_\alpha) \\
    \geq \sup_{[0,T)\times(\C\cap\C')}(u-v-\chi)\geq (u-v-\chi)(t_0,x_0),
\end{align*}
sending $\alpha\to\infty$, we deduce that
\begin{align*}
    (u-v-\chi)(t_0,x_0) = \sup_{[0,T)\times(\C\cap\C')}(u-v-\chi).
\end{align*}
This along with \eqref{e.u-v-chi} implies that $t_0>0$. In conclusion, we have $t_0\in(0,T)$, and thus $t_\alpha,t'_\alpha\in(0,T)$ for sufficiently large $\alpha$. Henceforth, we fix any such $\alpha$.

Before proceeding, we want to bound $|x_\alpha-x'_\alpha|_\cH$. First, we consider the case $\C\subset\C'$. Using $\Psi_\alpha(t_\alpha,x_\alpha,t'_\alpha,x_\alpha)-\Psi_\alpha(t_\alpha,x_\alpha,t'_\alpha,x'_\alpha)\leq 0$, the computation that
\begin{align*}
    \Psi_\alpha(t_\alpha,x_\alpha,t'_\alpha,x_\alpha)-\Psi_\alpha(t_\alpha,x_\alpha,t'_\alpha,x'_\alpha) = v(t'_\alpha,x'_\alpha)-v(t'_\alpha,x_\alpha)+\frac{\alpha}{2}|x_\alpha-x'_\alpha|^2_\cH,
\end{align*}
and the definition of $L$, we can get $\alpha|x_\alpha-x'_\alpha|_\cH \leq 2L$.
If $\C'\subset \C$, we use $\Psi_\alpha(t_\alpha,x'_\alpha,t'_\alpha,x'_\alpha)-\Psi_\alpha(t_\alpha,x_\alpha,t'_\alpha,x'_\alpha)\leq 0$, and
\begin{align*}
    \Psi_\alpha(t_\alpha,x'_\alpha,t'_\alpha,x'_\alpha)-\Psi_\alpha(t_\alpha,x_\alpha,t'_\alpha,x'_\alpha) = u(t_\alpha,x'_\alpha)-u(t_\alpha,x_\alpha)+\frac{\alpha}{2}|x_\alpha-x'_\alpha|^2_\cH
    \\
    -\Phi(t_\alpha,x'_\alpha) +\Phi(t_\alpha,x_\alpha).
\end{align*}
By the definition of $L$ and \eqref{e.|nabla_Phi|}, we can conclude that, in both cases,
\begin{align}\label{e.|x_alpha-x'_alpha|}
    \alpha|x_\alpha-x'_\alpha|_\cH \leq 2(L+M).
\end{align}

With this, we return to the proof. Since the function
\begin{align*}
    (t,x)\mapsto \Psi_\alpha(t,x,t'_\alpha,x'_\alpha)
\end{align*}
achieves its maximum at $(t_\alpha,x_\alpha)\in (0,T)\times \C$, by the assumption that $u$ is subsolution, we have
\begin{align}\label{e.u_sub_Phi_alpha}
    \alpha(t_\alpha-t'_\alpha) + \eps +\eps(T-t_\alpha)^{-2}+\partial_t\Phi(t_\alpha,x_\alpha)-\H\left(\alpha(x_\alpha - x'_\alpha)+\nabla\Phi(t_\alpha,x_\alpha)\right)\leq 0
\end{align}
On the other hand, since the function
\begin{align*}
    (t',x')\mapsto \Psi_\alpha(t_\alpha,x_\alpha,t',x')
\end{align*}
achieves its minimum at $(t'_\alpha,x'_\alpha)\in (0,\infty)\times \C'$, by the assumption that $v$ is subsolution, we have
\begin{align}\label{e.v_super_Phi_alpha}
    \alpha(t_\alpha-t'_\alpha) -\H\left(\alpha(x_\alpha - x'_\alpha)\right)\geq 0.
\end{align}
By \eqref{e.|nabla_Phi|} and \eqref{e.|x_alpha-x'_alpha|}, the arguments inside $\H$ in both \eqref{e.u_sub_Phi_alpha} and \eqref{e.v_super_Phi_alpha} have norms bounded by $2L+3M$. Taking the difference of \eqref{e.u_sub_Phi_alpha} and \eqref{e.v_super_Phi_alpha}, and using the definition of $V$ and \eqref{e.d_tPhi>}, we obtain that
\begin{align*}
    \eps\leq V|\nabla \Phi(t_\alpha,x_\alpha)|-\partial_t\Phi(t_\alpha,x_\alpha)\leq 0,
\end{align*}
contradicting the fact that $\eps>0$. Therefore, the desired result must hold.
\end{proof}

\section{Equations on an infinite-dimensional cone}\label{s.HJ_inft_d}
For a fixed positive integer $D$, let $\S^D$ be the space of $D\times D$-symmetric matrices, and $\S^D_+$ be the cone of $D\times D$-symmetric positive semidefinite matrices. We equip $\S^D$ with the inner product $a\cdot b = \tr(ab)$, for all $a,b\in\S^D$. We can view $\S^D_+$ as a closed convex cone in the Hilbert space $\S^D$. Naturally, $\S^D$ is endowed with the Borel sigma-algebra generated by the norm topology. For $a,b\in\S^D$, we write
\begin{align}\label{e.a>b_SD_+}
    a\geq b,\quad\text{if }a-b\in \S^D_+,
\end{align}
which defines a partial order on $\S^D$.

We work with the infinite-dimensional Hilbert space
\begin{align}\label{e.H_inft_d}
    \cH= L^2([0,1);\S^D)
\end{align}
namely, $\S^D$-valued squared integrable functions on $[0,1)$ endowed with the Borel sigma-algebra $\mathcal{B}_{[0,1)}$ and the Lebesgue measure. In addition to the Hilbert space $\cH$, we also need
\begin{align*}L^p = L^p([0,1);\S^D_+)
\end{align*}
for $p\in[1,\infty]$, whose norm is denoted by $|\cdot|_{L^p}$ with expression
\begin{align*}
    |\kappa|_{L^p} = \left(\int_0^1|\kappa(s)|^p\d s\right)^\frac{1}{p} \ \text{for}\ p\in[1,\infty), \quad\text{and}\quad |\kappa|_{L^\infty} =\esssup_{s\in[0,1]}|\kappa(s)|.
\end{align*}

We consider the following cone
\begin{align}\label{e.cone_M}
    \cM = \left\{\mu:[0,1)\to \S^D_+ \ \big|\  \text{$\mu$ is square-integrable, right-continuous with left limits, and increasing}  \right\}.
\end{align}
Here, $\mu$ is said to be increasing if $\mu(t)-\mu(s)\in\S^D_+$ whenever $t\geq s$. We view $\C\subset \cH$ by identifying every element in $\C$ with its equivalence class in $\cH$. Since $\{\mathds{1}_{[t,1)}\}_{t\in[0,1)}\subset\C$, it is immediate that $\C$ spans $\cH$. More precisely, \eqref{e.cl(C-C)=H} holds.

In this section, we study $\HJ(\cH,\C,\H)$ for $\cH$ and $\C$ given above. We start by introducing more notations and basic results in Section~\ref{s.HJ_inft_d_prelim}. The main results of this section are scattered in subsections afterwards. The comparison principle is given in Proposition~\ref{p.comp}. In Section~\ref{s.HJ_inft_d_cvg}, we show that any limit of finite-dimensional approximations is a viscosity solution (Proposition~\ref{p.lim_is_vis_sol}), and provide sufficient conditions for such a convergence (Proposition~\ref{p.exist_sol_approx}). In Section~\ref{s.HJ_inft_d_var_form}, we show that the Hopf--Lax formula and the Hopf formula are stable, when passed to the limit (Propositions~\ref{p.hopf_lax_cvg} and~\ref{p.hopf_cvg}). Lastly, in Section~\ref{s.HJ_inft_d_weak_boundary}, we briefly discuss a way to make sense of the boundary of $\C$ in a weaker notion.

Throughout, we denote elements in $\C$ by $\mu,\nu,\rho$; generic elements in $\cH$ by $\iota,\kappa$; and elements in finite-dimensional spaces by $x,y,z$.

\subsection{Preliminaries}\label{s.HJ_inft_d_prelim}

We introduce definitions and notations related to partitions of $[0,1)$, by which the finite approximations of $\HJ(\cH,\C,\H)$ are indexed. Projection maps and lifting maps between finite-dimensional approximations and their infinite-dimensional counterparts are used extensively. We record their basic properties in Lemma~\ref{l.basics_(j)}. We also need the projections of $\C$ and their dual cones, the properties of which are collected in Lemmas~\ref{l.dual_cone} and~\ref{l.proj_cones}. Lastly, in Lemma~\ref{l.derivatives}, we clarify the relation between the differentiability in finite-dimensional approximations and the one in infinite dimensions.

\subsubsection{Partitions}\label{s.partition}

We denote the collection of ordered tuples as partitions of $[0,1)$ by
\begin{align*}
    \J = \cup_{n\in \N}\left\{( t_1,t_2,\dots,t_n)\in(0,1]^n: 0< t_1 < t_2 <\cdots<t_{n-1}<t_n=1 \right\}.
\end{align*}
For every such tuple $j\in\J$, we set $t_0=0$, and denote by $|j|$ the cardinality of $j$. 

A natural partial order on $\J$ is given by the set inclusion. Under this partial order, a subcollection $\tilde \J\subset\J$ is said to be \textit{directed} if for every pair $j,j'\in \tilde \J$, there is $j''\in\tilde\J$ such that $j,j'\subset j''$. 

For each $j\in\J$, we associate a sigma-algebra $\F_j$ on $[0,1)$ generated by $\{[t_k,t_{k+1})\}_{t_k\in j}$. A subcollection $\tilde \J\subset\J$ is said to be \textit{generating} if $\tilde\J$ is directed, and the collection of sigma-algebras $\{\F_j\}_{j\in \tilde \J}$ generates the Borel sigma-algebra on $[0,1)$.

Let $\Junif$ be the collection of uniform partitions. A subcollection $\tilde \J\subset\J$ is said to be \textit{good} if $\tilde \J\subset \Junif$ and $\tilde \J$ is generating. Examples of good collections of partitions include $\Junif$ itself, and the collection of dyadic partitions.

In the following, we denote by $\Jgen$ a generic generating collection of partitions, and by $\Jgood$ a generic good collection.

Then, we introduce the notions of nets and the convergence of a net.
For any directed subcollection $\widetilde\J\subset \J$, a collection of elements $(x_j)_{j\in \widetilde\J}$, indexed by $\widetilde\J$, from some set $\cX$ is called a \textit{net} in $\cX$. If $\cX$ is a topological space, a net $(x_j)_{j\in\widetilde\J}$ is said to converge in $\cX$ to $x$ if for every neighborhood $\mathcal N$ of $x$, there is $j_{\mathcal N}\in \J$ such that $x_j\in\mathcal N$ for every $j\in\widetilde\J$ satisfying $j\supset j_{\mathcal N}$. In this case, we write $\lim_{j\in\widetilde \J} x_j = x$ in $\mathcal X$.

For each $j\in\J$ and every $\iota\in L^1$, we define
\begin{align}\label{e.(j)}
    \iota^\j(t) =  \sum_{k=1}^{|j|}\mathds{1}_{[t_{k-1},t_{k})}(t)\frac{1}{t_{k}-t_{k-1}}\int_{t_{k-1}}^{t_{k}} \iota(s) \d s,\quad\forall t\in[0,1).
\end{align}
It is easy to see that $\iota^\j$ is
the conditional expectation of $\iota$ on $\F_j$, namely,
\begin{align}\label{e.iota^j_cond_exp}
    \iota^\j(U) = \E \left[\iota(U)|\F_j\right].
\end{align}
Here, and throughout, $U$ is uniform random variable on $[0,1)$ defined on the probability space $([0,1),\mathcal{B}_{[0,1)},\mathrm{Leb})$. 
By Jensen's inequality, we have $\iota^\j\in L^p$ if $\iota\in L^p$, for any $p\in[1,\infty)$, which also holds obviously for $p=\infty$. In particular, $\iota^\j\in \cH$ if $\iota\in \cH$. It is straightforward to see that $\iota^\j\in \cM$ if $\iota\in\cM$.

\subsubsection{Projections and lifts}
We introduce finite-dimensional Hilbert spaces indexed by $\J$. For each $j\in\J$, we define
\begin{align}\label{e.def_H^j}
    \cH^j = (\S^D)^{|j|}
\end{align}
 equipped with the inner product
\begin{align}\label{e.inner_product_H^j}
    \la x,y\ra_{\cH^j} = \sum_{k=1}^{|j|} (t_{k}-t_{k-1})x_k\cdot y_k,\quad\forall x,y\in \cH^j.
\end{align}

For each $j\in\J$, we define the projection map $\pj_j:\cH \to \cH^j$ by
\begin{align}\label{e.def_pj_iota}
    \pj_j \iota = \left(\frac{1}{t_k - t_{k-1}}\int_{t_{k-1}}^{t_k}\iota(s)\d s\right)_{k\in \{1,\dots, |j|\}},\quad\forall \iota\in\cH.
\end{align}
Correspondingly, we define the associated lift map $\lf_j :\cH^j \to \cH$:
\begin{align}\label{e.lf_jx}
    \lf_j x = \sum_{k=1}^{|j|}x_k\mathds{1}_{[t_{k-1},t_{k})},\quad\forall x\in\cH^j. 
\end{align}

We define projections and lifts acting on functions.
\begin{definition}[Lifts and projections of functions]\label{d.lift_proj}
Let $j\in\J$.
\begin{itemize}
    \item For any $\cE\subset \cH$ and any $g:\cE\to\R$, its $j$-projection $g^j:\pj_j\cE\to\R$ is given by $g^j = g\circ \lf_j$.
    
    \item For any $\mathcal T\times \cE\subset \R_+\times\cH$ and any $f:\mathcal T\times \cE\to\R$, its $j$-projection $f^j:\mathcal T\times\pj_j \cE\to\R$ is given by $f^j(t,\cdot) = f(t,\lf_j(\cdot))$ for each $t\in\mathcal T$.

    \item For any $\cE\subset \cH^j$ and any function $g:\cE\to\R$, its lift $g^\uparrow:\lf_j\cE\to\R$, is given by $g^\uparrow = g\circ \pj_j$.
    \item for any $\mathcal T\times \cE\subset\R_+\times \cH^j$ and any $f:\mathcal T\times \cE\to\R$, its lift $f^\uparrow:\mathcal T\times \lf_j\cE\to\R$, is defined by $f^\uparrow(t,\cdot) = f(t,\pj_j(\cdot))$ for each $t\in\mathcal T$.
\end{itemize}

\end{definition}

\begin{remark}\rm
Let us clarify our use of indices. Objects with superscript $j$, for instance, $\cH^j$, $\C^j$ (introduced later in \eqref{e.C^j}), $f^j$, are always projections of infinite-dimensional objects either mapped directly by $\pj_j$ or induced by $\pj_j$. Superscript $\j$ is reserved for~\eqref{e.(j)}. Other objects directly associated with $j$ or whose existence depends on $j$ are labeled with subscript $j$. \end{remark}

We record the basic properties of projections and lifts in the following lemma.

\begin{lemma}\label{l.basics_(j)}
For every $j\in\J$, the following hold:
\begin{enumerate}
\item \label{i.adjoint}  $\pj_j$ and $\lf_j$ are adjoint to each other: $\la \pj_j\iota, x\ra_{\cH^j}=\la \iota,\lf_j x\ra_{\cH}$ for every $\iota\in\cH$ and $x\in \cH^j$;

\item \label{i.isometric}  $\lf_j$ is isometric: $\la \lf_jx,\lf_jy\ra_{\cH}= \la x,y\ra_{\cH^j}$ for every $x,y\in\cH^j$;

\item \label{i.pj_lf=ID}
 $\pj_j\lf_j$ is the identity map on $\cH^j$:
$\pj_j\lf_j x=x$ for every $x\in\cH^j$;

\item \label{i.lf_pj_(j)} $\lf_j\pj_j \iota = \iota^\j$ and $\pj_j\iota = \pj_j\iota^\j$ for every $\iota\in\cH$;

\item \label{i.|iota^j|<|iota|}
$\pj_j$ is a contraction: $|\pj_j \iota|_{\cH^j}\leq |\iota|_\cH$, or equivalently $|\iota^\j|_\cH\leq |\iota|_\cH$, for every $\iota\in\cH$;

\item \label{i.projective} 
if $j'\in\J$ satisfies $j\subset j'$, then $\pj_j \lf_{j'}\pj_{j'}\iota =\pj_j\iota$ for every $\iota \in \cH$.

\end{enumerate}
In addition, the following results on convergence hold:
\begin{enumerate}
    \setcounter{enumi}{6}
    \item \label{i.cvg} for every $\iota\in\cH$, $\lim_{j\in \Jgen}\iota^\j =\iota$ in $\cH$;

\item \label{i.cvg_iota^(j)_j} for any net $(\iota_j)_{j\in\Jgen}$ in $\cH$, if $\lim_{j\in\Jgen} \iota_j =\iota$ in $\cH$,  then $\lim_{j\in \Jgen}\iota^\j_j=\iota$ in $\cH$.
\end{enumerate}

\end{lemma}

\begin{proof}

Part~\eqref{i.adjoint}. We can compute:
\begin{align*}
    \la \pj_j \iota, x\ra_{\cH^j} &= \sum_{k=1}^{|j|}(t_k-t_{k-1}) \left(\frac{1}{t_k-t_{k-1}}\int_{t_{k-1}}^{t_k} \iota(s)\d s\right) \cdot x_k\\
    &=\sum_{k=1}^{|j|}\int_{t_{k-1}}^{t_k} \iota(s)\cdot x_k\d s =\int_0^1\iota(s)\cdot\left(\sum_{k=1}^{|j|}\mathds{1}_{[t_{k-1},t_k)}(s)x_k\right)\d s = \la \iota,\lf_j x\ra_{\cH}.
\end{align*}

Part~\eqref{i.isometric}. We use~\eqref{e.lf_jx} to compute explicitly to get the desired result.

Part~\eqref{i.pj_lf=ID}. Definitions of $\pj_j$ in \eqref{e.def_pj_iota} and $\lf_j$ in \eqref{e.lf_jx} directly yield \eqref{i.pj_lf=ID}.

Part~\eqref{i.lf_pj_(j)}. Comparing the definitions of $\pj_j$, $\lf_j$ and $\iota^\j$ in \eqref{e.(j)}, we can easily deduce the first identity in~\eqref{i.lf_pj_(j)}. The second identity follows from the first identity and \eqref{i.pj_lf=ID}.

Part~\eqref{i.|iota^j|<|iota|}. We use~\eqref{e.iota^j_cond_exp} and Jensen's equality to see
\begin{align*}
    \left|\iota^\j\right|^2_\cH = \E \left|\iota^\j(U)\right|^2= \E|\E[\iota(U)|\F_j]|^2\leq \E|\iota(U)|^2=|\iota|_\cH^2.
\end{align*}
The equivalent formulation follows from \eqref{i.isometric} and \eqref{i.lf_pj_(j)}.

Part~\eqref{i.projective}. We can directly use the definitions of projections and lifts. Heuristically, $j$ is a coarser partition and $j'$ is a refinement of $j$. The map $\lf_{j'}\pj_{j'}$ has the effect of locally averaging $\iota$ with respect to the finer partition $j'$. On the other hand, $\pj_j$ is defined via local averaging with respect to the coarser $j$. The result follows from the fact that local averaging first with respect to a finer partition and then to a coarser partition is equivalent to local averaging directly with respect to the coarser one.

Part~\eqref{i.cvg}. We argue by contradiction. We assume that there exists $\eps>0$ such that for every $j\in\Jgen$, there is some $j'\supset j$ satisfying $|\iota^{(j')}-\iota|_\cH\geq\eps$. Let us construct a sequence recursively. We start by choosing $j_1\in\Jgen$ to satisfy $|\iota^{(j_1)}-\iota|_\cH\geq\eps$. For $m> 1$, we choose $j_{m+1}\supset(j_m\cup j_{m}')$ such that $|\iota^{(j_{m+1})}-\iota|_\cH\geq\eps$, where we let $j_{m}'\in\Jgen$ be any partition satisfying $\max_{1\leq i\leq |j_{m}'|}\{|t_i-t_{i-1}|\}<\frac{1}{m}$. Denote this sequence by $\Jgen'$, which is directed and generating. 

By~\eqref{e.iota^j_cond_exp}, for $j_m,j_n\in\Jgen'$ such that $n\geq m$, we have
\begin{equation*}
    \iota^{(j_m)}(U)=\E\left[\iota(U)|\F_{j_m}\right]=\E\left[\E\left[\iota(U)|\F_{j_n}\right]|\F_{j_m}\right]=\E\left[\iota^{(j_n)}(U)|\F_{j_m}\right],
\end{equation*}
which implies that $(\iota^{(j_n)})_{n\in\N}$ is a martingale with respect to $(\F_{j_n})_{n\in\N}$. By the martingale convergence theorem, this sequence converges to $\iota$ in $\cH$ as $n\to\infty$, which is a contradiction to our construction of the sequence.

Part~\eqref{i.cvg_iota^(j)_j}. By the triangle inequality and~\eqref{i.|iota^j|<|iota|}, we have
\begin{align*}
    \left|\iota^\j_j - \iota\right|_\cH\leq \left|\iota^\j_j - \iota^\j\right|_\cH+ \left|\iota^\j - \iota\right|_\cH\leq \left|\iota_j - \iota\right|_\cH+ \left|\iota^\j - \iota\right|_\cH.
\end{align*}
Then,~\eqref{i.cvg_iota^(j)_j} follows from~\eqref{i.cvg}.
\end{proof}

\subsubsection{Cones and dual cones}

For each $j\in \J$, we introduce
\begin{align}\label{e.C^j}
    \cM^j = \{x \in \cH^j: 0\leq x_1\leq x_2\leq \cdots \leq x_{|j|} \},
\end{align}
where we used the notation in \eqref{e.a>b_SD_+}. It is clear that $\C^j$ and $\cH^j$ satisfy \eqref{e.cl(C-C)=H}.

Recall the definition of dual cones in~\eqref{e.C^*}.

\begin{lemma}[Characterizations of dual cones]\label{l.dual_cone}

\leavevmode

\begin{enumerate}
    \item \label{i.dual-cone-M^j}
For each $j\in\J$, the dual cone of $\cM^j$ in $\cH^j$ is
\begin{align*}
    (\cM^j)^* = \left\{x\in\cH^j: \sum_{i=k}^{|j|}(t_i-t_{i-1})x_i\in\S^D_+,\quad\forall k\in\{1,2,\dots,|j|\}\right\}.
\end{align*}

    \item \label{i.dual-cone-M}
The dual cone of $\cM$ in $\cH$ is
\begin{align*}
    \cM^* = \left\{\iota\in \cH: \int_t^1\iota(s)\d s\in\S^D_+,\quad \forall t\in [0,1) \right\}.
\end{align*}
\end{enumerate}

\end{lemma}
\begin{proof}
Part~\eqref{i.dual-cone-M^j}.
We denote the set on the right-hand side by $\RHS$. We first show that $(\cM^j)^*\subset \RHS$. Let $x\in(\cM^j)^*$. For every $k$ and every $a\in\S^D_+$, we can choose $y\in\cM^j$ such that $0=y_1=\cdots=y_{k-1}$ and $y_k=\cdots=y_{|j|}=a$. Then, we have
\begin{align*}
    \sum_{i=k}^{|j|}(t_i-t_{i-1})x_i\cdot a&=\sum_{i=k}^{|j|}(t_i-t_{i-1})x_i\cdot y_i\geq 0,
\end{align*}
which implies that $x\in \RHS$. In the other direction, we assume $x\in\RHS$. For every $y\in\cM^j$, by setting $y_0=0$, since $y_k-y_{k-1}\in\S^D_+$ for all $k$, we have
\begin{align*}
    \sum_{i=1}^{|j|}(t_i-t_{i-1})x_i\cdot y_i&=\sum_{k=1}^{|j|}\left(\sum_{i=k}^{|j|}(t_i-t_{i-1})x_i\cdot (y_k-y_{k-1})\right)\geq 0,
\end{align*}
which gives that $x\in(\cM^j)^*$. Now we can conclude that $(\cM^j)^*=\RHS$ as desired.

Part~\eqref{i.dual-cone-M}.
We denote the set on the right-hand side by $\RHS$. Let $\iota\in\C^*$. For any $a\in\S^D_+$ and $t\in [0,1)$, we set $\mu = a\mathds{1}_{[t,1)}$. It is clear that $\mu\in\C$. Due to $\la \iota, \mu\ra_\cH\geq 0$ by duality, we deduce that $\iota\in\RHS$. 

Now, let $\iota \in \RHS$. We argue by contradiction and assume $\iota\not\in \C^*$. Then, by definition, there is $\mu\in\C$ such that $\la \iota,\mu\ra_\cH<0$. By Lemma~\ref{l.basics_(j)}~\eqref{i.cvg}, there is a partition $j$ such that $\la\iota^\j,\mu^\j\ra_\cH<0$. Due to Lemma~\ref{l.basics_(j)}~\eqref{i.isometric} and~\eqref{i.lf_pj_(j)}, this can be rewritten as $\la\pj_j\iota,\pj_j\mu\ra_{\cH^j}<0$. 

On the other hand,
by the definition of $\pj_j$ in \eqref{e.def_pj_iota}, we can compute that, for every $k$,
\begin{align*}
    \sum_{i=k}^{|j|}(t_i-t_{i-1})(\pj_j\iota)_i = \int_{t_{k-1}}^1\iota(s)\d s\in\S^D_+
\end{align*}
by the assumption that $\iota \in \RHS$. Hence, by \eqref{i.dual-cone-M^j}, we have $\iota\in(\C^j)^*$. Since $\mu$ is increasing as $\mu\in\C$, it is easy to see that $\pj_j\mu\in\C^j$. The detailed computation can be seen in \eqref{e.pj_mu>0}. Therefore, we must have $\la \pj_j\iota,\pj_j\mu\ra_{\cH^j}\geq 0$, reaching a contradiction.
\end{proof}

\begin{lemma}\label{l.proj_cones}
For every $j\in \J$, the following hold:
\begin{enumerate}
    \item \label{i.lM^j} $\lf_j(\cM^j)\subset \cM$;
    \item \label{i.lM^j*} $\lf_j((\cM^j)^*)\subset \cM^*$;
    \item \label{i.pjM} $\pj_j(\cM)= \cM^j$;
    \item \label{i.pjM*} $\pj_j(\cM^*)= (\cM^j)^*$;
    \item \label{l.mu-mu^j}
    $\mu\in \mu^\j+\C^*$, for every $\mu\in \C$.
\end{enumerate}
\end{lemma}

\begin{proof}
We first show that
\begin{align}\label{e.pj(C)_subset_C}
    \pj_j(\C)\subset \C^j
\end{align}
and then verify each claim. For every $\mu\in\cM$, it follows from the definition that $\pj_j \mu\in\cH^j$. Since $\mu$ is increasing, setting $(\pj_j \mu)_0=0$ by our convention, we get,  
\begin{align}
(\pj_j \mu)_k-(\pj_j \mu)_{k-1}&=\frac{1}{t_k-t_{k-1}}\int _{t_{k-1}}^{t_k}\mu(s)\d s-\frac{1}{t_{k-1}-t_{k-2}}\int _{t_{k-2}}^{t_{k-1}}\mu(s)\d s\notag
\\
&\geq \mu(t_{k-1})-\mu(t_{k-1})=0,\label{e.pj_mu>0}
\end{align}
for $k\in\{2,\cdots,|j|\}$. Clearly when $k=1$, $(\pj_j \mu)_1=\frac{1}{t_1}\int _{0}^{t_1}\mu(s)ds\in\S^D_+$. Hence, we have $\pj_j \mu\in\cM^j$ and thus \eqref{e.pj(C)_subset_C}.

Part~\eqref{i.lM^j}. For any $x\in\cM^j$, recall the definition of $\lf_j x$ in \eqref{e.lf_jx}. Since $x_k\geq x_{k-1}$ for each $k$, it is clear that $\lf_j x$ is increasing and thus belongs to $\C$.

Part~~\eqref{i.lM^j*}. Let $x\in (\cM^j)^*$. For every $\mu\in \cM$, recalling the definition of $\pj_j\mu$ in \eqref{e.def_pj_iota}, we have
\begin{align*}
    \int_0^1\sum_{k=1}^{|j|}\mathds{1}_{[t_{k-1},t_k)}(s)x_i\cdot \mu(s)\d s&=\sum_{k=1}^{|j|}\int_{t_{k-1}}^{t_k}x_k\cdot \mu(s)\d s=\sum_{k=1}^{|j|}(t_k-t_{k-1})x_k\cdot\left(\frac{1}{t_k-t_{k-1}}\int_{t_{k-1}}^{t_k}\mu(s) \d s\right)\\
    &=\sum_{k=1}^{|j|}(t_k-t_{k-1})x_k\cdot (\pj_j \mu)_k\geq 0,
\end{align*}
where the last inequality holds due to $x\in(\cM^j)^*$ and $\pj_j \mu\in\cM^j$ by~\eqref{e.pj(C)_subset_C}. This implies that $\lf_j x\in\cM^*$, and thus $\lf_j ((\cM^j)^*)\subset \cM^*$.

Part~~\eqref{i.pjM}. For every $x\in\C^j$, by \eqref{i.lM^j}, we have $\lf_j x \in \C$. Lemma~\ref{l.basics_(j)}~\eqref{i.pj_lf=ID} implies that $x = \pj_j\lf_jx$. Hence, we get $\C^j\subset \pj_j(\C)$. Then, \eqref{i.pjM} follows from this and \eqref{e.pj(C)_subset_C}.

Part~~\eqref{i.pjM*}. Let $\iota\in\cM^*$. 
For every $x\in \cM^j$, we have by Lemma~\ref{l.basics_(j)}~\eqref{i.adjoint} that $\la \pj_j\iota, x\ra_{\cH^j} = \la \iota,\lf_j x\ra_\cH\geq 0$ due to $\lf_j x \in \C$ ensured by \eqref{i.lM^j}. Hence, we have $\pj_j(\C^*)\subset (\C^j)^*$. For the other direction, let $x\in (\C^j)^*$. Lemma~\ref{l.basics_(j)}~\eqref{i.pj_lf=ID} gives $x=\pj_j\lf_j x$. Invoking \eqref{i.lM^j*}, we can deduce that $(\C^j)^*\subset \pj_j(\C^*)$, completing the proof of \eqref{i.pjM*}.

Part~\eqref{l.mu-mu^j}.
We show that $\mu-\mu^\j\in\C^*$. Let $\tau\in [0,1)$ and $a\in\S^D_+$. We choose $t_{k_0}\in j$ such that $\tau\in[t_{k_0-1},t_{k_0})$. Using the definition of $\mu^\j$ in \eqref{e.(j)}, we can compute that
\begin{align*}
    &\int_\tau^1 \left(\mu-\mu^\j\right)(s)\d s
    \\
    &=\left(\int_\tau^{t_{k_0}}\mu(s)\d s+\int_{t_{k_0}}^1\mu(s)\d s\right)-\left(\frac{t_{{k_0}}-\tau}{t_{k_0}-t_{{k_0}-1}}\int_{t_{{k_0}-1}}^{t_{k_0}}\mu(s)\d s+\int_{t_{k_0}}^1\mu(s)\d s\right),\\
    &= \int_\tau^{t_{k_0}}\mu(s)\d s - \frac{t_{{k_0}}-\tau}{t_{k_0}-t_{{k_0}-1}}\int_{t_{{k_0}-1}}^{t_{k_0}}\mu(s)\d s
    \\
    &=(t_{k_0}-\tau)\left(\frac{1}{t_{k_0}-\tau}\int_\tau^{t_{k_0}}\mu(s)\d s-\frac{1}{t_{k_0}-t_{k_0-1}}\int_{t_{k_0-1}}^{t_{k_0}} \mu(s)\d s\right)
    \geq 0,
\end{align*}
where the last inequality follows from the fact that $\mu$ is increasing. By Lemma~\ref{l.dual_cone}~\eqref{i.dual-cone-M}, we conclude that $\mu-\mu^\j\in\C^*$ as desired.
\end{proof}

\subsubsection{Derivatives}

Recall Definition~\ref{d.differentiability}~\eqref{i.def_differentiable_C} for the differentiability of functions defined on $\C$. We denote by $\nabla_j$ the differential operator on functions defined on $\C^j$.

\begin{lemma}\label{l.derivatives}
For every $j\in \J$, the following hold.

\begin{enumerate}
    \item \label{i.char_nabla_j} If $g:\cM\to\R$ is differentiable at $\lf_j x$ for some $x\in \cM^j$, then $g^j:\cM^j\to\R$ is differentiable at $x$ and its differential is given by $\nabla_j g^j(x) = \pj_j(\nabla g(\lf_j x))$.
    
    \item \label{i.lift_gradient}
If $g:\cM^j\to\R$ is differentiable at $ x$ for some $x\in \cM^j$, then $g^\uparrow:\cM\to\R$ is differentiable at every $\mu\in\cM$ satisfying $\pj_j\mu = x$ and its differential is given by $\nabla g^\uparrow(\mu) = \lf_j(\nabla_j g(x)) $.
\end{enumerate}

\end{lemma}
\begin{proof}
Part~\eqref{i.char_nabla_j}
Recall that by definition, $g^j=g\circ l_j$. For every $y\in \cM^j$, we can see that
\begin{align*}
    g^j(y)-g^j(x)&= g\circ \lf_j (y)-g\circ\lf_j(x)\\
    &=\la \nabla g(\lf_j x),\lf_j y-\lf_j x\ra_{\cH}+o\left(|\lf_j y-\lf_j x|_\cH\right),\\
    &=\la \pj_j(\nabla g(\lf_j x)), y-x\ra_{\cH^j}+o\left(|y-x|_{\cH^j}\right),
\end{align*}
where the last equality follows from Lemma~\ref{l.basics_(j)}~~\eqref{i.adjoint} and~\eqref{i.isometric}. 

Part~\eqref{i.lift_gradient}.
Recall that by definition, $g^\uparrow = g\circ\pj_j$. Let $\mu\in\cM$ satisfy $\pj_j \mu=x$. Then for any $\nu\in\cM$, we get
\begin{align*}
    g^\uparrow(\nu)-g^\uparrow(\mu)&=g\circ\pj_j(\nu)-g\circ\pj_j(\mu),\\
    &=\la \nabla_j g(x), \pj_j\nu-x\ra_{\cH^j}+o\left(|\pj_j\nu-x|_{\cH^j}\right),\\
    &=\la \lf_j(\nabla_j g(x)), \nu-\mu\ra_{\cH}+o\left(|\nu-\mu|_{\cH}\right),
\end{align*}
where we used Lemma~\ref{l.basics_(j)}~~\eqref{i.adjoint} and~\eqref{i.|iota^j|<|iota|}.
\end{proof}

\subsection{Comparison principle}

To compensate for the lack of compactness in infinite dimensions, we need Stegall's variational principle \cite[Theorem on page 174]{stegall1978optimization} (see also \cite[Theorem~8.8]{cardaliaguet2010notes}). 

\begin{theorem}[Stegall's variational principle] \label{t.stegall}
Let $\cE$ be a convex and weakly compact set in a separable Hilbert space $\cX$ and $g:\cE\to\R$ be an upper semi-continuous function bounded from above. Then, for every $\delta>0$, there is $\iota\in \cX$ satisfying $|\iota|_\cX\leq \delta$ such that $g+\la \iota,\cdot\ra_\cX$ achieves maximum on $\cE$.
\end{theorem}

Originally, $\cE$ is only required to satisfy the Radon-Nikodym property which is weaker than being convex and weakly compact (see discussion on \cite[page~173]{stegall1978optimization}).

The goal of this subsection is to prove the following.

\begin{proposition}[Comparison principle]\label{p.comp}
Suppose that $\H:\cH\to\R$ is locally Lipschitz.
Let $u$ be a Lipschitz viscosity subsolution and $v$ be a Lipschitz viscosity supersolution of $\HJ(\cH,\C,\H)$. If $u(0,\cdot)\leq v(0,\cdot)$, then~${u\leq v}$.
\end{proposition}

\begin{proof}[Proof of Proposition~\ref{p.comp}]
It suffices to show $u(t,\cdot)-v(t,\cdot)\leq0$ for all $t\in[0,T)$ for any $T>0$. Henceforth, we fix any $T>0$. 
We set $L = \|u\|_\mathrm{Lip}\vee \|v\|_\mathrm{Lip}$, $M=2L+3$ and $V$ to be the Lipschitz coefficient of $\H$ restricted to the centered ball with radius $2L+M+3$.
We proceed in steps.

Step~1.
Let $\theta:\R\to\R_+$ be an increasing smooth function satisfying
\begin{align*}
    |\theta'|\leq 1\qquad\text{and}\qquad (r-1)_+\leq \theta(r)\leq r_+,\quad\forall r\in\R,
\end{align*}
where $\theta'$ is the derivative of $\theta$. For $R>1$ to be determined, we define
\begin{align*}
    \Phi(t,\mu) = M\theta\left(\left(1+|\mu|^2_\cH\right)^\frac{1}{2}+Vt-R\right),\quad\forall (t,\mu)\in \R_+\times \C.
\end{align*}
It is immediate that
\begin{gather}
    \sup_{(t,\mu)\in\R_+\times\C}|\nabla\Phi(t,\mu)|_\cH\leq M,\label{e.|nabla_Phi|_inft_d}
    \\
    \partial_t\Phi\geq V|\nabla \Phi|_\cH,\label{e.d_tPhi>_inft_d}
    \\
    \Phi(t,\mu)\geq M(|\mu|_\cH-R-1)_+, \quad\forall (t,\mu)\in\R_+\times\C. \label{e.Phi(t,x)>M|x|_inft_d}
\end{gather}

For $\eps,\sigma\in(0,1)$ to be determined, we consider
\begin{align*}
    \Psi(t,\mu,t',\mu') = u(t,\mu)-v(t',\mu') - \frac{1}{2\eps}(|t-t'|^2 + |\mu-\mu'|_\cH^2) - \Phi(t,\mu) -\sigma t - \frac{\sigma}{T-t},
    \\
    \quad\forall (t,\mu,t',\mu') \in [0,T)\times\C\times\R_+\times\C.
\end{align*}
Setting $C_0=u(0,0)-v(0,0)$, and using \eqref{e.Phi(t,x)>M|x|_inft_d} and the definition of $L$, we have
\begin{align}
    \Psi(t,\mu,t',\mu')\leq C_0 + L(2|t|+2|\mu|_\cH+|t-t'|+|\mu-\mu'|_\cH) - \frac{1}{2\eps}(|t-t'|^2 + |\mu-\mu'|_\cH^2) \label{e.Psi<}
    \\
    - M(|\mu|_\cH-R-1)_+ - \frac{\sigma}{T-t}.\notag
\end{align}
Hence, by the definition of $M$, $\Psi$ is bounded from above and its supremum is achieved over a bounded set. Invoking Theorem~\ref{t.stegall}, for $\delta\in(0,1)$ to be chosen, there is $(\bar s,\bar \iota,\bar s',\bar\iota')\in\R\times\cH\times\R\times \cH$ satisfying
\begin{align}\label{e.s,iota,s',iota'<delta}
    |\bar s|,\ |\bar\iota|_\cH,\ |\bar s'|,\ |\bar\iota'|_\cH\leq \delta,
\end{align}
such that the function
\begin{align*}
    \bar\Psi(t,\mu,t',\mu') = \Psi(t,\mu,t',\mu')-\bar s t - \la\bar \iota,\mu\ra_\cH -\bar s' t' -\la\bar \iota',\mu'\ra_\cH,\quad\forall (t,\mu,t',\mu')\in[0,T)\times\C\times\R_+\times\C
\end{align*}
achieves its maximum at $(\bar t,\bar\mu,\bar t',\bar\mu')$.

Step~2.
We derive bounds on $|\bar\mu|_\cH$, $|\bar\mu-\bar\mu'|_\cH$ and $|\bar t-\bar t'|$. Using $\bar\Psi(0,0,0,0)\leq \bar\Psi(\bar t,\bar\mu,\bar t',\bar\mu')$, \eqref{e.Psi<} and $\bar t\leq T$, we have
\begin{align*}
     C_0 &\leq  \frac{\eps}{T}+\Psi(\bar t,\bar\mu,\bar t',\bar\mu') +2\delta|\bar \mu|_\cH+2T\delta +\delta|\bar t-\bar t'|+\delta|\bar \mu-\bar \mu'|_\cH 
     \\
    &\leq \frac{\eps}{T}+C_0+2LT+\left( 2L|\bar\mu|_\cH - M(|\bar\mu|_\cH-R-1)_+\right)+\left( L|\bar t-\bar t'|-\frac{1}{2\eps}|\bar t-\bar t'|^2\right)
    \\
    &\qquad\qquad+\left( L|\bar\mu-\bar\mu'|_\cH-\frac{1}{2\eps}|\bar\mu-\bar\mu'|_\cH^2\right)+2\delta|\bar \mu|_\cH+2T\delta +\delta|\bar t-\bar t'|+\delta|\bar \mu-\bar \mu'|_\cH 
    \\
    &\leq \left( 2(L+\delta)|\bar\mu|_\cH - M(|\bar\mu|_\cH-R-1)_+\right) + \left(\frac{\eps}{T}+C_0 +2LT+\eps(L+\delta)^2+2T\delta\right).
\end{align*}
By this and the definition of $M$, there is $C_1>0$ such that, for all $\eps,\delta\in(0,1)$ and all $R>1$, 
\begin{align}\label{e.|bar_mu|<eps^-1/2}
    |\bar\mu|_\cH\leq C_1R.
\end{align}
Since
\begin{align*}
    0\geq \bar\Psi(\bar t,\bar\mu,\bar t',\bar\mu)-\bar\Psi(\bar t,\bar\mu,\bar t',\bar\mu') = v(\bar t',\bar\mu') - v(\bar t',\bar \mu)+\frac{1}{2\eps}|\bar \mu -\bar\mu'|^2_\cH+\la \bar\iota',\bar\mu'-\bar\mu\ra_\cH,
\end{align*}
by the definition of $L$ and \eqref{e.s,iota,s',iota'<delta}, we can get
\begin{align}\label{e.mu-mu'<2alphaL}
    |\bar \mu -\bar \mu'|_\cH\leq 2 (L+\delta)\eps.
\end{align}
Similarly, by
\begin{align*}
    0\geq \bar\Psi(\bar t,\bar\mu,\bar t,\bar\mu')-\bar\Psi(\bar t,\bar\mu,\bar t',\bar\mu') = v(\bar t',\bar\mu') - v(\bar t,\bar \mu')+\frac{1}{2\eps}|\bar t -\bar t'|^2+\bar s'(\bar t'-\bar t),
\end{align*}
we have
\begin{align}\label{e.t-t'<2alphaL}
    |\bar t -\bar t'|\leq 2 (L+\delta)\eps.
\end{align}

Step~3. 
We show that for every $\sigma,\eps\in(0,1)$, every $R>1$, and sufficiently small $\delta$, we have either $\bar t=0$ or $\bar t'=0$
We argue by contradiction and assume that $\bar t>0$ and $\bar t'>0$.
Since the function
\begin{align*}
    (t,\mu)\mapsto \bar\Psi(t,\mu,\bar t',\bar\mu')
\end{align*}
achieves its maximum at $(\bar t, \bar \mu)\in(0,T)\times\C$, by the assumption that $u$ is a subsolution, we have
\begin{align}\label{e.u_sub=>}
    \frac{1}{\eps}(\bar t- \bar t')+\partial_t\Phi(\bar t,\bar \mu) +\sigma + \sigma(T-\bar t)^{-2}+\bar s-\H\left(\frac{1}{\eps}(\bar \mu -\bar \mu') + \nabla \Phi(\bar t,\bar \mu)+\bar \iota\right)\leq 0.
\end{align}
Since the function
\begin{align*}
    (t',\mu')\mapsto \bar\Psi(\bar t,\bar\mu, t',\mu')
\end{align*}
achieves its maximum at $(\bar t', \bar \mu')\in(0,\infty)\times\C$, by the assumption that $v$ is a supersolution, we have
\begin{align}\label{e.v_sup=>}
    \frac{1}{\eps}(\bar t- \bar t') -\bar s'-\H\left(\frac{1}{\eps}(\bar \mu -\bar \mu') -\bar \iota'\right)\geq 0.
\end{align}
By \eqref{e.|nabla_Phi|_inft_d}, \eqref{e.s,iota,s',iota'<delta} and \eqref{e.mu-mu'<2alphaL}, for $\eps,\delta\in(0,1)$, we have
\begin{align*}
    \left|\frac{1}{\eps}(\bar \mu -\bar \mu') + \nabla \Phi(\bar t,\bar \mu)+\bar \iota\right|_\cH, \ \left|\frac{1}{\eps}(\bar \mu -\bar \mu') -\bar \iota'\right|_\cH\leq 2L+M+3.
\end{align*}
Taking the difference of terms in \eqref{e.u_sub=>} and \eqref{e.v_sup=>}, by the definition of $L$, \eqref{e.d_tPhi>_inft_d} and \eqref{e.s,iota,s',iota'<delta}, we obtain
\begin{align*}
    \sigma\leq -\bar s -\bar s'+V|\nabla \Phi(\bar t,\bar \mu)|_\cH + V(|\bar\iota|_\cH+|\bar\iota'|_\cH) - \partial_t \Phi(\bar t,\bar \mu)\leq 2(1+V)\delta.
\end{align*}
By making $\delta$ sufficiently small, we reach a contradiction, and thus we must have either $\bar t =0$ or $\bar t'=0$.

Step~4. 
We conclude our proof. Let us consider the case $\bar t=0$. Fixing any $(t,\mu)\in[0,T)\times\C$, by $\bar\Psi(t,\mu,t,\mu)\leq \bar\Psi(\bar t,\bar\mu,\bar t',\bar\mu')$, we have
\begin{align*}
    \Psi(t,\mu,t,\mu)\leq \Psi(\bar t,\bar\mu,\bar t',\bar\mu') + \delta(4 T+2 |\mu|_\cH + 2C_1R+ 2 (L+\delta)\eps)
\end{align*}
where we used $t,\bar t<T$, \eqref{e.s,iota,s',iota'<delta}, \eqref{e.|bar_mu|<eps^-1/2} and \eqref{e.mu-mu'<2alphaL}. Due to $u(0,\cdot)\leq v(0,\cdot)$ and $\bar t=0$, using \eqref{e.mu-mu'<2alphaL} and \eqref{e.t-t'<2alphaL}, we can see 
\begin{align*}
    \Psi(\bar t,\bar\mu,\bar t',\bar\mu') \leq u(0,\bar\mu)-v(\bar t',\bar \mu')\leq v(0,\bar\mu)-v(\bar t',\bar \mu')\leq L|\bar t-\bar t'|+L|\bar\mu-\bar\mu'|_\cH \leq 4L(L+\delta)\eps.
\end{align*}
Combining the above two displays and recalling the definition of $\Psi$, we get
\begin{align*}
    u(t,\mu)-v(t,\mu)\leq \Phi(t,\mu)+\sigma t+\frac{\sigma}{T-t}+ 4L(L+\delta)\eps  + \delta(4 T+2 |\mu|_\cH + 2C_1R+ 2 (L+\delta)\eps) .
\end{align*}
First sending $\delta\to0$, then $\eps,\sigma\to0$, and finally $R\to\infty$, by the above and the definition of $\Phi$, we obtain $u(t,\mu)-v(t,\mu)\leq 0$ as desired.
The case $\bar t'=0$ is similar.
\end{proof}

\subsection{Convergence of approximations}\label{s.HJ_inft_d_cvg}

Recall our notation, for Hamilton--Jacobi equations on different cones with different nonlinearities, given above Definition~\ref{d.vs}.
Throughout, for a directed subcollection $\tilde\J\subset \J$, and a metric space $\cX$, we say that a net $(g_j)_{j\in \tilde\J}$, consisting of $g_j:\cX\to\R$, converges in the local uniform topology to some $g:\cX\to \R$, if $(g_j)_{j\in \tilde\J}$ converges uniformly to $g$ on any closed metric ball with a finite radius.

\begin{proposition}[Limit of approximations is a solution]\label{p.lim_is_vis_sol}
Suppose that $\H$ is continuous.
For each $j\in\Jgen$, let $f_j$ be a viscosity subsolution (respectively, supersolution) of $\HJ(\cH^j,\cM^j,\H^j)$.
If $f=\lim_{j\in \Jgen}f_j^\uparrow$ in the local uniform topology, then $f$ is a viscosity subsolution (respectively, supersolution) of $\HJ(\cH,\C,\H)$. 
\end{proposition}

\begin{proof}
Suppose that $\{f_j\}_{j\in\Jgen}$ is a collection of viscosity subsolutions.
Let us assume that $f-\phi$ achieves a local maximum at $(t,\mu)\in(0,\infty)\times\cM$ for some smooth function $\phi$. We define
\begin{align*}\widetilde\phi(s,\nu)= \phi(s,\nu)+|s-t|^2+|\nu-\mu|^2_\cH,\quad\forall (s,\nu)\in\R_+\times\C.
\end{align*}
Then, there is some $R>0$ such that
\begin{align}\label{e.f-tilde_phi}
    f(s,\nu)-\widetilde \phi(s,v)  = f(t,\mu)-\tilde\phi(t,\mu)- |(s,\nu) - (t,\mu)|^2_{\R\times\cH},\quad\forall (s,\nu)\in B
\end{align}
where
\begin{align*}
    B=\{(s,\nu)\in (0,\infty)\times \C: |(s,\nu)-(t,\mu)|_{\R\times\cH}\leq 2R\}.
\end{align*}
Note that $f-\tilde\phi$ achieves a local maximum at $(t,\mu)$ and that the derivatives of $\tilde\phi $ coincide with those of $\phi$ at $(t,\mu)$. For lighter notation, we redefine $\phi $ to be $\tilde\phi$ henceforth.
It is also clear from Definition~\ref{d.differentiability}~\eqref{i.phi_smooth} that $\phi$ is locally Lipschitz. Hence, there is $L>0$ such that
\begin{align}\label{e.phi_loc_L}
    |\phi(s,\nu)-\phi(s',\nu')|\leq L|(s,\nu)-(s',\nu')|_{\R\times \cH},\quad\forall (s,\nu),\,(s',\nu') \in B.
\end{align}

For each $j\in \Jgen$, we set
\begin{align*}
    B_j = \{(s,y)\in(0,\infty)\times \C^j:|(s,y)- (t,\pj_j \mu)|_{\R\times\cH^j}\leq R\}.
\end{align*}
By making $2R<|t|$ sufficiently small, we can ensure that both $B$ and $B_j$ are closed. Let $(t_j,x_j)\in B_j$ be the point at which $f_j-\phi^j$ achieves the maximum over $B_j$. Here, $\phi^j$ is the $j$-projection of $\phi$ given in Definition~\ref{d.lift_proj}.

For any $\delta\in(0,1)$, we choose $j'\in \Jgen$ such that, for all $j\in\Jgen$ satisfying $j\supset j'$,
\begin{gather}
    \sup_{B}\left|f^\uparrow_j - f\right|<\frac{\delta^2}{4}, \label{e.sup|f^uparrow_j-f|}
    \\
    \left|\mu - \mu^\j\right|_\cH< \min\left\{R ,\frac{\delta^2}{4L}\right\}. \label{e.|mu-mu^j|<R...}
\end{gather}
We claim that, for all $j\in\Jgen$ satisfying $j\supset j'$,
\begin{align}\label{e.(t,x)-(t,mu)}
    \left|(t_j, \lf_j x_j) - (t,\mu)\right|_{\R\times\cH} <\delta.
\end{align}
We argue by contradiction and suppose that there is $j\supset j'$ such that
\begin{align}\label{e.contradiction}
    \left|(t_j, \lf_j x_j) - (t,\mu)\right|_{\R\times\cH} \geq \delta.
\end{align}
Before proceeding, we note that
\begin{align}\label{e.(t_j,l_jx_j)-...}
    \left|(t_j, \lf_j x_j) - (t,\mu)\right|_{\R\times\cH}\leq \left|(t_j, \lf_j x_j) - \left(t,\mu^\j\right)\right|_{\R\times\cH} + \left| \mu - \mu^\j\right|_\cH \leq 2R
\end{align}
where in the last inequality we used \eqref{e.|mu-mu^j|<R...}, and the fact that $(t_j,  x_j)\in B_j$ together with Lemma~\ref{l.basics_(j)}~\eqref{i.isometric} and~\eqref{i.lf_pj_(j)}.
Then, we have
\begin{align*}
    f_j(t_j,x_j)-\phi^j(t_j,x_j) & = f^\uparrow_j(t_j,\lf_j x_j) - \phi(t_j, \lf_j x_j)\\
    &\leq f(t_j, \lf_j x_j) - \phi(t_j, \lf_j x_j) + \frac{\delta^2}{4} 
    \\
    & \leq f(t,\mu) - \phi(t,\mu) - \frac{3\delta^2}{4}\\ 
    &\leq f^\uparrow_j(t,\mu) -\phi(t,\mu)-\frac{\delta^2}{2}
    \\
    & \leq f^\uparrow_j\left(t,\mu^\j\right) -\phi\left(t,\mu^\j\right)-\frac{\delta^2}{4}\\
    &= f_j\left(t,\pj_j\mu\right) -\phi^j\left(t,\pj_j\mu\right)-\frac{\delta^2}{4}
\end{align*}
where the first and the last equalities follow from the definitions of lifts and projections of functions in Definition~\ref{d.lift_proj} together with Lemma~\ref{l.basics_(j)}~\eqref{i.pj_lf=ID} and \eqref{i.lf_pj_(j)}; the first and third inequalities follow from \eqref{e.sup|f^uparrow_j-f|} and $(t_j,\lf_jx_j)\in B$ due to \eqref{e.(t_j,l_jx_j)-...}; the second inequality follows from \eqref{e.contradiction} and \eqref{e.f-tilde_phi}; the fourth inequality follows from the observation that $f^\uparrow_j(t,\mu) = f^\uparrow_j(t,\mu^\j)$ due to the definition of lifts of functions and Lemma~\ref{l.basics_(j)}~\eqref{i.lf_pj_(j)}, and \eqref{e.phi_loc_L} along with \eqref{e.|mu-mu^j|<R...}. The relation in the above display contradicts the fact the maximality of $f_j-\phi_j$ over $B_j$ at $(t_j,x_j)$. Hence, we must have \eqref{e.(t,x)-(t,mu)} and thus
\begin{align}\label{e.(t_j,l_jx_j)_cvg}
    \lim_{j\in\Jgen} (t_j,\lf_jx_j) = (t,\mu)\quad \text{in }(0,\infty)\times \C.
\end{align}

Using \eqref{e.(t_j,l_jx_j)_cvg} and Lemma~\ref{l.basics_(j)}~\eqref{i.pj_lf=ID} and~\eqref{i.|iota^j|<|iota|}, we also have that $\lim_{j\in\Jgen} |(t_j,x_j)- (t,\pj_j\mu)|_{\R\times\cH^j}=0$. Hence, we deduce that, for sufficiently fine $j\in\Jgen$, $(t_j,x_j)$ lies in the interior of $B_j$ relative to $(0,\infty)\times \C^j$. Since $f_j$ is a viscosity subsolution, we get that \begin{align}\label{e.phi^j_HJ_new}
    \left(\partial_t \phi^j - \H^j\left(\nabla_j \phi^j\right)\right)(t_j, x_j) \leq 0.
\end{align}
Using the definition of projections of functions, Lemma~\ref{l.derivatives}~\eqref{i.char_nabla_j}, and Lemma~\ref{l.basics_(j)}~\eqref{i.lf_pj_(j)}, we have that
\begin{gather*}
    \partial_t \phi^j(t, x_j) = \partial_t \phi(t, \lf_j x_j),\qquad \nabla_j \phi^j(t_j,x_j) = \pj_j\left(\nabla \phi(t_j, \lf_j x_j)\right),
    \\
    \H^j\left(\nabla_j\phi^j(t_j,x_j)\right) = \H\left(\left(\nabla\phi(t_j,\lf_jx_j)\right)^\j\right).
\end{gather*}
Then, using \eqref{e.(t_j,l_jx_j)_cvg}, the continuity of differentials (see Definition~\ref{d.differentiability}~\eqref{i.phi_smooth}), and Lemma~\ref{l.basics_(j)}~\eqref{i.cvg_iota^(j)_j}, we can pass \eqref{e.phi^j_HJ_new} to the limit to obtain that
\begin{align*}
    \left(\partial_t \phi - \H\left(\nabla \phi\right)\right)(t, \mu) \leq 0.
\end{align*}
Hence, we have verified that $f$ is a viscosity subsolution.
The same argument also works for viscosity supersolutions.
\end{proof}

Recall that $\Junif$ is the collection of uniform partitions of $[0,1)$, which is generating in the sense given in Section~\ref{s.partition}.
\begin{proposition}[Convergence of approximations]\label{p.exist_sol_approx}
Suppose that $\H:\cH\to\R$ is locally Lipschitz and $\C^*$-increasing and that $\psi:\C\to\R$ is $\C^*$-increasing and satisfies, for some $C>0$ and $p\in[1,2)$,
\begin{align}\label{e.psi_Lip_Lp}
    |\psi(\mu)-\psi(\nu)|\leq C|\mu-\nu|_{L^p}.
\end{align}
For every $j\in\Jgood$, let $f_j$ be the viscosity solution of $\HJ(\cH^j,\C^j,\H^j;\psi^j)$ given by Theorem~\ref{t.hj_cone}~\eqref{i.main_equiv}.
Then, $(f^\uparrow_j)_{j\in \Jgood}$ converges in the local uniform topology to a Lipschitz function $f:\R_+\times \C\to\R$ satisfying $f(0,\cdot)=\psi$,
\begin{gather}
    \sup_{t\in\R_+}\|f(t,\cdot)\|_\mathrm{Lip} \leq \|\psi\|_\mathrm{Lip}, \label{e.f(t,)_lip_infty}
    \\
    \sup_{\mu\in \C} \|f(\cdot, \mu)\|_\mathrm{Lip}\leq \sup_{\substack{\iota\in\cH \\ |\iota|\leq \|\psi\|_\mathrm{Lip}}}|\H(\iota)|. \label{e.f(,mu)_lip_infty}
\end{gather}
\end{proposition}
To prove this result, we follow the proof of \cite[Proposition~3.7]{mourrat2020nonconvex}. First, we show that the lift of a solution is still a solution.

\begin{lemma}\label{l.f_jtoj'}
Suppose that $\H:\cH\to\R$ is $\C^*$-increasing.
Let $j,j'\in\J$ satisfy $j\subset j'$. If $f_j$ is a viscosity subsolution (respectively, supersolution) of $\HJ(\cH^j,\C^j,\H^j)$, then the function defined by
\begin{align}\label{e.f_jtoj'}
    f_{j\to j'}(t,x) = f_j(t, \pj_j\lf_{j'} x),\quad\forall (t,x)\in\R_+\times (\pj_j\lf_{j'})^{-1}(\C^j)
\end{align}
is a viscosity subsolution (respectively, supersolution) of $\HJ(\cH^{j'},(\pj_j\lf_{j'})^{-1}(\C^j),\H^{j'})$.
\end{lemma}

Let us mention and fix a small inaccuracy in \cite{mourrat2020nonconvex}. In the proof \cite[Proposition~3.7]{mourrat2020nonconvex} (see \textit{Step~1} therein), it was claimed that, rephrased in the notation here, the lift of a solution of $\HJ(\cH^j,\C^j,\H^j)$ solves $\HJ(\cH^{j'},\C^{j'},\H^{j'})$ for $j'\supset j$. There, a Dirichlet-type boundary condition was imposed in the definition of viscosity solutions. But, on the boundary of $\C^{j'}$, the lift does not satisfy the condition. This can be fixed using Theorem~\ref{t.hj_cone}~\eqref{i.main_equiv} to see that the boundary is not relevant.

\begin{proof}[Proof of Lemma~\ref{l.f_jtoj'}]
Setting $\tilde \C = (\pj_j\lf_{j'})^{-1}(\C^j)$ for convenience, we suppose that $f_{j\to j'}-\phi$ has a local maximum at $(t,x)\in(0,\infty)\times \tilde\C$ for some smooth function $\phi$. We define
\begin{align*}
    \phi_j(s,y) = \phi(s,x+\pj_{j'}\lf_j y - \pj_{j'}\lf_j \pj_j\lf_{j'} x),\quad\forall (s,y)\in\R_+\times \C^j.
\end{align*}
Using Lemma~\ref{l.basics_(j)}~\eqref{i.projective} and~\eqref{i.pj_lf=ID}, we can show that
\begin{gather}
    \pj_j\lf_{j'}\left(x+\pj_{j'}\lf_j y - \pj_{j'}\lf_j \pj_j\lf_{j'} x\right) =  \pj_j\lf_j y = y\in\C^j, \label{e.pjlfj'(x...)}
\end{gather}
for every $y\in\C^j$, which implies that $x+\pj_{j'}\lf_j y - \pj_{j'}\lf_j \pj_j\lf_{j'} x\in\tilde\C$ for every $y\in\C^j$. Hence, $\phi_j$ is well-defined.

Setting $\bar y = \pj_j\lf_{j'} x$, we want to show that $f_j-\phi_j$ achieves a local maximum at $(t,\bar y)$. Let us fix some $r>0$ sufficiently small such that
\begin{align}\label{e.sup_B_j'}
    \sup_{B_{j'}} (f_{j\to j'}-\phi) =  f_{j\to j'}(t,x)-\phi(t,x)
\end{align}
where
\begin{align*}
    B_{j'}=\left\{(s,z)\in(0,\infty)\times\widetilde\C:|s-t|+|z-x|_{\cH^{j'}}\leq r\right\}.
\end{align*}
Then, we set $B_j = \{(s,y)\in (0,\infty)\times\C^j:|s-t|+|y-\bar y|_{\cH^j}\leq r\}$.
Using Lemma~\ref{l.basics_(j)}~\eqref{i.isometric} and~\eqref{i.|iota^j|<|iota|}, we have that
\begin{align*}
    \left| \pj_{j'}\lf_j y - \pj_{j'}\lf_j \bar y \right|_{\cH^{j'}}\leq |y-\bar y|_{\cH^j}, \quad\forall y \in \C^j,
\end{align*}
which along with \eqref{e.pjlfj'(x...)} implies that
\begin{align*}
    \left(s, x+\pj_{j'}\lf_j y - \pj_{j'}\lf_j \bar y\right)\in B_{j'},\quad\forall (s,y)\in B_j.
\end{align*}
Using \eqref{e.pjlfj'(x...)}, the definition of $f_{j\to j'}$ in \eqref{e.f_jtoj'}, and the definition of $\phi_j$, we also have that
\begin{align*}
    f_j(s,y)-\phi_j(s,y)=f_{j\to j'}(s,x+\pj_{j'}\lf_j y - \pj_{j'}\lf_j \bar y)-\phi(s,x+\pj_{j'}\lf_j y - \pj_{j'}\lf_j \bar y),
    \quad\forall (s,y)\in\R_+\times\C^j.
\end{align*}
Using this, the previous display, and \eqref{e.sup_B_j'}, we obtain that
\begin{align*}
    \sup_{B_j}(f_j-\phi_j)\leq \sup_{B_{j'}} (f_{j\to j'}-\phi)=f_{j\to j'}(t,x)-\phi(t,x)=f_j(t,\bar y)-\phi_j(t,\bar y),
\end{align*}
which implies that $f_j-\phi_j$ achieves a local maximum at $(t,\bar y)$.

Since $f_j$ is a viscosity subsolution, we have
\begin{align*}
    \left(\partial_t\phi_j-\H^j(\nabla_j\phi_j)\right)(t,\bar y)\leq 0.
\end{align*}
Using the definition of $\phi_j$, we can compute that, for any $h\in\cH^j$ sufficiently small,
\begin{align*}
    \la h, \nabla_j\phi_j(s,y)\ra_{\cH^j}+o\left(|h|_{\cH^j}\right)&=\phi_j(s,y+h)-\phi_j(s,y)
    \\
    &= \la \pj_{j'}\lf_j h,\, \nabla_{j'}\phi(\cdots)\ra_{\cH^{j'}}+o\left(|\pj_{j'}\lf_j h|_{\cH^{j'}}\right)
    \\
    &=\la h,\,\pj_j\lf_{j'}\nabla_{j'}\phi(\cdots)\ra_{\cH^j}+o\left(|h|_{\cH^j}\right), \quad\forall (s,y)\in\C^j,
\end{align*}
where in $(\cdots)$ we omitted $(s,x+\pj_{j'}\lf_j y-\pj_{j'}\lf_j\pj_j\lf_{j'}x)$, and, in the last equality, we used Lemma~\ref{l.basics_(j)}~\eqref{i.adjoint} and \eqref{i.projective} to get the term in the bracket and Lemma~\ref{l.basics_(j)}~\eqref{i.isometric} and \eqref{i.|iota^j|<|iota|} for the error term. The above display implies that $\nabla_j\phi_j(t,\bar y) = \pj_j\lf_{j'}\nabla_{j'}\phi(t,x)$. It is easy to see $\partial_t\phi_j(t,\bar y) = \partial_t\phi(t,x)$. These along with the previous display and the definition of $\H^j$ yield
\begin{align*}
    \left(\partial_t\phi-\H(\lf_j\pj_j\lf_{j'}\nabla_{j'}\phi)\right)(t,x)\leq 0.
\end{align*}
We claim that
\begin{align}\label{e.lf_nabla_lpl_nabla}
    \lf_{j'}\nabla_{j'}\phi(t,x)-\lf_j\pj_j\lf_{j'}\nabla_{j'}\phi(t,x) \in \C^*.
\end{align}
Since $\H$ is $\C^*$-increasing, recalling that $\H^{j'}=\H(\lf_{j'}(\cdot))$, we deduce from \eqref{e.lf_nabla_lpl_nabla} and the previous display that
\begin{align*}
    \left(\partial_t\phi-\H^{j'}(\nabla_{j'}\phi)\right)(t,x)\leq 0,
\end{align*}
verifying that $f_{j\to j'}$ is a viscosity subsolution of $\HJ(\cH^{j'},\tilde\C,\H^{j'})$. 

To prove \eqref{e.lf_nabla_lpl_nabla}, by the duality of cones, it suffices to show that
\begin{align*}
    \la \iota,\ \lf_{j'}\nabla_{j'}\phi(t,x)-\lf_j\pj_j\lf_{j'}\nabla_{j'}\phi(t,x)\ra_\cH\geq 0, \quad\forall \iota\in\C.
\end{align*}
By Lemma~\ref{l.basics_(j)}~\eqref{i.adjoint}, the above is equivalent to
\begin{align}\label{e.<,nabla>>0}
    \la \pj_{j'}\iota - \pj_{j'}\lf_j\pj_j\iota,\ \nabla_{j'}\phi(t,x)\ra_{\cH^{j'}}\geq 0,\quad\forall \iota\in\C. 
\end{align}
Fix any $\iota\in\C$.
Lemma~\ref{l.basics_(j)}~\eqref{i.projective} yields
\begin{align}\label{e.p_jl_j'z=0}
    \pj_j\lf_{j'}\left(\pj_{j'}\iota - \pj_{j'}\lf_j\pj_j\iota\right) = \pj_j\iota - \pj_j\lf_j\pj_j\iota = 0.
\end{align}
Hence, setting $z= \pj_{j'}\iota - \pj_{j'}\lf_j\pj_j\iota$, we have $z\in\tilde\C$, and thus $\eps z+x\in\tilde\C$ for any $\eps>0$. Since $f_{j\to j'}-\phi$ has a local maximum at $(t,x)$, we can see that, for $\eps>0$ sufficiently small,
\begin{align*}
    \la \eps z,\nabla_{j'}\phi(t,x)\ra_{j'} + o(\eps)=\phi(t,x+\eps z)-\phi(t,x)\geq f_{j\to j'}(t,x+\eps z) - f_{j\to j'}(t,x)
    \\
    =f_j(t, \pj_j\lf_{j'}x + \eps \pj_j\lf_{j'} z) - f_j(t, \pj_j\lf_{j'}x ) = 0
\end{align*}
where the last equality follows from \eqref{e.p_jl_j'z=0} and the definition of $z$. Sending $\eps\to0$, we can verify~\eqref{e.<,nabla>>0} and complete the proof for subsolutions. The argument for supersolutions is the same with inequalities reversed.
\end{proof}

\begin{proof}[Proof of Proposition~\ref{p.exist_sol_approx}]
It can be readily checked that $\H^j$ is locally Lipschitz and $(\C^j)^*$-increasing on $\cH$ (see Lemma~\ref{l.proj_cones}~\eqref{i.lM^j*}) and that $\psi^j$ is Lipschitz and $(\C^j)^*$-increasing. Hence, Theorem~\ref{t.hj_cone} is applicable, which along with its part~\eqref{i.main_equiv} gives the unique solution $f_j$ of $\HJ(\cH^j,\C^j,\H^j;\psi^j)$ for each $j$.

Let $j,j'\subset \Jgood$ satisfy $j\subset j'$, and $f_j$, $f_{j'}$ be the viscosity solutions. We define $f_{j\to j'}$ by \eqref{e.f_jtoj'}.
By Lemma~\ref{l.f_jtoj'}, $f_{j\to j'}$ is a viscosity solution of $\HJ(\cH^{j'},\,(\pj_j\lf_{j'})^{-1}(\C^j),\,\H^{j'};\,\psi^j(\pj_j\lf_{j'}(\cdot)))$. By Lemma~\ref{l.proj_cones}~\eqref{i.pjM} and~\eqref{i.lM^j}, we have
\begin{align}\label{e.pjljC_subset_C}
    \C^{j'}\subset (\pj_j\lf_{j'})^{-1}(\C^j).
\end{align}
Throughout this proof, we denote by $C$ an absolute constant, which may vary from instance to instance.
We claim that there is $C>0$ such that
\begin{align}\label{e.f_j-h'-f_j'}
    |f_{j\to j'}(t,x)-f_{j'}(t,x)|\leq C|j|^{-\frac{2-p}{2p}}\left(t+|x|_{\cH^{j'}}\right),\quad\forall (t,x)\in \R_+\times \C^{j'}.
\end{align}
Let us use this to derive the desired results. For $\mu\in\C$, we set $x= \pj_{j'}\mu$. Lemma~\ref{l.basics_(j)}~~\eqref{i.projective} implies that $\pj_j \lf_{j'} x = \pj_j \mu$. Hence, by definitions, we have
\begin{gather*}
    f^\uparrow_j(t,\mu) = f_j(t,\pj_j \mu) = f_{j\to j'}(t,x)
\end{gather*}
and $f^\uparrow_{j'}(t,\mu) = f_{j'}(t, x)$. Now using~\eqref{e.f_j-h'-f_j'} and Lemma~\ref{l.basics_(j)}~~\eqref{i.|iota^j|<|iota|}, we have
\begin{align*}
    \left|f^\uparrow_j(t,\mu)-f^\uparrow_{j'}(t,\mu)\right|\leq C|j|^{-\frac{2-p}{2p}}\left(t+|\mu|_{\cH}\right).
\end{align*}
We could now conclude the existence of a limit $f(t,\mu)$ by arguing that the above together with the triangle inequality yields that $(f^\uparrow_j(t,\mu))_{j\in\Jgood}$ is a \textit{Cauchy net} in $\R$ (see \cite[Definition~2.1.41]{megginson2012introduction} and \cite[Proposition~2.1.49]{megginson2012introduction}). Denoting the pointwise limit by $f$, and passing $j'$ to limit in the above display to see that $f^\uparrow_j$ converges in the local uniform topology to some $f:\R_+\times\C\to\R$. By Lemma~\ref{l.basics_(j)}~\eqref{i.cvg}, it is straightforward to see $f(0,\cdot)=\psi$.

Then, we show~\eqref{e.f(t,)_lip_infty} and~\eqref{e.f(,mu)_lip_infty}. By \eqref{e.psi_Lip_Lp} and H\"older's inequality, we have $\|\psi\|_\mathrm{Lip}<C$. Theorem~\ref{t.hj_cone}~\eqref{i.main_lip} implies that, for every $j$,
\begin{gather}\label{e.f_lip_finite_d}
    \sup_{t\in\R_+}\|f_j(t,\cdot)\|_\mathrm{Lip} = \|\psi^j\|_\mathrm{Lip},\qquad
    \sup_{x\in \C^j} \|f_j(\cdot, x)\|_\mathrm{Lip}\leq \sup_{\substack{p\in\cH^j \\ |p|_{\cH^j}\leq \|\psi^j\|_\mathrm{Lip}}}|\H^j(p)|.
\end{gather}
By the definition of $\psi^j$ and Lemma~\ref{l.basics_(j)}~~\eqref{i.isometric}, we can see that, for every $x,y\in \C^j$,
\begin{align*}
    |\psi^j(x)-\psi^j(y)|= |\psi(\lf_j x)-\psi(\lf_j y)|\leq \|\psi\|_\mathrm{Lip}|\lf_j x - \lf_j y|_{\cH} = \|\psi\|_\mathrm{Lip}|x - y|_{\cH^j},
\end{align*}
which implies that
\begin{align}\label{e.psi^j_Lip<psi_Lip}
    \|\psi^j\|_\mathrm{Lip}\leq \|\psi\|_\mathrm{Lip},\quad\forall j\in\J.
\end{align}
Using this, the first result in~\eqref{e.f_lip_finite_d} and Lemma~\ref{l.basics_(j)}~~\eqref{i.|iota^j|<|iota|}, we have, for every $t\in\R_+$ and every $\mu,\nu\in\C$,
\begin{align*}
    |f^\uparrow_j(t,\mu)- f^\uparrow_j(t,\nu)|= |f_j(t,\pj_j\mu)-f_j(t,\pj_j\nu)|\leq \|\psi^j\|_\mathrm{Lip}|\pj_j\mu - \pj_j\nu|_{\cH^j}\leq \|\psi\|_\mathrm{Lip}|\mu - \nu|_{\cH},
\end{align*}
yielding~\eqref{e.f(t,)_lip_infty} after passing $j$ to the limit. To see~\eqref{e.f(,mu)_lip_infty}, for every $p$ satisfying the condition under supremum in the second result in~\eqref{e.f_lip_finite_d}, we have, by Lemma~\ref{l.basics_(j)}~~\eqref{i.isometric}, that
\begin{align*}
    |\lf_j p|_\cH = |p|_{\cH^j}\leq \|\psi^j\|_\mathrm{Lip} \leq \|\psi\|_\mathrm{Lip}.
\end{align*}
Since $\H^j(p) = \H(\lf_j p)$ by definition, the right-hand side of the second result in~\eqref{e.f_lip_finite_d} is thus bounded by the right-hand side of~\eqref{e.f(,mu)_lip_infty}. Passing $j$ to the limit, we can verify~\eqref{e.f(,mu)_lip_infty}. Hence, the proof is complete modulo~\eqref{e.f_j-h'-f_j'}
\end{proof}

\begin{proof}[Proof of~\eqref{e.f_j-h'-f_j'}]
Due to~\eqref{e.f_lip_finite_d} and~\eqref{e.psi^j_Lip<psi_Lip}, we have
\begin{align}\label{e.f_j,f_j'_Lip}
    \sup_{t\in\R_+}\|f_j(t,\cdot)\|_\mathrm{Lip},\quad \sup_{t\in\R_+}\|f_{ j'}(t,\cdot)\|_\mathrm{Lip} \leq \|\psi\|_\mathrm{Lip}.
\end{align}
The definition of $f_{j\to j'}$ in \eqref{e.f_jtoj'} implies
\begin{align*}
    |f_{j\to j'}(t,x)&- f_{j\to j'}(t,y)|  = |f_j(t,\pj_j\lf_{j'}x)-f_j(t,\pj_j\lf_{j'}y)|
    \\
    &\leq \|\psi\|_\mathrm{Lip}\left|\pj_j\lf_{j'}x-\pj_j\lf_{j'}y\right|_{\cH^j}\leq \|\psi\|_\mathrm{Lip} |x-y|_{\cH^{j'}},\quad\forall t\geq 0,\ \forall x,y\in (\pj_j\lf_{j'})^{-1}(\C^j),
\end{align*}
where we used Lemma~\ref{l.basics_(j)}~\eqref{i.isometric} and~\eqref{i.|iota^j|<|iota|} to derive the last inequality. Hence, we have 
\begin{align}\label{e.f_jtoj'_Lip}
    \sup_{\R_+}\|f_{j\to j'}(t,\cdot)\|_\mathrm{Lip}\leq \|\psi\|_\mathrm{Lip}.
\end{align}
Using \eqref{e.pjljC_subset_C} and Proposition~\ref{p.comp_fin_d_strong} with $M$ replaced by $2\|\psi\|_\mathrm{Lip}+1$ and $R>1$ to be determined, we have that
\begin{align}
    \sup_{(t,x)\in\R_+\times \C^{j'}} f_{j\to j'}(t,x) - f_{j'}(t,x)- M(|x|_{\cH^{j'}}+Vt-R)_+ \notag
    \\
    = \sup_{x\in \C^{j'}} f_{j\to j'}(0,x) - f_{j'}(0,x)- M(|x|_{\cH^{j'}}-R)_+. \label{e.sup_f-f_at_t=0}
\end{align}

The term inside the supremum on the right-hand side of~\eqref{e.sup_f-f_at_t=0} can be rewritten as
\begin{align*}
    \psi\left((\lf_{j'}x)^\j\right)-\psi(\lf_{j'}x)- M(|x|_{\cH^{j'}}-R)_+,
\end{align*}
where we used the definition of $f_{j\to j'}$ in \eqref{e.f_jtoj'} and Lemma~\ref{l.basics_(j)}~\eqref{i.lf_pj_(j)}.
By~\eqref{e.psi_Lip_Lp} and H\"older's inequality, we have
\begin{align}
    \left|\psi\left((\lf_{j'}x)^\j\right)-\psi(\lf_{j'}x)\right|\leq C\left|(\lf_{j'}x)^\j - \lf_{j'}x\right|^\frac{2-p}{p}_{L^1}\left|(\lf_{j'}x)^\j -\lf_{j'}x\right|_\cH^\frac{2p-2}{p} \notag
    \\
    \leq C\left|(\lf_{j'}x)^\j - \lf_{j'}x\right|^\frac{2-p}{p}_{L^1}\left|x\right|_{\cH^{j'}}^\frac{2p-2}{p} \label{e.psi_bound}
\end{align}
where we used Lemma~\ref{l.basics_(j)}~\eqref{i.isometric} and~\eqref{i.|iota^j|<|iota|} in the last inequality. Setting $J=|j|$ and $J'=|j'|$, due to $j'\supset j$ and $j,j'\in\Jgood\subset\Junif$, we know that there is $N\in \N$ such that $J'= JN$. 
Before estimating the $L^1$ norm, we remark that it suffices to assume $D=1$, namely, $\lf_{j'}x(s)\in \R_+$ for each $s\in[0,1)$. Indeed, if $D>1$, we can reduce the problem to the real-valued case by considering
\begin{align*}
    s\mapsto I_D \cdot\lf_{j'}x(s)
\end{align*}
where $I_D$ is the $D\times D$ identity matrix. This reduction is valid due to $C^{-1}_KI_D\cdot a \leq |a|\leq C_D I_D\cdot a$ for every $a\in\S^D_+$ and some constant $C_D>0$. With this simplification clarified, we assume $D=1$. Writing $j'=(t_1,t_2,\dots,t_{J'})$ with $t_k = \frac{k}{J'}$ and $j=(s_1,\dots, s_J)$ with $s_m=\frac{m}{J}$, we can compute that
\begin{align}
    \left| \lf_{j'}x-(\lf_{j'}x)^\j\right|_{L^1}
    =\sum_{m=1}^J\sum_{k:s_{m-1}<t_k\leq s_m}(t_k-t_{k-1})\left|x_k-\frac{1}{s_m-s_{m-1}}\sum_{k':s_{m-1}<t_{k'}\leq s_m }(t_{k'}-t_{k'-1})x_{k'}\right| \notag
    \\
    = \sum_{m=1}^J\sum_{k=N(m-1)+1}^{Nm}\frac{1}{JN}\left|x_k - \frac{1}{N}\sum_{k'=N(m-1)+1}^{Nm}x_{k'}\right| 
    \leq \frac{1}{JN^2}\sum_{m=1}^J\sum_{k=N(m-1)+1}^{Nm}\sum_{k'=N(m-1)+1}^{Nm}|x_k-x_{k'}| \notag
    \\
    = \frac{2}{JN^2}\sum_{m=1}^J\sum_{k,k':N(m-1)<k'<k\leq Nm}|x_k-x_{k'}|. \label{e.L^1_est}
\end{align}
Let $B>0$ be chosen later. Since $x_{k}\geq x_{k'}\geq 0$ for $k>k'$ due to $x\in\C^{j'}$, we have
\begin{align}
    \frac{2}{JN^2}\sum_{m=1}^J\sum_{k,k':N(m-1)<k'<k\leq Nm}|x_k-x_{k'}|\mathds{1}_{|x_k|\geq B} \leq \frac{2}{JN^2}\sum_{m=1}^J\sum_{k,k':N(m-1)<k'<k\leq Nm}|x_k|\mathds{1}_{|x_k|\geq B}\notag
    \\
    \leq \frac{2}{JN}\sum_{m=1}^J\sum_{k=N(m-1)+1}^{Nm}\frac{|x_k|^2}{B}=\frac{2}{B}\sum_{k=1}^{J'}\frac{1}{J'}|x_k|^2 =\frac{2}{B}|x|^2_{\cH^{j'}}.\label{e.cond_on_|x_k|>B}
\end{align}
On the other hand, switching summations, we have
\begin{align*}
    &\frac{2}{JN^2}\sum_{m=1}^J\sum_{k,k':N(m-1)<k'<k\leq Nm}|x_k-x_{k'}|\mathds{1}_{|x_k|\leq B} 
    \\
    &= \frac{2}{JN^2}\sum_{r,r':0<r'<r\leq N}\sum_{m=1}^J|x_{N(m-1)+r}-x_{N(m-1)+r'}|\mathds{1}_{|x_{N(m-1)+r}|\leq B}
\end{align*}
Again using $x_k\geq x_{k'}\geq 0$ for $k>k'$ and setting $m^*=\max\{m\in\{1,\dots,J\}:x_{N(m-1)+r}\leq B\}$, we can see that
\begin{align*}
    \sum_{m=1}^J|x_{N(m-1)+r}-x_{N(m-1)+r'}|\mathds{1}_{|x_{N(m-1)+r}|\leq B}= \sum_{m=1}^{m^*}(x_{N(m-1)+r}-x_{N(m-1)+r'})\mathds{1}_{x_{N(m-1)+r}\leq B}
    \\
    \leq x_{N(m^*-1)+r}\mathds{1}_{x_{N(m*-1)+r}\leq B}\leq B.
\end{align*}
Here in the penultimate inequality, we also used the fact that $-x_{N(m-1)+r'}+x_{N(m-2)+r}\leq 0$ because $N(m-1)+r'> N(m-2)+r$ due to $|r-r'|<N$. Therefore,
\begin{align*}
    \frac{2}{JN^2}\sum_{m=1}^J\sum_{k,k':N(m-1)<k'<k\leq Nm}|x_k-x_{k'}|\mathds{1}_{|x_k|\leq B}  \leq \frac{B}{J}.
\end{align*}
Inserting into~\eqref{e.L^1_est} the above estimate combined with~\eqref{e.cond_on_|x_k|>B}, and choosing $B= \sqrt{J}|x|_{\cH^{j'}}$, we conclude that
\begin{align*}
    \left|(\lf_{j'}x)^\j - \lf_{j'}x\right|_{L^1}\leq 3J^{-\frac{1}{2}}|x|_{\cH^{j'}}.
\end{align*}
Plugging this into~\eqref{e.psi_bound} yields
\begin{align*}
    f_{j\to j'}(0,x) - f_{j'}(0,x) - M(|x|_{\cH^{j'}}-R)_+ \leq CJ^{-\frac{2-p}{2p}}|x|_{\cH^{j'}},\quad\forall x\in\C^{j'}.
\end{align*}
Due to $f_{j\to j'}(0,0) = f_{j'}(0,0)=\psi(0)$, \eqref{e.f_j,f_j'_Lip}, and \eqref{e.f_jtoj'_Lip}, the choice of $M=\|\psi\|_\mathrm{Lip}+1$ ensures that
\begin{align*}
    f_{j\to j'}(0,x) - f_{j'}(0,x) - M(|x|_{\cH^{j'}}-R)_+ \leq 2\|\psi\|_\mathrm{Lip}|x|_{\cH^{j'}}-M|x|_{\cH^{j'}}+MR= MR- |x|_{\cH^{j'}},\quad\forall x\in\C^{j'}.
\end{align*}
These two estimates imply that the left-hand side of them is bounded by $CJ^{-\frac{2-p}{2p}}MR$. Absorbing $M$ into $C$ and using~\eqref{e.sup_f-f_at_t=0}, we arrive at \begin{align*}
    \sup_{(t,x)\in\R_+\times \C^{j'}} f_{j\to j'}(t,x) - f_{j'}(t,x)- M(|x|_{\cH^{j'}}+Vt-R)_+ \leq CJ^{-\frac{2-p}{2p}}R.
\end{align*}
Replacing $R$ by $|x|_{\cH^{j'}}+Vt$ for each $(t,x)\in\R_+\times \C^{j'}$, we obtain one bound for~\eqref{e.f_j-h'-f_j'}. 

For the opposite bound, we again use \eqref{e.pjljC_subset_C} and Proposition~\ref{p.comp_fin_d_strong} to get a result as in \eqref{e.sup_f-f_at_t=0} with $f_{j\to j'}$ and $f_{j'}$ swapped. Then, the same arguments as above give the other bound to complete the proof of \eqref{e.f_j-h'-f_j'}.
\end{proof}

\subsection{Limits of variational formulas}\label{s.HJ_inft_d_var_form}

Below, $f_j$ and $f$ are not assumed to be solutions. The following two propositions are only about the limits of variational formulas. 

\begin{proposition}[Hopf-Lax formula in the limit]\label{p.hopf_lax_cvg}
Suppose
\begin{itemize}
    \item $\psi:\C\to\R$ is $\C^*$-increasing and continuous;
    \item $\H:\cH\to\R$ satisfies $\H\left(\nu^\j\right)\leq \H(\nu)$ for every $\nu\in\C$ and every $j\in \Jgen$;
    \item for each $j\in\Jgen$, $f_j:\R_+\times\C^j\to(-\infty,\infty]$ is given by
    \begin{align*}f_j(t,x)=\sup_{y \in \C^j}\inf_{z\in\C^j}\left\{\psi^j(x+y)-\la y,z\ra_{\cH^j}+ t \H^j\left(z\right)\right\},\quad\forall(t,x)\in\R_+\times\C^j.
    \end{align*}
\end{itemize}
If $\lim_{j\in\Jgen}f^\uparrow_j(t,\mu)$ exists in $\R$ at some $(t,\mu)\in\R_+\times\C$, then the limit is given by
\begin{align*}f(t,\mu) =\sup_{\nu \in \C}\inf_{\rho\in\C}\left\{\psi(\mu+\nu)- \la \nu,\rho\ra_{\cH}+t\H\left(\rho\right)\right\}.
\end{align*}
\end{proposition}

\begin{proof}
Using Lemma~\ref{l.basics_(j)}~\eqref{i.lf_pj_(j)} and Lemma~\ref{l.proj_cones}~\eqref{i.pjM}, we can rewrite
\begin{align}
    f^\uparrow_j(t,\mu)
    &= \sup_{y \in \C^j}\inf_{z\in\C^j}\left\{\psi^j(\pj_j\mu+  y) - \la y, z \ra_{\cH^j}+ t \H^j (z)\right\}\nonumber
    \\
    &= \sup_{\nu \in \C}\inf_{\rho\in\C}\left\{\psi\left(\mu^\j +  \nu^\j\right)-  \la \nu^\j,\rho \ra_\cH+ t \H \left(\rho^\j\right)\right\}\label{eq.fupj_sup}.
\end{align}
By the assumption on $\H$, we have $\H\left(\rho^\j\right)\leq \H(\rho)$.
Also, $\psi$ is $\C^*$-increasing and Lemma~\ref{l.proj_cones}~\eqref{l.mu-mu^j} yields $\mu-\mu^\j\in\C^*$.
Using these, we obtain
\begin{align*}
    f^\uparrow_j(t,\mu)\leq \sup_{\nu \in \C}\inf_{\rho\in\C}\left\{\psi\left(\mu+  \nu^\j\right)-  \la \nu^\j, \rho \ra_\cH+ t \H (\rho)\right\} \leq f(t,\mu)
\end{align*}
where the last inequality follows from $\{\nu^\j:\nu\in\C\}\subset \C$.
Passing $j$ to the limit, we get $\lim_{j\in\Jgen}f^\uparrow_j(t,\mu)\leq f(t,\mu)$.

For the other direction, fixing any $\eps>0$, we can find $\nu$ to satisfy
\begin{align*}
    f(t,\mu) &\leq \eps + \psi(\mu+\nu) + \inf_{\rho\in\C}\left\{- \la \nu,\rho\ra_{\cH}+t\H\left(\rho\right)\right\}.
\end{align*}
Since $\psi$ is continuous, by Lemma~\ref{l.basics_(j)}~\eqref{i.cvg}, we can find $j'\in\Jgen$ such that, for all $j\supset j'$,
\begin{align*}
    f(t,\mu) &\leq 2\eps + \psi\left(\mu^\j+\nu^\j\right) + \inf_{\rho\in\C}\left\{- \la \nu,\rho^\j\ra_{\cH}+t\H\left(\rho^\j\right)\right\}
\end{align*}
where we also used $\{\rho^\j:\rho\in\C\}\subset\C$ to bound the infimum. 
Due to $\la \nu,\rho^\j\ra_{\cH} = \la \nu^\j,\rho\ra_{\cH}$ and~\eqref{eq.fupj_sup}, we get $f(t,\mu) \leq 2\eps + f^\uparrow_j(t,\mu)$ for all $j\supset j'$.
Passing $j$ to the limit and then sending $\eps\to0$, we obtain the matching bound, which completes the proof.
\end{proof}

\begin{remark}\rm\label{r.hopf_lax_cvg}
In Proposition~\ref{p.hopf_lax_cvg}, if we only change the condition on $\H$ to
\begin{itemize}
    \item $\H:\cH\cap L^\infty\to\R$ satisfies $\H\left(\nu^\j\right)\leq \H(\nu)$ for every $\nu\in\C\cap L^\infty$ and every $j\in \Jgen$,
\end{itemize}
then we can show 
\begin{align*}f(t,\mu) =\sup_{\nu \in \C\cap L^\infty}\inf_{\rho\in\C\cap L^\infty}\left\{\psi(\mu+\nu)- \la \nu,\rho\ra_{\cH}+t\H\left(\rho\right)\right\}.
\end{align*}
The proof is almost verbatim, after one observes $\C^j=\pj_j(\C\cap L^\infty)$ and $\{\mu^\j:\mu\in\C\cap L^\infty\}\subset \C\cap L^\infty$.
\end{remark}

\begin{proposition}[Hopf formula in the limit]\label{p.hopf_cvg}
Suppose
\begin{itemize}
    \item $\psi:\C\to\R$ is $\cM^*$-increasing;
    \item $\H:\cH\to\R$ is continuous;
    \item for each $j\in\Jgen$, $f_j:\R_+\times\C^j\to(-\infty,\infty]$ is given by
    \begin{align*}f_j(t,x)=\sup_{z\in \C^j}\inf_{y\in \C^j}\left\{ \psi^j(y)+ \la x-y,z \ra_{\cH^j}+t\H^j(z)\right\},\quad \forall (t,x)\in\R_+\times \C^j.
    \end{align*}
\end{itemize}
If $\lim_{j\in\Jgen}f^\uparrow_j(t,\mu)$ exists in $\R$ at some $(t,\mu)\in\R_+\times\C$, then the limit is given by
\begin{align*}
    f(t,\mu) = \sup_{\rho \in \cM}\inf_{\nu \in \cM}\left\{ \psi(\nu) + \la \mu - \nu, \rho\ra_\cH +t\H(\rho)\right\}.
\end{align*}

\end{proposition}

\begin{proof}
We can express
\begin{align*}
    f^\uparrow_j(t,\mu) = \sup_{\rho\in \C}\inf_{\nu\in \C}\left\{ \psi\left(\nu^\j\right)+ \la \mu^\j-\nu^\j, \rho^\j\ra_{\cH}+t\H\left(\rho^\j\right)\right\}.
\end{align*}
Fix any $(t,\mu)$. For $\eps>0$, we choose $\rho$ such that
\begin{align*}
    f(t,\mu) &\leq \eps + \inf_{\nu \in \cM}\left\{\psi(\nu)+ \la \mu - \nu,\rho\ra_\cH +t\H(\rho)\right\}
    \\
    &\leq \eps + \inf_{\nu \in \cM}\left\{\psi\left(\nu^\j\right)+ \la \mu - \nu^\j,\rho\ra_\cH +t\H(\rho)\right\}
\end{align*}
for all $j\in \J$, where the last inequality follows from the fact that $\{\nu^\j:\nu\in\cM\}\subset\cM$.
Allowed by the continuity of $\H$ and Lemma~\ref{l.basics_(j)}~\eqref{i.cvg}, we can find $j'\in \J$ such that for all $j\supset j'$, 
\begin{align*}
    \la \mu, \rho \ra_\cH + t\H(\rho)\leq \eps +  \la {\mu^\j}, \rho\ra_\cH + t\H\left(\rho^\j\right).
\end{align*}
Using this and the fact that $\la \iota,\kappa^\j\ra_\cH = \la \iota^\j,\kappa^\j\ra_\cH$ for all $\iota,\kappa\in\cH$, we get
\begin{align*}f(t,\mu) \leq 2\eps + \inf_{\nu\in \C}\left\{\psi\left(\nu^\j\right)+ \la \mu^\j-\nu^\j, \rho^\j\ra_{\cH}+t\H\left(\rho^\j\right)\right\}
    \leq 2\eps+f^\uparrow_j(t,\mu),\quad\forall j\supset j'.
\end{align*}
Passing $j$ to the limit and sending $\eps\to0$, we obtain $f(t,\mu)\leq \lim_{j\in\Jgen}f^\uparrow_j(t,\mu)$.

To see the converse inequality, fixing any $\eps>0$, we choose $\rho_j$, for each $j\in \J$, to satisfy
\begin{align*}
    f^\uparrow_j(t,\mu) \leq \eps+ \inf_{\nu \in \cM}\left\{ \psi\left(\nu^\j\right)+ \la \mu^\j - \nu^\j , \rho^\j_j \ra_\cH +t\H\left(\rho^\j_j\right)\right\}.
\end{align*}
On the other hand, it is clear from the definition of $f(t,\mu)$ that
\begin{align*}
    f(t,\mu) &\geq \inf_{\nu \in \cM}\left\{ \psi(\nu)  + \la \mu - \nu, \rho^\j_j \ra_\cH +t\H\left(\rho^\j_j\right)\right\}
    \\
    &\geq \inf_{\nu \in \cM}\left\{ \psi\left(\nu^\j\right)  + \la \mu^\j - \nu^\j, \rho^\j_j \ra_\cH +t\H\left(\rho^\j_j\right)\right\},\quad\forall j\in\J,
\end{align*}
where in the last inequality we used $\nu-\nu^\j\in\C^*$ (Lemma~\ref{l.proj_cones}~\eqref{l.mu-mu^j}) and that $\psi$ is $\C^*$-increasing.
Hence, we get $f(t,\mu)\geq f^\uparrow_j(t,\mu) -\eps$.
Passing $j$ to the limit along $\Jgen$ and sending $\eps\to0$, we get $f(t,\mu)\geq \lim_{j\in\Jgen}f^\uparrow_j(t,\mu)$, completing the proof.
\end{proof}

\begin{remark}\rm\label{r.hopf_cvg}
In Proposition~\ref{p.hopf_cvg}, if we only change the condition on $\H$ to
\begin{itemize}
    \item $\H:\cH\cap L^\infty\to\R$ is continuous (in the topology of $\cH$),
\end{itemize}
then, by using $\C^j=\pj_j(\C\cap L^\infty)$, we modify the above proof to show 
\begin{align*}
    f(t,\mu) = \sup_{\rho \in \cM\cap L^\infty}\inf_{\nu \in \cM\cap L^\infty}\left\{ \psi(\nu) + \la \mu - \nu, \rho\ra_\cH +t\H(\rho)\right\}.
\end{align*}

\end{remark}

\subsection{Weak boundary}\label{s.HJ_inft_d_weak_boundary}

It can be checked that $\cM$ has an empty interior in $\cH$. Therefore, the boundary of $\cM$ is equal to $\cM$. On the other hand, for each $j\in\J$, the interior of $\cM^j$ is not empty. We denote its boundary by $\partial \cM^j$.

\begin{lemma}[Characterizations of $\partial\C^j$]
\label{l.weak_bd_fin_d}
Let $j\in\J$ and $x \in \cM^j$. Then, the following are equivalent:
\begin{enumerate}
    \item \label{i.weak_bd_1} $x\in\partial \C^j$;
    \item \label{i.weak_bd_2} there is $y\in (\C^j)^*\setminus\{0\}$ such that $\la x,y\ra_{\cH^j}=0$;
    \item \label{i.weak_bd_3} there is $k\in\{1,2,\dots,|j|\}$ such that $x_k=x_{k-1}$.
\end{enumerate}
\end{lemma}

For \eqref{i.weak_bd_3}, recall our convention that $x_0=0$.
\begin{proof}
First, we show that \eqref{i.weak_bd_3} implies \eqref{i.weak_bd_2}. Let $I_D$ be the $D\times D$ identity and matrix. If $k>1$, we set $y_k=\frac{1}{t_k-t_{k-1}} I_D$, $y_{k-1} = -\frac{1}{t_{k-1}-t_{k-2}}I_D$ and $y_i = 0$ for all $i\in\{1,2,\dots,|j|\}\setminus\{k-1,k\}$. If $k=1$, we set $y_1= I_D$ and $y_i=0$ otherwise. By Lemma~\ref{l.dual_cone}~\eqref{i.dual-cone-M^j}, we have $y\in(\C^j)^*$. It is also clear that $y\neq 0$ and $\la x, y\ra_{\cH^j} = 0$, verifying \eqref{i.weak_bd_2}.

Next, we show that \eqref{i.weak_bd_2} implies \eqref{i.weak_bd_1}. Assuming \eqref{i.weak_bd_2}, we suppose that $x$ is in the interior. Then, there is $\eps>0$ sufficiently small such that $x-\eps y\in\C^j$, which implies that $\la x-\eps y,y\ra_{\cH^j}\geq 0$. However, by assumption \eqref{i.weak_bd_2}, we must have $-\eps|y|^2_{\cH^j}\geq 0$ and thus $y=0$, reaching a contradiction.

Finally, we show that \eqref{i.weak_bd_1} implies \eqref{i.weak_bd_3}. Assuming \eqref{i.weak_bd_1}, we suppose that \eqref{i.weak_bd_3} is not true. Since the coordinates of $x$ are increasing, we can find $\delta>0$ such that $x_k \geq \delta I_D + x_{k-1}$ for all $k$. By the finite dimensionality, there is a constant $C>0$ such that
\begin{align*}
    y_k-C\eps I_D\leq x_k\leq y_k+C\eps I_D
\end{align*}
for every $y\in \cH^j$ satisfying $|y-x|_{\cH^j}\leq \eps$, for every $\eps>0$ and every $k\in\{1,,2,\dots,|j|\}$. Choosing $\eps$ sufficiently small, we can see that, for such $y$, we have $y_k\geq y_{k-1}$ for all $k$, namely $y\in\C^j$, which contradicts \eqref{i.weak_bd_1}. 
\end{proof}

The equivalence between \eqref{i.weak_bd_1} and \eqref{i.weak_bd_2} holds for more general cones in finite dimensions.
It is thus natural to define a weak notion of boundary for $\cM$.

\begin{definition}
The weak boundary of $\cM$ denoted by $\wbd\cM$ is defined by
\begin{align*}
    \wbd\cM = \left\{ \mu \in \cM:\exists \iota\in\cM^*\setminus\{0\}, \ \la \mu, \iota\ra_\cH =0 \right\}.
\end{align*}
\end{definition}

When $D=1$, for every $\mu\in\C$, since $\mu$ is increasing, we have that $\mu$ is differentiable a.e.\ and we denote its derivative by $\dot \mu$. If $D>1$, we can choose a basis for $\S^D$ consisting of elements in $\S^D_+$. For each $a$ from the basis, the derivative of $s\mapsto a\cdot\mu(s)$ exists a.e. We can use these to define $\dot \mu$. We define the essential support of an $\S^D$-valued function on $[0,1)$ as the smallest closed set relative to $[0,1)$, outside which the function is zero a.e.

\begin{lemma}[Characterization of $\wbd\C$]\label{l.esssppt}
For $\mu\in\cM$, it holds that $\mu \in \wbd\cM$ if and only if the essential support of $\dot\mu$ is not $[0,1)$.
\end{lemma}
\begin{proof}
Let $\mu\in\cM$. By adding a constant, we may assume $\mu(0)=0$. For any fixed $\iota\in\C^*$, we set $\kappa:[0,1)\to\R$ by $\kappa(t) = \int_t^1\iota (s)\d s$. Then, $\kappa$ is continuous, nonnegative (by Lemma~\ref{l.dual_cone}~\eqref{i.dual-cone-M}), and differentiable with its derivative is given by $-\iota$. Since $\mu(0)=0$ and $\lim_{t\to1}\kappa(t)=0$, by integration by parts, we have that
\begin{align*}
    \la \mu,\iota\ra_\cH = \int_0^1\kappa(s)\dot\mu(s)ds.
\end{align*}
First, suppose that the essential support of $\dot\mu$ is $[0,1)$. Let $\iota$ be nonzero and thus so is $\kappa$. Then, the integral above is positive, and thus $\mu\not\in \wbd\cM$. For the other direction, suppose that the essential support of $\dot \mu$ is a strict subset of $[0,1)$. This implies the existence of a nonempty open set $O\subset [0,1)$ on which $\dot\mu$ vanishes. We then choose a nonnegative and smooth $\kappa$ such that $\kappa>0$ only on a subset of $O$. Setting $\iota = -\dot \kappa$, we clearly have $\iota\in \cM^*\setminus\{0\}$. In this case, the integral in the above display is zero, implying $\mu\in \wbd\cM$.
\end{proof}

It is thus tempting to use $\wbd\C$ as a more suitable notion of boundary. However, it is still not optimal, due to the following immediate consequence of Lemma~\ref{l.esssppt}.

\begin{lemma}
For $j\in\J$, then $\lf_j x \in \wbd\cM$ for every $x \in  \cM^j$; and $\mu^\j\in\wbd\C$ for every $\mu\in \C$.
\end{lemma}

In other words, any point from $\C^j$ is lifted to the boundary of $\C$, no matter whether it is in the interior of $\C^j$ or not. The following lemma could potentially be a remedy.

\begin{lemma}If $x \in \cM^j\setminus\partial \cM^j$, then there is $\mu\in\cM\setminus\wbd \cM$ such that $ \pj_j\mu = x$.
\end{lemma}
\begin{proof}
By the equivalence between \eqref{i.weak_bd_1} and \eqref{i.weak_bd_3} in Lemma~\ref{l.weak_bd_fin_d}, we can find $\delta>0$ such that $x_k-x_{k-1}\geq\delta I_D$ for all $k$, where $I_D$ is the $D\times D$ identity matrix. Then, we define $\mu:[0,1)\to\S^D$ by
\begin{align*}
\mu(s)&=\eps I_D\left(s-\frac{t_k+t_{k-1}}{2}\right)+x_k,\quad\textrm{if }s\in[t_{k-1},t_k),
\end{align*}
for $\eps>0$. It is straightforward to check that $\pj_j \mu =x$. By choosing $\eps>0$ sufficiently small, we can ensure that $\mu$ is strictly increasing on $[0,1)$. Hence, Lemma~\ref{l.esssppt} implies that $\mu\in \C\setminus\wbd\C$.
\end{proof}

The applications of these results are still unclear to us. So, we leave them for future investigations.

\section{Application to the spin glass setting}\label{s.spin-glass}

We start by presenting a general vector spin glass model. We describe the external field indexed by a monotone probability measure. Then, we review the geometry of the space of monotone probability measures, which leads to an isometry between it and $\C$ in~\eqref{e.cone_M}. These parts are not needed in the study of the equation~\eqref{e.HJ_spin_glass} but serve as motivation for our ensuing definition of viscosity solutions of~\eqref{e.HJ_spin_glass}. Then, we move towards establishing the precise version of Theorem~\ref{t.informal}. After that, we show that different notions of solutions of~\eqref{e.HJ_spin_glass} considered in \cite{mourrat2019parisi,mourrat2020extending,mourrat2020nonconvex,mourrat2023free} are in fact viscosity solutions.

\subsection{Mean-field spin glass models}\label{s.spin-glass_setting}
We follow the setting in \cite{mourrat2023free} which encompasses a wide class of mean-field vector spin models. Recall that $D$ is a positive integer. 
For each $N\in\N$, let $\mathfrak{H}_N$ be a finite-dimensional Hilbert space.
Let $P_N$ be a probability measure on $\mathfrak{H}_N^D$ supported on the closed centered ball in $\mathfrak{H}_N^D$ with radius $\sqrt{N}$.
We interpret $\mathfrak{H}_N^D$ as the state space of spin configurations $\sigma = (\sigma_d)_{d=1}^D$, where each $\sigma_d$ lies in $\mathfrak{H}_N$, and $P_N$ as the reference measure to distribute~$\sigma$.

Let $\xi:\R^{D\times D}\to \R$ be locally Lipschitz. For each $N\in\N$, we assume the existence of a centered Gaussian field $(H_N(\sigma))_{\sigma\in\mathfrak{H}_N^D}$ with covariance
\begin{align}\label{e.Gaussian_covariance}
    \E [H_N(\sigma) H_N(\tau)] = N\xi\left(\frac{\sigma \tau^\intercal}{N}\right),\quad\forall \sigma, \tau \in \mathfrak{H}_N^D,
\end{align}
where the $D\times D$ real-valued matrix $\sigma \tau^\intercal$ is given by
\begin{align*}
    \sigma\tau^\intercal = \left(\la \sigma_d, \tau_{d'}\ra_{\mathfrak{H}_N}\right)_{1\leq d,d'\leq D}.
\end{align*}
We interpret $H_N(\sigma)$ as the random Hamiltonian of the $\mathfrak{H}_N^D$-valued spin configuration $\sigma$. 
Examples of $\xi$ and $H_N(\sigma)$ are given in \cite[Section~6]{mourrat2023free}.

\begin{example}\label{e.SK}
The Sherrington--Kirkpatrick model~\cite{sherrington1975solvable} corresponds to $D=1$, $\mathfrak{H}_N=\R^N$, $\xi(r)= r^2$, and $P_N$ uniform on $\{-1,+1\}^N$. The Hamiltonian can be expressed as
\begin{align*}
    H_N^{\mathrm{SK}}(\sigma) = \frac{1}{\sqrt{N}}\sum_{i,j=1}^Ng_{ij}\sigma_i\sigma_j,\quad\forall \sigma \in \R^N,
\end{align*}
where $\{g_{ij}\}_{1\leq i,j\leq N}$ is a collection of independent standard Gaussian random variables. 
\rrv{For the inverse temperature $\beta = \sqrt{2t}$ and $F_N$ introduced later in~\eqref{e.F_N}, we have
\begin{align*}
    \frac{1}{N}\E\log \int\exp\left(\beta H_N^{\mathrm{SK}}(\sigma)\right)\d P_N(\sigma) = -\bar F_N(t,\delta_0) + Nt.
\end{align*}
It is proven in~\cite{mourrat2019parisi} that $\bar F_N(t,\delta_0)$ converges to $f(t,\delta_0)$ where $f$ is given by the Hopf--Lax formula for~\eqref{e.HJ_spin_glass} with initial condition $f= \bar F_1(0,\cdot)$.}
\end{example}

Back to the general setting, in addition to the Hamiltonian $H_N(\sigma)$, we want to add an external field parametrized by $\mathcal{P}^\uparrow_1$, the set of monotone probability measures on $\S^D_+$ with finite first moments. 
This set will be defined in the next subsection.
Here, we present the construction of the external field associated with a discrete $\varrho\in \mathcal{P}^\uparrow_1$ in \cite{mourrat2023free}.
Using the isometry to be introduced in~\eqref{e.isometry}, one can check that $\varrho$ has the form
\begin{align}\label{e.varrho}
    \varrho= \sum_{k=0}^K(\zeta_{k+1}-\zeta_k)\delta_{q_k}
\end{align}
for some $K\in\N$, where $\zeta_0,\dots,\zeta_{K+1}\in\R$ satisfy $0=\zeta_0<\zeta_1<\cdots<\zeta_{K+1}=1$ and $q_0,\dots,q_K\in \S^D_+$ satisfy $0\leq q_0<q_1<\cdots<q_{K-1}<q_K$. For convenience, we also set $q_{-1} =0$.

Then, let $\mathcal{A}=\N^0\cup\N^1\cup\cdots\cup\N^K$ be a rooted tree with countably infinite degrees and depth $K$, where $\N^0=\{\emptyset\}$ contains the root. For each leaf $\alpha\in\N^K$, writing $\alpha = ( n_1,n_2,\dots,n_K)$, we denote the path from the root to $\alpha$ by $p(\alpha) = (\emptyset, (n_1), (n_1,n_2),\dots,(n_1,n_2,\dots,n_K))$. For each node $\beta\in\mathcal{A}$, we set $|\beta|$ to be its depth satisfying $\beta \in \N^{|\beta|}$.

Associated with the sequence $(\zeta_k)_{k=0}^{K+1}$,
there exists a family of random nonnegative weights $(\nu_\alpha)_{\alpha\in \N^K}$ (see ~\cite[(2.46)]{pan}), called the Poisson--Dirichlet cascade, indexed by the leaves of $\mathcal{A}$. The construction and properties can be seen in \cite[Chapter~2]{pan}. Since $\sum_{\alpha\in\N^K}\nu_\alpha=1$ almost surely, we can view $(\nu_\alpha)_{\alpha\in \N^K}$ as a random probability measure on the leaves of $\mathcal{A}$.
One can embed $\mathcal{A}$ into the unit sphere of a Hilbert space and interpret $(\nu_\alpha)_{\alpha\in \N^K}$ as a random probability measure supported on an ultrametric set. Again, we refer to~\cite[Chapter~2]{pan} for the detail. We take $(\nu_\alpha)_{\alpha\in \N^K}$ to be independent of $(H_N(\sigma))_{\sigma\in \mathfrak{H}^D_N}$.

Let $(z_\beta)_{\beta\in\mathcal{A}}$ be a family of independent standard $\mathfrak{H}^D_N$-valued Gaussian vectors, which can be defined through an isometry between $\mathfrak{H}^D_N$ and a Euclidean space. Then, we define the centered $\mathfrak{H}^D_N$-valued Gaussian process $(\mathsf{w}^\varrho(\alpha))_{\alpha\in\N^K}$ by
\begin{align*}
    \mathsf{w}^\varrho(\alpha) = \sum_{\beta\in p(\alpha)} \sqrt{q_{|\beta|}-q_{|\beta|-1}} z_\beta.
\end{align*}
Here, the square root is taken on matrices in $\S^D_+$. For $a\in\R^{D\times D}$ and $\mathbf{h} = (\mathbf{h}_d)_{d=1}^D\in \mathfrak{H}^D_N$, we understand $a \mathbf{h} = \left(\sum_{d'=1}^D a_{d,d'}\mathbf{h}_{d'}\right)_{d=1}^D \in \mathfrak{H}^D_N$. For $\mathbf{h}, \mathbf{h}'\in \mathfrak{H}^D_N$, we denote the inner product between them by $\mathbf{h}\cdot \mathbf{h}'$.

For $t\geq 0$ and $\varrho$ of form~\eqref{e.varrho}, we define
\begin{align}\label{e.F_N}
    F_N(t,\varrho) = - \frac{1}{N}\log \sum_{\alpha\in \N^K}\nu_\alpha \int \exp\left(\sqrt{2t}H_N(\sigma)- Nt\xi\left(\frac{\sigma\sigma^\intercal}{N}  \right)  + \sqrt{2}\mathsf{w}^\varrho(\alpha)\cdot \sigma - \sigma \cdot q_K\sigma\right)\d P_N( \sigma).
\end{align}
We also set $\bar F_N(t,\varrho) = \E F_N(t,\varrho)$ where $\E$ integrates all randomness.
Notice that $N\xi(\sigma\sigma^\intercal/N)$ is the variance of $H_N(\sigma)$ and $\sigma\cdot q_K\sigma$ is the variance of $\mathsf{w}(\alpha)^\varrho\cdot \sigma$. These terms are added to ensure that the exponential term has an expectation equal to one. 
It is proved in \cite[Proposition~3.1]{mourrat2023free} that, for each fixed $t$, $\bar F_N(t,\cdot)$ is Lipschitz in $\mathcal{P}^\uparrow_1$ (with metric defined below). Hence, we can extend $\bar F_N$ by continuity to $\R_+\times \mathcal{P}^\uparrow_1$.

\subsection{Monotone probability measures and isometry}\label{s.monotone_measure} 

A probability measure $\varrho$ on $\S^D_+$ is said to be \textit{monotone}, if
\begin{equation}\label{e.monotone_measure}
\P\left\{a\cdot X<a\cdot X'\textrm{ and }b\cdot X > b\cdot X'\right\} = 0, \quad\forall a, b \in \S^D_+,
\end{equation}
where $X$ and $X'$ are two independent $\S^D_+$-valued random variables with the same law $\varrho$. We denote the collection of such probability measures by $\cP^\uparrow$.

For $p\in[1,\infty)$, denote by $\cP^\uparrow_p$ the restriction of $\cP^\uparrow$ to those probability measures with finite $p$-th moments. We equip $\cP^\uparrow_p$ with the $p$-Wasserstein metric $\metric_p$ given by
\begin{align*}\metric_p(\varrho,\vartheta) =\inf_{\pi\in\Pi(\varrho,\vartheta)}\left(\int |x-y|^p\pi(\d x, \d y)\right)^\frac{1}{p}, \quad\forall \varrho,\vartheta \in \cP^\uparrow_p
\end{align*}
where $\Pi(\varrho,\vartheta)$ is collection of all couplings of $\varrho,\vartheta$. Here, a probability measure $\pi$ on $\S^D_+\times \S^D_+$ is said to be a coupling of $\varrho,\vartheta$ if the first marginal of $\pi$ is $\varrho$ and its second marginal is $\vartheta$. 

We want to embed $\cP^\uparrow_2$ isometrically onto the cone $\C$ given in \eqref{e.cone_M} with the ambient Hilbert space $\cH$ in \eqref{e.H_inft_d}. Throughout this section, let $U$ be the random variable distributed uniformly over $[0,1)$. By \cite[Propositions~2.4 and~2.5]{mourrat2023free}, the map $\mu\mapsto \hat \mu$ given by
\begin{align}\label{e.isometry}
    \begin{cases}
    \C \to  \cP^\uparrow_2
    \\
    \mu \mapsto \hat\mu=\Law(\mu(U))
    \end{cases}
\end{align}
is an isometric bijection. Moreover, we have
\begin{align*}\metric_p(\varrho,\vartheta)=\left|\check{\varrho}-\check\vartheta\right|_{L^p}= \left(\E\left[\left|\check\varrho(U)-\check\vartheta(U)\right|^p\right]\right)^\frac{1}{p},\quad\forall \varrho,\vartheta\in \cP^\uparrow_p,\ \forall p\in[2,\infty),
\end{align*}
where $\check{\varrho}$ is the inverse of $\varrho$ under~\eqref{e.isometry} and the same for $\check\vartheta$.
For $g:\cP^\uparrow_2\to\R$ and $f:\R_+\times\cP^\uparrow_2\to\R$, the actions of the isometry on them are given by
\begin{align*}\hat g: 
    \begin{cases}
    \C\to \R
    \\
    \mu \mapsto g(\hat \mu)
    \end{cases},
    \qquad
    \hat f:
    \begin{cases}
    \R_+\times \C\to\R
    \\
    (t,\mu)\mapsto f(t,\hat \mu)
    \end{cases}.
\end{align*}

\subsection{Interpretation of the equation}

We give an informal discussion on how to interpret the equation~\eqref{e.HJ_spin_glass} using the isometry.
We start by clarifying the meaning of $\partial_\varrho$.
For $g:\cP^\uparrow_2\to \R$ and any fixed $\varrho$, in a fashion as in \cite[Chapter~10]{ambrosio2005gradient}, we view $\partial_\varrho g$ as the element in $L^2((\S^D_+,\varrho);\,\S^D)$ satisfying
\begin{align}\label{e.diff_wass}
    g(\vartheta)-g(\varrho) = \int_{\S^D_+}\partial_\varrho g\cdot (\mathbf{t}^{\varrho\to\vartheta}-\mathbf{i})\d \varrho + o(\metric_2(\vartheta,\varrho))
\end{align}
as $\vartheta \to\varrho$ in $\cP^\uparrow_2$, where $\mathbf{i}$ is the identity map on $\S^D_+$ and $\mathbf{t}^{\varrho\to\vartheta}:\S^D_+\to\S^D_+$ is the optimal transport map. More precisely, the pushforward of $\varrho$ by $(\mathbf{i},\mathbf{t}^{\varrho\to\vartheta})$ is the optimal coupling of $(\varrho,\vartheta)$, namely, the law of $(\check\varrho(U),\check\vartheta(U))$. Hence, expressing $\partial_\varrho g$ as a map from $\S^D_+$ to $\S^D$, we can rewrite
\begin{align*}
    \int_{\S^D_+}\partial_\varrho g\cdot (\mathbf{t}^{\varrho\to\vartheta}-\mathbf{i})\d \varrho &= \E\left[ \left(\partial_\varrho g\circ\check\varrho(U)\right)\cdot\left(\check\vartheta(U)-\check\varrho(U)\right)\right]
    = \la \partial_\varrho g\circ \check\varrho,\,\check\vartheta-\check\varrho \ra_\cH.
\end{align*}
On the other hand, in view of Definition~\ref{d.differentiability}~\eqref{i.def_differentiable_C} and the isometry~\eqref{e.isometry}, the differential $\nabla\hat g$ of $\hat g$ at $\check\varrho$ satisfies
\begin{align*}
    \hat g(\check\vartheta)-\hat g(\check\varrho) = \la \nabla \hat g(\check \varrho),\, \check\vartheta-\check\varrho\ra_\cH + o(\metric_2(\vartheta,\varrho)),
\end{align*}
as $\vartheta\to\varrho$.
Comparing this with the previous two displays, we get
\begin{align*}
    \partial_\varrho g\circ \check\varrho=\nabla \hat g(\check \varrho)
\end{align*}
in $\cH$. Therefore, the first-order calculus on $\mathcal{P}^\uparrow_2$ is the same as that on $\C$.

Then, we can rewrite the nonlinearity in~\eqref{e.HJ_spin_glass} as
\begin{align*}
    \int \xi(\partial_\varrho g)\d \varrho= \E \left[\xi\big(\partial_\varrho g\circ \check\varrho(U)\big)\right] = \int_0^1 \xi\left(\nabla \hat g(\check \varrho)\right) 
\end{align*}
where $\int_0^1 \xi\left(\nabla \hat g(\check \varrho)\right) = \int_0^1\xi(\kappa)\d s$ by viewing $\nabla \hat g(\check \varrho)$ as a function $\kappa$ in $\cH$.
Hence, the equation \eqref{e.HJ_spin_glass} can be viewed as 
\begin{align}\label{e.HJ_spin_glass_xi_hat}
    \partial_t \hat f- \int_0^1\xi\left(\nabla \hat f\right)=0 , \quad\text{on }\R_+\times \C.
\end{align}
We do not attempt to make the above informal discussion rigorous and only hope to motivate our choice of interpreting~\eqref{e.HJ_spin_glass} as~\eqref{e.HJ_spin_glass_xi_hat} by an application of the isometry.

\subsection{Definition of solutions}

Since it is cumbersome to write out the isometry~\eqref{e.isometry}, henceforth, we identify every element in $\mathcal{P}^\uparrow_2$ with its preimage in $\C$ and we work directly with elements in $\C$. Hence, we view the enriched free energy $\bar F_N$ as a function $(t,\mu)\mapsto \bar F_N(t,\mu)$ on $\R_+\times\C$. Also, we can drop the hat in~\eqref{e.HJ_spin_glass_xi_hat} and view it as in~\eqref{e.HJ}.

For technical reasons, we need a regularized version of $\xi$.
Recall the definition of being increasing along a cone in \eqref{e.C^*-increasing}.
A function $g:\S^D_+\to\R$ is said to be \textit{proper} if $g$ is $\S^D_+$-increasing and, for every $b\in\S^D_+$, the function $\S^D_+\ni a\mapsto g(a+b)-g(a)$ is $\S^D_+$-increasing. 
\begin{definition}\label{d.regularization}
A function $\bar\xi:\S^D_+\to\R$ is said to be a \textit{regularization} of $\xi:\R^{D\times D}\to\R$ in \eqref{e.Gaussian_covariance} if
\begin{enumerate}
    \item \label{e.bar_xi_coincide} $\bar\xi$ coincides with $\xi$ on the intersection between $\S^D_+$ and the closed unit ball in $\S^D$;
    \item \label{e.bar_xi_Lip_and_proper} $\bar\xi$ is Lipschitz and proper;
\item \label{i.bar_xi_convex} $\bar\xi$ is convex, if, in addition, $\xi$ is convex on $\S^D_+$.
\end{enumerate}
\end{definition}

For any regularization $\xi$, we define the function $\H:\cH\to\R$ by
\begin{align}\label{e.def_H_spin_glass}
    \H(\kappa) = \inf\left\{\int_0^1\bar\xi(\mu(s))\d s:\: \mu \in \C\cap(\kappa+\C^*)\right\},\quad\forall \kappa\in\cH.
\end{align}
We study the properties of $\H$ in the next subsection.
Recall the definition of viscosity solutions in Definition~\ref{d.vs}.

\begin{definition}[Viscosity solutions of~\eqref{e.HJ_spin_glass}]\label{d.vis_sol_spin_glass}
Assume that ${\xi}:\R^{D\times D}\to\R$ is locally Lipschitz and $\xi\lfloor_{\S^D_+}$ is proper. A function $f:\R_+\times \C\to\R$ is said to be a \textit{viscosity subsolution} (respectively, \textit{supersolution}) of \eqref{e.HJ_spin_glass}, if there is a regularization $\bar\xi$ such that $f$ is a viscosity subsolution (respectively, supersolution) of $\HJ(\cH,\C,\H)$, namely,
\begin{align*}
    \partial_t f- \H(\nabla f) =0,\quad\text{on $\R_+\times\C$}
\end{align*}
for $\cH$, $\C$, $\H$ given in \eqref{e.H_inft_d}, \eqref{e.cone_M}, \eqref{e.def_H_spin_glass}, respectively. 
The function $f$ is said to be a \textit{viscosity solution} of \eqref{e.HJ_spin_glass} if $f$ is both a subsolution and a supersolution.
\end{definition}

In Lemma~\ref{l.exist_regularization} below, we show that the assumption on $\xi$ guarantees the existence of $\bar \xi$. As a corollary of the main theorem to be stated, we will see that $f$ does not depend on the choice of $\bar \xi$. In other words, we can define $f$ to be the function that solves $\HJ(\cH,\C,\H)$ for every choice of $\bar \xi$.

\subsection{Properties of regularization and nonlinearity}

\subsubsection{Comments on the regularization}
Let us justify the condition \eqref{e.bar_xi_coincide} in Definition~\ref{d.regularization} in the spin glass setting. It is expected that $(t,\varrho)\mapsto \bar F_N(t,\varrho)$ converges as $N\to\infty$ to a solution $f$ of \eqref{e.HJ_spin_glass}, at least when $\xi$ is convex on $\S^D_+$. Due to our assumption on the support of $P_N$, it has been shown in \cite[Proposition~3.1]{mourrat2023free} that
\begin{align}\label{e.bar_F_N_lip}
    \left|\bar F_N(t,\mu)-\bar F_N(t,\nu)\right|\leq \E |\mu(U)-\nu(U)|= \left|\mu-\nu\right|_{L^1},\quad\forall t\geq 0,\ \forall \mu,\nu\in\C,\ \forall N\in\N.
\end{align} 
Hence, $|\nabla \bar F_N(t,\mu)|_{L^\infty}\leq 1$ for every $N,t,\mu$. Passing to the limit, the same bound is expected to hold for $ f$.
In view of~\eqref{e.HJ},
this means that only values of $\xi$ on the closed unit ball of $\S^D$ matter. In addition, by \cite[Proposition~3.8]{mourrat2023free}, for every $N$ and $t$,
\begin{align}\label{e.F_N_C^*_nondecrea}
    \text{$\bar F_N(t,\cdot)$ is $\C^*$-increasing}
\end{align}
which by the duality of cones implies $\nabla\bar F_N(t,\mu)\in\C$ for every $\mu\in\C$. Passing to the limit, we expect $\nabla  f\in \C$ everywhere. Hence, only values of $\xi$ on $\S^D_+$ matter. Therefore, condition~\eqref{e.bar_xi_coincide} can be justified.

\subsubsection{Existence of regularization}
Since $\xi$ is the covariance function of a Gaussian field, there are many structures to exploit. Under the assumption that $\xi$ admits a convergent power series expansion, \cite[Propositions~6.4 and~6.6]{mourrat2023free} yield that $\xi$ is proper when restricted to $\S^D_+$.
The following lemma guarantees the existence of $\bar\xi$.

\begin{lemma}\label{l.exist_regularization}
If ${\xi}:\R^{D\times D}\to\R$ is locally Lipschitz and $\xi\lfloor_{\S^D_+}$ is proper, then there exists a regularization $\bar\xi$ of $\xi$. 
\end{lemma}

\begin{proof}
We follow the construction in \cite[Proposition~6.8]{mourrat2023free}. There, the definition of regularizations only requires \eqref{e.bar_xi_coincide} and \eqref{e.bar_xi_Lip_and_proper}. Here, we verify that $\bar\xi $ constructed is convex if $\xi$ is so. For $r>0$, we set $B_\tr(r) = \{a\in\S^D_+:\tr(a)\leq r\}$. Then, all $a\in\S^D_+$ with entries in $[-1,1]$ belong to $B_\tr(D)$.
In particular, $B_\tr(D)$ contains the closed unit ball intersected with $\S^D_+$.
For every $a\in\S^D_+$, we denote by $|a|_\infty$ the largest eigenvalue of $a$. Setting $L= \||\nabla \xi|_\infty\|_{L^\infty(B_\tr(2D))}$, we define, for every $a\in\S^D_+$,
\begin{align*}
    \bar\xi(a) = 
    \begin{cases}
    \xi(a)\vee (\xi(0)+2L(\tr(a)-D)), & \text{if }a\in B_\tr(2D),
    \\
    \xi(0)+2L(\tr(a)-D),& \text{if }a\not\in B_\tr(2D).
    \end{cases}
\end{align*}
Since $\xi$ is proper, we have $\xi(a)\geq \xi(0) \geq \xi(0)+2L(\tr(a)-D)$ for all $a\in B_\tr(D)$. Hence, $\bar\xi$ coincides with $\xi$ on matrices with entries in $[-1,1]$, verifying \eqref{e.bar_xi_coincide}. Note that $\bar\xi$ is continuous on $\{a\in\S^D_+:\tr(a)=2D\}$. Then, it is easy to check that $\bar\xi$ is Lipschitz. Due to the choice of $L$, we can also see that the gradient of $\bar\xi$ is increasing and thus $\bar\xi$ is proper, verifying \eqref{e.bar_xi_Lip_and_proper}.

Now, assuming that $\xi$ is convex on $\S^D_+$, we show that $\bar\xi$ is also convex. If $a,b\in B_\tr(2D)$ or $a,b\not\in B_\tr(2D)$, it is easy to check that
\begin{align}\label{e.bar_xi_convex}
    \bar\xi(\lambda a+(1-\lambda )b)\leq \lambda \bar\xi(a) + (1-\lambda)\bar\xi(b),\quad\forall \lambda\in[0,1].
\end{align}
Then, we consider $a\in B_\tr(2D)$ and $b\not\in B_\tr(2D)$. If $\lambda$ satisfies $\lambda a+(1-\lambda )b\not\in B_\tr(2D)$, then \eqref{e.bar_xi_convex} holds. Now, let $\lambda$ be such that $\lambda a+(1-\lambda )b\in B_\tr(2D)$.
There is $\gamma\in[0,\lambda]$ such that $c= \gamma a + (1-\gamma)b$ satisfies $\tr (c) =2D$. Then, for $\alpha = \frac{\lambda-\gamma}{1-\gamma}$, we have $\alpha a + (1-\alpha )c = \lambda a + (1-\lambda)b$. Since $\bar\xi$ is convex on $B_\tr(2D)$, the left-hand side of \eqref{e.bar_xi_convex} is bounded from above by $\alpha \bar\xi(a)+(1-\alpha)\bar\xi(c)$. On the other hand, using $\bar\xi(c) = \xi(0)+ 2L(\tr(c)-D)$ and the definition of $\bar\xi$, we have $\bar\xi(c)\leq \gamma \bar\xi(a)+(1-\gamma)\bar\xi(b)$. Combining these and using the choice of $\alpha$, we recover \eqref{e.bar_xi_convex}, verifying \eqref{i.bar_xi_convex}. 
\end{proof}

\begin{remark}\rm
By a straightforward modification of the above proof, for any $r\geq 1$, we can construct a regularization $\bar \xi$ of $\xi$ that coincides with $\xi$ on a centered closed ball with radius $r$ intersected with $\S^D_+$.
\end{remark}

\subsubsection{Properties of $\H$}
\begin{lemma}\label{l.H_lip}
Let $\H$ be given in~\eqref{e.def_H_spin_glass}. Then, the following hold:
\begin{enumerate}
\item \label{i.H=H_xi} $\H(\mu) = \int_0^1\bar\xi(\mu(s))\d s$ for every $\mu\in\C$;
\item \label{i.xi_proper=>} $\H$ is $\C^*$-increasing;
\item \label{i.xi_lip=>H_lip} $\H$ is Lipschitz;
\item \label{i.H_bdd_below} $\H$ is bounded below;
\item \label{i.xi_convex=>H_convex} if ${\xi}$ is convex on $\S^D_+$, then $\H$ is convex and satisfies $\H(\iota^\j)\leq \H(\iota)$ for every $j\in\J$ and every $\iota\in \cH$.
\end{enumerate}
\end{lemma}

\begin{proof}
We set $\H_{\bar\xi}(\mu) = \int_0^1\bar\xi(\mu(s))\d s$ for every $\mu\in\C$.

Part~\eqref{i.H=H_xi}. We first show that $\H_{\bar\xi}$ is $\C^*$-increasing on $\C$. We argue that it suffices to show $\H_{\bar\xi}^j$ is $(\C^j)^*$-increasing on $\C^j$, where $\H_{\bar\xi}^j$ is the $j$-projection of $\H_{\bar\xi}$. Indeed, for $\mu,\nu\in\C$ satisfying $\mu-\nu\in\C^*$, we have $\H_{\bar\xi}(\mu^\j) - \H_{\bar\xi}(\nu^\j) = \H^j_{\bar\xi}(\pj_j\mu)- \H^j_{\bar\xi}(\pj_j\nu)$. Due to Lemma~\ref{l.proj_cones}~~\eqref{i.pjM*}, we have $\pj_j \mu - \pj_j \nu \in (\C^j)^*$. Hence, $\H_{\bar\xi}(\mu^\j) - \H_{\bar\xi}(\nu^\j)\geq 0$ for every $j\in\J$. Passing to the limit along some $\Jgen$, and using Lemma~\ref{l.basics_(j)}~~\eqref{i.cvg} and the continuity of $\H_{\bar\xi}$ to conclude that $\H_{\bar\xi}(\mu) - \H_{\bar\xi}(\nu)\geq 0$. 

With this explained, we show that $\H_{\bar\xi}^j$ is $(\C^j)^*$-increasing. First, let us assume that $\bar\xi$ is differentiable everywhere. Since for every $x\in\C^j$,  $\H_{\bar\xi}^j(x) = \sum_{k=1}^{|j|}(t_k-t_{k-1}){\bar\xi}(x_k)$, we have (recall the inner product on $\cH^j$ given in \eqref{e.inner_product_H^j})
\begin{align*}
    \nabla_j \H_{\bar\xi}^j(x) = (\nabla{\bar\xi}(x_k))_{k=1,2,\dots,|j|},\quad\forall x\in\C^j.
\end{align*}
Here, $\nabla_j$ denotes the gradient of functions defined on subsets of $\cH^j$ and $\nabla$ on $\S^D_+$. Since ${\bar\xi}$ is proper, the above display implies $\nabla_j \H_{\bar\xi}^j(x)\in\C^j$. For any $x,y$ satisfying $x-y \in (\C^j)^*$, we have
\begin{align*}
    \H_{\bar\xi}^j(x) - \H_{\bar\xi}^j(y) = \int_0^1 \la \nabla_j \H_{\bar\xi}^j(sx+(1-s)y),\, x-y\ra_{\cH^j}\d s\geq 0,
\end{align*}
where we used the definition of dual cones to deduce that the integrand is nonnegative. Hence, $\H_{\bar\xi}^j$ is $(\C^j)^*$-increasing. The same can be deduced in general via a mollification argument. 

Having shown that $\H_{\bar\xi}$ is $\C^*$-increasing, we return to the proof. Let $\mu\in\C$. By the definition of $\H$, we clearly have $\H(\mu) \leq \H_{\bar\xi}(\mu)$. On the other hand, for every $\nu\in \C\cap(\mu+\C^*)$, the monotonicity of $\H_{\bar\xi}$ implies that $\H_{\bar\xi}(\mu)\leq \H_{\bar\xi}(\nu)$. Taking infimum in $\nu$, we obtain $\H_{\bar\xi}(\mu) \leq \H(\mu)$ verifying~\eqref{i.H=H_xi}.

Part~\eqref{i.xi_proper=>}. Let $\iota,\kappa\in\cH$ satisfy $\iota-\kappa\in\C^*$. For every $\mu\in\C\cap(\iota+\C^*)$, it is immediate that $\mu\in\C\cap(\kappa+\C^*)$, implying $\H_{\bar\xi}(\mu)\geq \H(\kappa)$. Taking infimum over $\mu\in\C\cap(\kappa+\C^*)$, we obtain~\eqref{i.xi_proper=>}.

Part~\eqref{i.xi_lip=>H_lip}.
Fix any $\iota,\iota'\in\cH$. Let $\nu$ be the projection of $\iota-\iota'$ to $\C$. Since $\C$ is closed and convex, we have
\begin{align}\label{e.iota-iota'-nu,rho-nu}
     \la \iota-\iota'-\nu,\ \rho-\nu\ra_\cH\leq 0,\quad\forall \rho\in\C.
\end{align}
Since $s\nu\in\C$ for all $s\geq 0$, \eqref{e.iota-iota'-nu,rho-nu} yields
\begin{align}
    \la \iota-\iota'-\nu,\nu\ra_{\cH}&= 0\label{e.iota-iota'-nu,nu}.
\end{align}
Inserting this back to \eqref{e.iota-iota'-nu,rho-nu}, we have $\la\iota-\iota'-\nu,\rho\ra_{\cH}\leq 0$ for all $\rho\in\C$, which implies
\begin{align*}
    \iota'-\iota+\nu\in{\cM}^*,
\end{align*}
For all $\mu\in\C\cap(\iota'+\C^*)$, the above display implies that $\mu+\nu\in\C\cap(\iota+\C^*)$. Since $\H$ is $\C^*$-increasing by \eqref{i.xi_proper=>}, we have
\begin{align*}
    \H_{\bar\xi}(\mu+\nu)\geq \H(\iota),\quad\forall \mu \in \C\cap(\iota'+\C^*).
\end{align*}
By \eqref{i.H=H_xi}, we get
\begin{align*}
    |\H_{\bar\xi}(\mu+\nu)-\H_{\bar\xi}(\mu)| \leq \E|{\bar\xi}(\mu(U)+\nu(U))-{\bar\xi}(\mu(U))|\leq \|{\bar\xi}\|_\mathrm{Lip}|\nu|_\cH.
\end{align*}
The above two displays imply
\begin{align*}
    \H(\iota)-\H_{\bar\xi}(\mu)\leq \|{\bar\xi}\|_\mathrm{Lip}|\nu|_\cH, \quad\forall \mu \in \C\cap(\iota'+\C^*).
\end{align*}
Due to \eqref{e.iota-iota'-nu,nu}, we can see that
\begin{align*}
    |\iota-\iota'|_{\cH}^2=|\iota-\iota'-\nu|_{\cH}^2+|\nu|_{\cH}^2\geq |\nu|_\cH^2.
\end{align*}
Using this and taking supremum over $\mu \in \C\cap(\iota'+\C^*)$, we obtain
\begin{align*}
    \H(\iota)-\H(\iota')\leq \|{\bar\xi}\|_\mathrm{Lip}|\iota-\iota'|_\cH.
\end{align*}
By symmetry, we conclude that $\H$ is Lipschitz.

Part~\eqref{i.H_bdd_below}. Since $\bar\xi$ is proper, we have $\bar\xi(a)\geq \bar\xi(0)$ for every $a\in\S^D_+$.
It is clear from~\eqref{e.def_H_spin_glass} that $\H$ is bounded below.

Part~\eqref{i.xi_convex=>H_convex}. 
By Definition~\ref{d.regularization}~\eqref{i.bar_xi_convex}, we have that $\bar\xi$ is convex.
From~\eqref{i.H=H_xi}, we can see that $\H$ is convex on $\C$. For every $\iota,\kappa\in\cH$ and every $s\in[0,1]$, we have $s \mu+(1-s)\nu \in \C\cap(s \iota+(1-s)\kappa+\C^*)$ if $\mu\in \C\cap(\iota+\C^*)$ and $\nu\in \C\cap(\kappa+\C^*)$. In view of this, the convexity of $\H$ on $\cH$ follows from its convexity on $\C$ and~\eqref{e.def_H_spin_glass}. To see the second claim, using Jensen's inequality, we have that, for every $\mu\in\C$ and every $j\in\J$,
\begin{align*}
    \H\left(\mu^\j\right)= \sum_{k=1}^{|j|}(t_{k+1}-t_k){\bar\xi}\left(\frac{1}{t_{k}-t_{k-1}}\int_{t_{k-1}}^{t_k}\mu(s)\d s\right)\leq \int_0^1{\bar\xi}(\mu(s))\d s = \H(\mu).
\end{align*}
Fix any $\iota\in\cH$. By Lemma~\ref{l.basics_(j)}~\eqref{i.lf_pj_(j)} and Lemma~\ref{l.proj_cones}, we have $\mu^\j\in \C\cap(\iota^\j+\C^*)$ for every $j\in \J$, and every $\mu\in \C\cap(\iota+\C^*)$. Therefore, the above display along with the definition of $\H$ implies that
\begin{align*}
    \H\left(\iota^\j\right)\leq \H(\mu),\quad\forall \mu\in \C\cap(\iota+\C^*).
\end{align*}
Taking infimum over $\mu\in \C\cap(\iota+\C^*)$, we conclude that $\H(\iota^\j)\leq \H(\iota)$.
\end{proof}

\subsection{Proof of the main result}

We state the rigorous version of Theorem~\ref{t.informal}.
Recall that we have identified $\mathcal{P}^\uparrow_2$ with $\C$ via the isometry~\eqref{e.isometry} and recall the notion of viscosity solutions of~\eqref{e.HJ_spin_glass} given in Definition~\ref{d.vis_sol_spin_glass}.
Also recall the Hilbert spaces $\cH$ and $\cH^j$, $j\in\J$, in \eqref{e.H_inft_d} and \eqref{e.def_H^j}, respectively; the cones $\C$ and $\C^j$, $j\in\J$, in \eqref{e.cone_M} and \eqref{e.C^j}, respectively; the definition of $\C^*$-increasingness in \eqref{e.C^*-increasing}; the definition of good collections of partitions at the beginning of Section~\ref{s.partition}; lifts and projections of functions in Definition~\ref{d.lift_proj}.
We write
\begin{align}\label{e.bxi}
    \bxi(\kappa) = \int_0^1\xi(\kappa(s))\d s,\quad\forall \kappa\in L^\infty.
\end{align}
For $j\in\J$, we can define the $j$-projection $\bxi^j:\cH^j\to\R$ as in Definition~\ref{d.lift_proj}. So, for $x\in \cH^j$, we have $\bxi^j(x) =\sum_{k=1}^{|j|}(t_k-t_{k-1})\xi(x_k)$. Now, we are ready to state the result.

\begin{theorem}\label{t.main}
Suppose
\begin{itemize}
    \item ${\xi}:\R^{D\times D}\to\R$ is locally Lipschitz and $\xi\lfloor_{\S^D_+}$ is proper;
    \item $\psi:\C\to\R$ is $\C^*$-increasing and satisfies, 
    \begin{align}\label{e.psi_1-wasser-Lip}
        |\psi(\mu)-\psi(\nu)|\leq |\mu-\nu|_{L^1},\quad\forall \mu,\nu\in \C.
    \end{align}
\end{itemize}
Then, there is a unique Lipschitz viscosity solution $f$ of~\eqref{e.HJ_spin_glass} with $f(0,\cdot) = \psi$.
Moreover, 
\begin{enumerate}
    \item \label{i.fin_d_approx} $ f=\lim_{j\in\Jgood}f^\uparrow_j$ in the local uniform topology, for any good collection of partitions $\Jgood$, where each $f_j:\R_+\times\C^j\to\R$ is the unique Lipschitz viscosity solution of $\HJ(\cH^j,\mathring\C^j,\bxi^j; \psi^j)$;
    \item \label{i.hopf_lax} if $\xi$ is convex on $\S^D_+$, then $f$ is given by the Hopf--Lax formula
    \begin{align}\label{e.Hopf-Lax_spin_glass}
        f(t,\mu) = \sup_{\nu \in \C\cap L^\infty}\inf_{\rho\in\C\cap L^\infty}\left\{\psi(\mu+\nu)- \la \nu,\rho\ra_{\cH}+t\int_0^1\xi(\rho(s))\d s\right\},\quad\forall(t,\mu)\in\R_+\times\C;
    \end{align}
    \item \label{i.hopf} if $\psi$ is convex, then $f$ is given by the Hopf formula
    \begin{align}\label{e.Hopf_spin_glass}
        f(t,\mu) = \sup_{\rho \in \cM\cap L^\infty}\inf_{\nu \in \cM\cap L^\infty}\left\{ \psi(\nu) + \la \mu - \nu, \rho\ra_\cH +t\int_0^1\xi(\rho(s))\d s\right\},\quad\forall(t,\mu)\in\R_+\times\C.
    \end{align}
\end{enumerate}
\end{theorem}

\begin{proof}
Lemma~\ref{l.exist_regularization} guarantees the existence of regularizations (see Definition~\ref{d.regularization}). We fix any regularization $\bar\xi$ and let $\H$ be given in~\eqref{e.def_H_spin_glass}.
The properties of $\H$ are listed in Lemma~\ref{l.H_lip}. 
The existence of a Lipschitz viscosity solution $f$ of $\HJ(\cH,\C,\H;\psi)$ follows from Propositions~\ref{p.lim_is_vis_sol} and Proposition~\ref{p.exist_sol_approx}. 
In view of Definition~\ref{d.vis_sol_spin_glass}, $f$ is a viscosity solution of~\eqref{e.HJ_spin_glass} with $f(0,\cdot)=\psi$. Notice that the uniqueness of $f$ will follow from~\eqref{i.fin_d_approx}.

Let us prove~\eqref{i.fin_d_approx}.
Let $f_j$ be given in Proposition~\ref{p.exist_sol_approx}, which is the unique Lipschitz viscosity solution of $\HJ(\cH^j,\C^j,\H^j;\psi^j)$ given by Theorem~\ref{t.hj_cone}~\eqref{i.main_equiv}. The convergence $f=\lim_{j\in\Jgood}f^\uparrow_j$ is already given by Proposition~\ref{p.exist_sol_approx}.

We verify that $f_j$ is the unique Lipschitz solution of $\HJ(\cH^j,\mathring\C^j,\bxi^j,\psi^j)$.
On $\cH^j$, we consider norms $|x|_{\ell^1} = \sum_{k=1}^{|j|}(t_k-t_{k-1})|x_k|$ and $|x|_{\ell^\infty}=\sup_{1\leq k\leq |j|} |x_k|$. Due to~\eqref{e.psi_1-wasser-Lip}, we can verify $|\psi^j(x)-\psi^j(x')|\leq |x-x'|_{\ell^1}$ for all $x,x'\in\cH^j$. Invoking Proposition~\ref{p.lip_l^p}, we have $|f_j(t,x)-f_j(t,x')|\leq |x-x'|_{\ell^1}$ for all $t\geq 0$ and $x,x'\in\cH^j$. 
Theorem~\ref{t.hj_cone}~\eqref{i.main_equiv} ensures that $f_j$ belongs to the class~\eqref{e.class_sol} and thus $f_j(t,\cdot)$ is $(\C^j)^*$-increasing for every $t\geq 0$. These properties of $f_j$ imply that, if $f_j-\phi$ achieves a local extremum at some $(t,x)\in (0,\infty)\times \mathring\C^j$ for some smooth $\phi$, then
\begin{align}\label{e.obs_1_pf_main}
    \nabla\phi(t,x)\in \{a\in \C^j: |a|_{\ell^\infty}\leq 1\}.
\end{align}
By Lemma~\ref{l.H_lip}~\eqref{i.H=H_xi} and Definition~\ref{d.regularization}~\eqref{e.bar_xi_coincide}, we can see
\begin{align}\label{e.obs_2_pf_main}
    \H^j(a) = \sum_{k=1}^{|j|} (t_k-t_{k-1})\bar \xi(a_k) = \bxi^j(a), \quad\forall a\in \C^j:\: |a|_{\ell^\infty}\leq 1.
\end{align}
From \eqref{e.obs_1_pf_main}, \eqref{e.obs_2_pf_main}, and Definition~\ref{d.vs} of viscosity solutions, we can verify that $f_j$ is a viscosity solution of $\HJ(\cH^j,\mathring\C^j,\bxi^j,\psi^j)$. Since $\xi\lfloor_{\S^D_+}$ is proper, we can use the same argument in the proof of Lemma~\ref{l.H_lip}~\eqref{i.H=H_xi} (substituting $\xi$ for $\bar\xi$ therein) to see that $\bxi^j$ is $(\C^j)^*$-increasing.
Since $\bxi^j$ is also locally Lipschitz and $\psi^j:\C^j\to\R$ is clearly Lipschitz and $(\C^j)^*$-increasing, the uniqueness follows from Theorem~\ref{t.hj_cone}. This completes~\eqref{i.fin_d_approx}.

To prove~\eqref{i.hopf_lax} and~\eqref{i.hopf}, we want to use Theorem~\ref{t.hj_cone}~\eqref{i.main_var_rep}. For this, Proposition~\ref{p.fenchel-moreau}, stated and proved later, ensures that $\C^j$ has the Fenchel--Moreau property defined in Definition~\ref{d.fenchel_moreau_prop}. 

Given that $\xi\lfloor_{\S^D_+}$ is convex and proper, it is easy to see that $\bxi^j\lfloor_{\C^j}$ is convex and bounded below for every $j$.
Since $f_j$ solves $\HJ(\cH^j,\mathring\C^j,\bxi^j,\psi^j)$, Theorem~\ref{t.hj_cone}~\eqref{i.main_var_rep} gives
\begin{align*}
    f_j(t,x)=\sup_{y \in \C^j}\inf_{z\in\C^j}\left\{\psi^j(x+y)-\la y,z\ra_{\cH^j}+ t \bxi^j\left(z\right)\right\},\quad\forall(t,x)\in\R_+\times\C^j.
\end{align*}
The convexity of $\xi$ on $\S^D_+$ together with Jensen's inequality also implies $\bxi\left(\nu^\j\right)\leq \bxi(\nu)$ for every $\nu\in\C\cap L^\infty$.
Hence, Proposition~\ref{p.hopf_lax_cvg} along with Remark~\ref{r.hopf_lax_cvg} yields~\eqref{i.hopf_lax}.

Under the assumption that $\psi$ is convex, it is straightforward to see that $\psi^j:\C^j\to\R$ is also convex. Invoking Theorem~\ref{t.hj_cone}~\eqref{i.main_var_rep}, we get
\begin{align*}
    f_j(t,x) = \sup_{z\in\C^j}\inf_{y\in\C^j} \left\{\psi^j(y)+\la x-y,z\ra_{\cH^j}+t\bxi^j(z)\right\},\quad\forall (t,x)\in\R_+\times\C^j.
\end{align*}
Then,~\eqref{i.hopf} follows from Proposition~\ref{p.hopf_cvg} along with Remark~\ref{r.hopf_cvg}.
\end{proof}

\begin{remark}[Comparison principle]
\rm
In view of Definition~\ref{d.vis_sol_spin_glass} and Lemma~\ref{l.H_lip}, Proposition~\ref{p.comp} supplies a comparison principle for \eqref{e.HJ_spin_glass}.
\end{remark}

\begin{remark}[Assumptions on $\xi$ and $\psi$]\rm \label{r.assump_xi_psi}
In most of the interesting models, $\xi$ in \eqref{e.Gaussian_covariance} is a convergent power series and proper on $\S^D_+$ (see \cite[Propositions~6.4 and~6.6]{mourrat2023free}). In practice, $\psi$ will be the limit of $\bar F_N(0,\cdot)$ as $N\to\infty$. Due to~\eqref{e.bar_F_N_lip} and~\eqref{e.F_N_C^*_nondecrea}, the assumption on $\psi$ is natural. In general, $\psi$ is neither concave nor convex, which renders the Hopf formula less useful. A discussion on the existence of variational formulas for the limit free energy is in \cite[Section~6]{mourrat2020nonconvex}.
\end{remark}

\begin{remark}[Independence of regularizations]\rm
In Definition~\ref{d.vis_sol_spin_glass}, $f$ is only required to solve $\HJ(\cH,\C,\H)$ for $\H$ in~\eqref{e.def_H_spin_glass} associated with some regularization $\bar\xi$ of $\xi$. In the proof of Theorem~\ref{t.main}, the choice of $\bar\xi$ is arbitrary. Hence, in fact, we can define $f$ to solve $\HJ(\cH,\C,\H)$ for any regularization of $\xi$.
\end{remark}

\begin{remark}[Solving regularized finite-dimensional equations]\rm \label{r.f_j_sol_reg}
For any regularization $\bar\xi$, let us define $\bar{\bxi}:\C\to\R$ by
\begin{align}\label{e.bar_bold_xi}
    \bar{\bxi}(\kappa) = \int_0^1 \bar\xi(\kappa(s))\d s,\quad\forall \kappa\in\cH
\end{align}
which is well-defined because $\bar\xi$ is Lipschitz.
In the proof of the theorem, we initially take $f_j$ to be the unique viscosity solution of $\HJ(\cH^j,\C^j,\H^j;\psi^j)$.
From~\eqref{e.obs_2_pf_main}, we also have $\H^j(a)= \bar{\bxi}^j(a)$ for all $a\in\C^j$ satisfying $|a|_{\ell^\infty}\leq 1$. Combining this with~\eqref{e.obs_1_pf_main}, we can similarly deduce that $f_j$ in Theorem~\ref{t.main}~\eqref{i.fin_d_approx} is the unique Lipschitz viscosity solution of $\HJ(\cH^j,\mathring\C^j,\bar{\bxi}^j;\psi^j)$ for any regularization $\bar\xi$.
\end{remark}

\begin{remark}[Regularized variational formulas]\rm
We continue from the results in the previous remark.
If $\xi$ is convex on $\S^D_+$, Definition~\ref{d.regularization}~\eqref{i.bar_xi_convex} ensures that $\bar\xi$ is convex. Moreover, $\bar\xi$ is proper by definition. By the same argument in the proof of Theorem~\ref{t.main}~\eqref{i.hopf_lax} (without using Remark~\ref{r.hopf_lax_cvg}), we can get
\begin{align*}
    f(t,\mu) = \sup_{\nu \in \C}\inf_{\rho\in\C}\left\{\psi(\mu+\nu)- \la \nu,\rho\ra_{\cH}+t\int_0^1\bar\xi(\rho(s))\d s\right\},\quad\forall(t,\mu)\in\R_+\times\C.
\end{align*}
If $\psi$ is convex, by the same reasoning in the proof of Theorem~\ref{t.main}~\eqref{i.hopf} (without using Remark~\ref{r.hopf_cvg}), we have
\begin{align*}
    f(t,\mu) = \sup_{\rho \in \cM}\inf_{\nu \in \cM}\left\{ \psi(\nu) + \la \mu - \nu, \rho\ra_\cH +t\int_0^1\bar\xi(\rho(s))\d s\right\},\quad\forall(t,\mu)\in\R_+\times\C.
\end{align*}
We stress that these formulas hold for any regularization $\bar\xi$. Also, by Remarks~\ref{r.hopf_lax_cvg} and~\ref{r.hopf_cvg}, we can replace $\nu\in\C$ and $\rho\in\C$ by $\nu\in\C\cap L^\infty$ and $\rho\in\C\cap L^\infty$ in the above formulas.
\end{remark}

\subsection{Other notions of solution} \label{s.def_vs_equiv}

We show that solutions of~\eqref{e.HJ_spin_glass} in \cite{mourrat2019parisi,mourrat2020extending,mourrat2020nonconvex,mourrat2023free} are viscosity solutions.

\subsubsection{Solutions in \cite{mourrat2020nonconvex,mourrat2023free}} 
The spin glass setting in \cite{mourrat2023free} is the same as in Section~\ref{s.spin-glass_setting}.
The definition of solutions of~\eqref{e.HJ_spin_glass} is in \cite[Proposition~4.5]{mourrat2023free}. We briefly present it here in our notation. Fix any regularization $\bar\xi$ and let $\bar{\bxi}$ be given in~\eqref{e.bar_bold_xi}. 
Recall that we have identified $\mathcal{P}^\uparrow_2$ with $\C$ via the isometry~\eqref{e.isometry}.
Let $\psi:\C\to\R$ be $\C^*$-increasing and satisfy~\eqref{e.psi_1-wasser-Lip}.
In \cite{mourrat2023free}, the uniform partitions of $[0,1)$ are considered. For each $j\in \Junif$, we set $\sF_j(b) = \inf\{\bar{\bxi}^j(a):\: a\in \C^j\cap (b+(\C^j)^*)\}$ for $b\in \cH^j$ (see \cite[(4.18) and (4.19)]{mourrat2023free}).
Then, let $f_j$ (see \cite[(4.21)]{mourrat2023free}) be the viscosity solution of
\begin{align*}
    \begin{cases}
        \partial_t f_j - \sF_j(\nabla f_j) =0,\quad & \text{in $(0,\infty)\times \mathring\C^j$},
        \\
        \mathbf{n}\cdot\nabla f_j=0 ,\quad & \text{on $\partial\C^j$}.
    \end{cases}
\end{align*}
with initial condition $f_j(0,\cdot) = \psi^j$. The precise notion of solutions of the displayed equation is given in \cite[Definition~4.1]{mourrat2023free}. We omit the exact definition of the Neumann boundary condition and only mention that, by restricting to the interior, $f_j$ solves $\HJ(\cH^j,\mathring\C^j,\sF_j;\psi^j)$ in the sense of Definition~\ref{d.vs}.
Then, in~\cite[Proposition~4.5]{mourrat2023free}, the solution of~\eqref{e.HJ_spin_glass} is defined to be the pointwise limit of $f^\uparrow_j$.

\begin{proposition}
Suppose that $\xi$ and $\psi$ satisfy the conditions in Theorem~\ref{t.main}.
The solution of~\eqref{e.HJ_spin_glass} with initial condition $\psi$ defined in~\cite[Proposition~4.5]{mourrat2023free} is the unique Lipschitz viscosity solution of~\eqref{e.HJ_spin_glass} with $f(0,\cdot) = \psi$ given by Theorem~\ref{t.main}.
\end{proposition}
\begin{proof}
It is verified in \cite[Proposition~4.3]{mourrat2023free} that $f_j$ is Lipschitz. By a similar argument in Lemma~\ref{l.H_lip}, we can show that $\sF_j$ coincides with $\bar{\bxi}^j$ on $\C^j$ and $\sF_j$ is $(\C^j)^*$-increasing and Lipschitz.
Since $f_j$ solves $\HJ(\cH^j,\mathring\C^j,\sF_j;\psi^j)$, Proposition~\ref{p.drop_bdy} implies that $f_j$ is a Lipschitz viscosity solution of $\HJ(\cH^j,\C^j,\sF_j;\psi^j)$.
Theorem~\ref{t.hj_cone}~\eqref{i.main_equiv} implies that $f_j$ is the unique Lipschitz viscosity solution of $\HJ(\cH^j,\mathring\C^j,\bar{\bxi}^j;\psi^j)$. Now, by Remark~\ref{r.f_j_sol_reg}, $f_j$ is exactly the one in Theorem~\ref{t.main}~\eqref{i.fin_d_approx}, which implies the desired result. 
\end{proof}

Therefore, in view of this proposition and Remark~\ref{r.assump_xi_psi}, the main result \cite[Theorem~3.4]{mourrat2023free} can be restated as follows.

\begin{theorem}[\cite{mourrat2023free}]\label{t.genbound}
Under the setting in Section~\ref{s.spin-glass_setting}, suppose
\begin{itemize}
    \item ${\xi}:\R^{D\times D}\to\R$ is locally Lipschitz and $\xi\lfloor_{\S^D_+}$ is proper;
    \item $\bar F_N(0,\cdot)$ converges pointwise to some $\psi:\C\to\R$ as $N\to\infty$.
\end{itemize}
Then, for every $(t,\mu) \in \R_+\times\C$,
\begin{align*}
    \liminf_{N\to\infty} \bar F_N(t,\mu) \geq f(t,\mu)
\end{align*}
where $f$ is the unique Lipschitz viscosity solution of~\eqref{e.HJ_spin_glass} with $f(0,\cdot) = \psi$ given by Theorem~\ref{t.main}.
\end{theorem}

It is mentioned in \cite[Remark~3.5]{mourrat2023free} that $f$ is expected to be independent of the choice of $\bar\xi$ and the missing ingredient therein is an $\ell^\infty$-control of the gradient of the solution by the initial condition. Here, such control is supplied by Proposition~\ref{p.lip_l^p} and we have verified the independence from $\bar\xi$.

We state a corollary of this theorem and Theorem~\ref{t.main}~\eqref{i.hopf_lax}.

\begin{corollary}
Under the same setup of Theorem~\ref{t.genbound}, if $\xi$ is convex on $\S^D_+$, then, for every $(t,\mu) \in \R_+\times\C$,
\begin{align*}
    \liminf_{N\to\infty} \bar F_N(t,\mu) \geq \sup_{\nu \in \C\in L^\infty}\inf_{\rho\in\C\cap L^\infty}\left\{\psi(\mu+\nu)- \la \nu,\rho\ra_{\cH}+t\int_0^1\xi(\rho(s))\d s\right\}.
\end{align*}
\end{corollary}

Notice the minus sign in the definition of $F_N$ in \eqref{e.F_N}, which is not included in most of the literature. Hence, the lower bound here corresponds to the usual upper bound. Classically, the upper bound is obtained via the Guerra interpolation \cite{gue03}, which requires $\xi$ to be convex on $\R^{D\times D}$. If $D=1$, one can use Talagrand's positivity principle (see \cite[Theorem~3.4]{pan}) to allow $\xi$ to be convex only on $\R_+$. For $D>1$, there is no general method to weaken the convexity condition. The above corollary provides a solution to this issue. It will be shown in \cite[Theorem~1.1]{HJ_critical_pts} via cavity computation that the Hopf--Lax formula in the corollary is in fact the limit.

\subsubsection{Solutions in \cite{mourrat2019parisi,mourrat2020extending}}
The solutions in both works are defined directly as the Hopf--Lax formula.
We present the setting in \cite{mourrat2020extending} in the notation here. Let $D=1$ in which case $\S^D=\R$ and $\S^D_+=\R_+$. Set $\mathfrak{H}_N = \R^N$. Hence, for each $N\in\N$, spins are real-valued and the spin configuration $\sigma$ is a vector in $\R^N$. We also set $P_N = P_1^{\otimes N}$, which means that the spins are i.i.d. We also assume that $P_1$ is supported on $[-1,1]$. Suppose $\xi(r)= \sum_{p\geq 2}\beta^2_pr^p$ for $r\in\R$ where the sequence $(\beta_p)_{p}$ of real numbers is assumed to decay sufficiently fast. 
We set $\psi = \bar F_1(0,\cdot)$ and
\begin{align}\label{e.xi^*=}
    \xi^*(r) = \sup_{s\geq 0}\left\{rs-\xi(s)\right\},\quad\forall r\in\R.
\end{align}
Using the isometry in~\eqref{e.isometry}, we can rewrite the main result \cite[Theorem~1.1]{mourrat2020extending} as follows.
\begin{theorem}[\cite{mourrat2020extending}]\label{t.hjsoft}
Let $\xi$ and $\psi$ be given above.
For every $t\geq0$ and $\mu\in \C\cap L^\infty$,
\begin{align*}
    \lim_{N\to\infty} \bar F_N(t,\mu) = \sup_{\nu\in \C\cap L^\infty}\left\{\psi(\nu) - t\int_0^1 \xi^*\left(\frac{\nu(s)-\mu(s)}{t}\right)\d s\right\}.
\end{align*}
\end{theorem}

In \cite{mourrat2019parisi,mourrat2020extending}, the right-hand side is defined to be the solution of~\eqref{e.HJ_spin_glass}. We show that it is the viscosity solution.

\begin{proposition}
The limit in Theorem~\ref{t.hjsoft} coincides at every $(t,\mu)\in\R_+\times (\C\cap L^\infty)$ with the unique Lipschitz viscosity solution $f$ of~\eqref{e.HJ_spin_glass} with $f(0,\cdot) = \psi$ given by Theorem~\ref{t.main}.    
\end{proposition}

\begin{proof}
It is easy to verify that $\xi$ satisfies the condition in Theorem~\ref{t.main}. By Remark~\ref{r.assump_xi_psi}, $\psi$ also satisfies the required condition.
Note that $\xi$ is convex on $\S^1_+=\R_+$. So, the viscosity solution $f$ is given by the Hopf--Lax formula~\eqref{e.Hopf-Lax_spin_glass} in Theorem~\ref{t.main}~\eqref{i.hopf_lax}. 
By some elementary analysis methods, we verify in Proposition~\ref{p.equival_exp_hopf_lax} that~\eqref{e.Hopf-Lax_spin_glass} coincides with the formula in Theorem~\ref{t.hjsoft} on $(t,\mu)\in\R_+\times (\C\cap L^\infty)$.
\end{proof}

\appendix

\section{Hopf--Lax formula in one dimension}

Throughout this section, we set $D=1$, in which case $\S^D=\R$ and $\S^D_+=\R_+$. Hence, for $a,b\in\S^D$, $a\cdot b= ab$.
The goal is to prove Proposition~\ref{p.equival_exp_hopf_lax}.

We introduce the notation for the \textit{nondecreasing rearrangement}. For every uniform partition $j\in\Junif$ (see Section~\ref{s.partition}) and $x\in \cH^j$, we set $x_\sharp = (x_{\sigma(k)})_{k=1,2,\dots,|j|}$ where $\sigma$ is a permutation of $\{1,2,\dots,|j|\}$ satisfying $x_{\sigma(k)}  - x_{\sigma(k-1)}\in\R_+$ for every $k\geq 2$. Using this notation, for every $j\in\Junif$ and every $\iota\in \cH$, we set $\iota^\j_\sharp = \lf_j ((\pj_j \iota)_\sharp)$.
We also take $\cH_+ = \{\iota\in\cH:\iota(s)\in\R_+,\,\text{a.e.}\,s\in[0,1]\}$.
\begin{lemma}\label{l.nond_rearrang}
For every $j\in\Junif$,
\begin{enumerate}
    \item \label{i.iota^j_sharp_in_iota^j+C*} $\iota^\j_\sharp \in \iota^\j+\C^*$ for every $\iota\in\cH$;
    \item \label{i.iota^j_sharp_in_C} $\iota^\j_\sharp\in \C$ for every $\iota\in\cH_+$;
    \item \label{i.Eh(iota^j)} $\int_0^1 h\left(\iota^\j(s)\right)\d s =\int_0^1 h\left(\iota^\j_\sharp(s)\right)\d s$ for every measurable function $h:\R\to\R$ and every $\iota\in\cH$.
\end{enumerate}
\end{lemma}
\begin{proof}
Part~\eqref{i.iota^j_sharp_in_iota^j+C*}.
For every $x\in\cH^j$, by the rearrangement inequality, we have
\begin{align*}
    \la x_\sharp, y \ra_{\cH^j} = \frac{1}{|j|}\sum_{k=1}^{|j|}x_{\sigma(k)}y_k \geq \frac{1}{|j|}\sum_{k=1}^{|j|}x_ky_k = \la x, y \ra_{\cH^j},\quad\forall y \in \C^j,
\end{align*}
where $\sigma$ is the permutation in the definition of $x_\sharp$. This implies $x_\sharp - x \in (\C^j)^*$. By the definition of $\iota^\j_\sharp$ and Lemma~\ref{l.basics_(j)}~\eqref{i.pj_lf=ID}, we have $\pj_j\left(\iota^\j_\sharp \right)= (\pj_j\iota)_\sharp$. Hence, we get $\pj_j\left(\iota^\j_\sharp \right ) - \pj_j\iota \in (\C^j)^*$, which along with Lemma~\ref{l.proj_cones}~~\eqref{i.lM^j*} implies that $\left(\iota^\j_\sharp \right)^\j - \iota^\j \in \C^*$. By Lemma~\ref{l.basics_(j)}~\eqref{i.pj_lf=ID} and~\eqref{i.lf_pj_(j)}, we have
\begin{align*}
    \left(\iota^\j_\sharp \right)^\j = \lf_j\pj_j\lf_j((\pj_j\iota)_\sharp) = \lf_j((\pj_j\iota)_\sharp) = \iota^\j_\sharp.
\end{align*}
Then, \eqref{i.iota^j_sharp_in_iota^j+C*} follows.

Part~\eqref{i.iota^j_sharp_in_C}. Let $\iota\in\cH_+$. It is clear from the definition that $(\pj_j\iota)_\sharp\in \C^j$. Then, by Lemma~\ref{l.proj_cones}~~\eqref{i.lM^j}, we get $\iota^\j_\sharp\in \C$. 

Part~\eqref{i.Eh(iota^j)}. We can compute
\begin{align*}
    \int_0^1 h\left(\iota^\j(s)\right)\d s= \frac{1}{|j|}\sum_{k=1}^{|j|}h\left(|j|\int_{\frac{k-1}{|j|}}^{\frac{k}{|j|}}\iota(s)\d s\right)= \frac{1}{|j|}\sum_{k=1}^{|j|}h\left(|j|\int_{\frac{\sigma(k)-1}{|j|}}^{\frac{\sigma(k)}{|j|}}\iota(s)\d s\right) =  \int_0^1 h\left(\iota^\j_\sharp(s)\right)\d s.
\end{align*}
This completes the proof.
\end{proof}

Recall $\xi^*$ in~\eqref{e.xi^*=} and $\bxi$ in~\eqref{e.bxi}. We define
\begin{align*}
    \bxi^*(\iota) = \sup_{\nu\in\C\cap L^\infty}\left\{\la \iota,\nu\ra_\cH - \bxi(\nu)\right\},\quad\forall \iota\in \cH.
\end{align*}

\begin{lemma}\label{l.H^circledast=E_xi^*}
For every $\mu\in \C$, it holds that $\bxi^*(\mu) = \int_0^1{\xi}^*(\mu(s))\d s$.
\end{lemma}

\begin{proof}
Step~1. We show
\begin{align}\label{e.H^*=H^oplus}
    \bxi^*(\mu) = \bxi^\oplus(\mu),\quad\forall\mu\in\C,
\end{align}
where $\bxi^\oplus$ is defined by
\begin{align*}
    \bxi^\oplus(\iota) = \sup_{\kappa\in \cH_+\cap L^\infty}\{\la\iota,\kappa\ra_\cH - \bxi(\kappa)\},\quad\forall \iota\in\cH.
\end{align*}
Due to $\C\subset \cH_+$, we have $\bxi^*(\mu)\leq \bxi^\oplus(\mu)$ for $\mu\in\C$. For any $\eps>0$, there is $\kappa\in\cH_+\cap L^\infty$ such that
\begin{align*}
    \bxi^\oplus(\mu) \leq \la\mu,\kappa\ra_\cH - \bxi(\kappa) +\eps.
\end{align*}
Using Lemma~\ref{l.basics_(j)}~~\eqref{i.cvg} and the local Lipschitzness of $\xi$, we can find $j\in \Junif$ such that
\begin{align*}
    \bxi^\oplus(\mu) \leq \la\mu,\kappa^\j\ra_\cH - \bxi\left(\kappa^\j\right) +2\eps.
\end{align*}
Lemma~\ref{l.nond_rearrang} implies that $\la\mu,\kappa^\j\ra_\cH\leq \langle\mu,\kappa^\j_\sharp\rangle_\cH$, $\kappa^\j_\sharp\in\C$, and $\bxi(\kappa^\j) = \bxi(\kappa^\j_\sharp)$. These together with the above display yield $\bxi^\oplus(\mu)\leq \bxi^*(\mu) +2\eps$. Since $\eps$ is arbitrary, we obtain~~\eqref{e.H^*=H^oplus}.

Step~2. We show
\begin{align}\label{e.H^oplus=E_xi^*}
    \bxi^\oplus(\iota) = \int_0^1{\xi}^*(\iota(s))\d s,\quad\forall \iota \in \cH.
\end{align}
For every $\iota\in\cH$, $\kappa\in\cH_+\cap L^\infty$, by the definition of ${\xi}^*$, we have
\begin{align*}
    {\xi}^*(\iota(s))\geq \iota(s) \kappa(s) - {\xi}(\kappa(s)),\quad\forall s\in[0,1).
\end{align*}
Integrating in $s$, we get
\begin{align*}
    \int_0^1{\xi}^*(\iota(s))\d s\geq \la \iota,\kappa\ra_\cH - \bxi(\kappa).
\end{align*}
Taking supremum over $\kappa\in\cH_+\cap L^\infty$, we obtain $\int_0^1{\xi}^*(\iota(s))\d s\geq \bxi^\oplus(\iota)$ for every $\iota\in\cH$.

For the other direction, fix any $\iota\in\cH$. Note that ${\xi}^*$ is lower-semicontinuous and ${\xi}^*(\iota)\geq -{\xi}(0)$.
Using Lemma~\ref{l.basics_(j)}~\eqref{i.cvg}, we can extract from $\Junif$ a sequence $(j_n)_{n=1}^\infty$ satisfying $\lim_{n\to\infty}\iota^{(j_n)}=\iota$ a.e.\ on $[0,1)$. Using the lower semi-continuity of $\xi^*$ and Fatou's lemma, we get
\begin{align*}
    \int_0^1{\xi}^*(\iota(s))\d s\leq \int_0^1 \liminf_{n\to \infty} {\xi}^*\left(\iota^{(j_n)}(s)\right)\d s\leq \liminf_{n\to \infty}\int_0^1 {\xi}^*\left(\iota^{(j_n)}(s)\right)\d s.
\end{align*}
Recall the definitions of $\iota^\j$ in~\eqref{e.(j)} and $\pj_j\iota$ in \eqref{e.def_pj_iota}. For every $j\in\Junif$, we can compute
\begin{align*}
    \int_0^1 {\xi}^*\left(\iota^\j(s)\right)\d s=\sum_{k=1}^{|j|}\frac{1}{|j|}{\xi}^*((\pj_j\iota)_k)= \sum_{k=1}^{|j|}\frac{1}{|j|}\sup_{x_k \in \R_+}\left\{x_k \cdot (\pj_j\iota)_k - {\xi}(x_k)\right\}
    \\
    = \sup_{x\in \cH^j_+}\left\{\la x,\pj_j\iota\ra_{\cH^j}-\int_0^1{\xi}(\lf_jx(s))\d s\right\} = \sup_{x\in \cH^j_+}\{\la \lf_j x,\iota\ra_{\cH}-\bxi(\lf_j x)\}\leq \bxi^\oplus(\iota),
\end{align*}
where $\cH^j_+$ stands for $\pj_j(\cH_+) = \{x\in \cH^j:x_k\in\R_+,\,\forall k\}$.
The above two displays together yield $\int_0^1{\xi}^\ast(\iota(s))\d s\leq \bxi^\oplus(\iota)$ for every $\iota\in\cH$, verifying~\eqref{e.H^oplus=E_xi^*}. 

The desired result follows from~\eqref{e.H^*=H^oplus} and~\eqref{e.H^oplus=E_xi^*}.
\end{proof}

Now, we are ready to prove the following.

\begin{proposition}[Hopf-Lax formula in one dimension]\label{p.equival_exp_hopf_lax}
Let $D=1$.
Suppose that $\psi:\C\to\R$ is $\C^*$-nondecreasing and continuous and that $\xi:\R\to\R$ is increasing on $\R_+$. 
Then, for every $t\geq 0$ and $\mu\in\C\cap L^\infty$,
\begin{align}\label{e.hopf_lax_xi}
    \sup_{\nu \in \C\cap L^\infty}\inf_{\rho\in\C\cap L^\infty}\left\{\psi(\mu+\nu)- \la \nu,\rho\ra_{\cH}+t\int_0^1\xi(\rho(s))\d s\right\}  \notag
    \\= \sup_{\nu \in \C\cap L^\infty}\left\{\psi(\nu)-t\int_0^1{\xi}^*\left(\frac{\nu(s)-\mu(s)}{t}\right) \d s\right\}.
\end{align}

\end{proposition}
\begin{proof}
Let us denote the left-hand side in \eqref{e.hopf_lax_xi} by $\LHS$ and the right-hand side by $\RHS$. By Lemma~\ref{l.H^circledast=E_xi^*} (together with changing $\nu$ to $t\nu$ and then $\mu+t\nu$ to $\nu$),
\begin{align*}
    \LHS = \sup_{\nu \in \C\cap L^\infty}\left\{\psi(\mu+t\nu)- t \int_0^1\xi^*(\nu(s))\d s\right\} = \sup_{\nu \in \mu+ \C\cap L^\infty}\left\{\psi(\nu)-t\int_0^1{\xi}^*\left(\frac{\nu(s)-\mu(s)}{t}\right) \d s\right\}.
\end{align*}
It is clear that $\LHS\leq \RHS$.
We only need to show $\RHS\leq \LHS$. Since $\xi$ is nondecreasing on $\R_+$, the definition of $\xi^*$ implies $\xi^*(r) = \xi^*(0)$ for all $r<0$.
For every $\kappa\in\cH$, we define $\kappa_+$ by $\kappa_+(s) = (\kappa(s))\vee 0$ for all $s\in[0,1)$. Then, we have
\begin{align*}
    \RHS = \sup_{\nu\in\C\cap L^\infty}\left\{\psi(\nu)-t\int_0^1{\xi}^*\left(\frac{(\nu-\mu)_+(s)}{t}\right)\d s\right\}.
\end{align*}
For every $\eps>0$, we can find $\nu \in \C\cap L^\infty$ such that
\begin{align*}
    \RHS \leq \psi(\nu)-t\int_0^1{\xi}^*\left(\iota(s)\right)\d s + \eps,
\end{align*}
where we set $\iota=\frac{1}{t}(\nu-\mu)_+\in\cH_+$. We choose a sufficiently fine $j\in\Junif$ satisfying $\psi(\nu)\leq \psi\left(\nu^\j\right)+\eps$. Since ${\xi}^*$ is convex, we have $\int_0^1{\xi}^*\left(\iota^\j(s)\right)\d s\leq \int_0^1{\xi}^*(\iota(s))\d s$ by Jensen's inequality. Hence, the above display becomes
\begin{align*}
    \RHS \leq \psi\left(\nu^\j\right)-t\int_0^1{\xi}^*\left(\iota^\j(s)\right)\d s + 2
    \eps.
\end{align*}
Setting $\rho = \mu + t\iota^\j_\sharp$, we have
\begin{align*}
    \rho - \nu^\j = \left(\mu -\mu^\j\right) + t\left(\iota^\j_\sharp -\iota^\j\right) + \left(t\iota^\j - (\nu-\mu)^\j\right) \in \C^*
\end{align*}
where $\mu -\mu^\j\in\C^*$ due to Lemma~\ref{l.proj_cones}~\eqref{l.mu-mu^j}, $\iota^\j_\sharp -\iota^\j\in\C^*$ due to Lemma~\ref{l.nond_rearrang}~\eqref{i.iota^j_sharp_in_iota^j+C*}, and $t\iota^\j - (\nu-\mu)^\j$ due to $\cH_+\subset \C^*$. As $\iota\in\cH_+$, Lemma~\ref{l.nond_rearrang}~\eqref{i.iota^j_sharp_in_C} implies $\rho \in\C$ and $\rho-\mu\in\C$, which are also in $L^\infty$. Since $\psi$ is $\C^*$-nondecreasing, the above display gives $\psi\left(\nu^\j\right)\leq \psi(\rho)$. 
Lemma~\ref{l.nond_rearrang}~\eqref{i.Eh(iota^j)} also gives $\int_0^1{\xi}^*\left(\frac{\rho-\mu}{t}\right)=\int_0^1{\xi}^*\left(\iota^\j_\sharp\right) = \int_0^1{\xi}^*\left(\iota^\j\right)$.
Combining these, we obtain
\begin{align*}
    \RHS \leq \psi(\rho)- t \int{\xi}^*\left(\frac{\rho(s)-\mu(s)}{t}\right)\d s+2\eps  \leq \LHS + 2\eps.
\end{align*}
Sending $\eps\to0$, we obtain the desired result.
\end{proof}

\section{Fenchel--Moreau identity on cones}\label{s.fenchel-moreau}

Recall Definition~\ref{d.fenchel_moreau_prop} of the Fenchel--Moreau property. 
To apply Theorem~\ref{t.hj_cone}~\eqref{i.main_var_rep} to equations on $\R_+\times \C^j$, $j\in\J$, we need to show that $\C^j$ given in \eqref{e.C^j} has the Fenchel--Moreau property. Adapting the definition of monotone conjugate in \eqref{e.def_u*} to $\C^j$ with ambient Hilbert space $\cH^j$ given in \eqref{e.def_H^j}, in this section, for any $g:\C^j\to(-\infty,\infty]$, we set
\begin{align}\label{e.g^*_C^j}
     g^*(y) = \sup_{x\in \C^j}\{\la x,y\ra_{\cH^j}-g(x)\},\quad\forall y \in \cH^j,
\end{align}
and $g^{**} = (g^*)^*$, where $g^*$ is understood to be its restriction to $\C^j$.

\begin{proposition}\label{p.fenchel-moreau}
For every $j\in\J$, the closed convex cone $\C^j$ possesses the Fenchel--Moreau property: for $g:\cM^j\to (-\infty,\infty]$ not identically equal to $\infty$, we have $g^{**}=g$ if and only if $g$ is convex, lower semicontinuous and $(\cM^j)^*$-increasing.
\end{proposition}

The proof largely follows the steps in \cite{chen2020fenchel}. We first recall the basic results of convex analysis. Then, we show Lemma~\ref{l.biconj_interior} which treats the case where the effective domain of $g$ has a nonempty interior. Finally, in the last subsection, we extend the result to the general case.

\subsection{Basic results of convex analysis}

For $a\in\cH^j$ and $\nu\in\R$, we define the affine function $ L_{a,\nu}$ with slope $a$ and translation $\nu$ by
\begin{align*}L_{a,\nu}(x)=\la a,x\ra_{\cH^j}+\nu,\quad \forall x\in \cH^j.
\end{align*}
For a function $g:\cE\to(-\infty,\infty]$ defined on a subset $\cE\subset \cH^j$, we can extend it in the standard way to $g:\cH^j\to(-\infty,\infty]$ by setting $g(x)=\infty$ for $x\not\in\cE$. For $g:\cH^j\to(-\infty,\infty]$, we define its \textit{effect domain} by
\begin{align*}\dom g = \left\{x\in\cH^j:\ g(x)<\infty\right\}.
\end{align*}
Henceforth, we shall not distinguish functions defined on $\C^j$ from their standard extensions to $\cH^j$. We denote by $\Gamma_0(\cE)$ the collection of convex and lower semicontinuous functions from $\cE\subset \cH^j$ to $(-\infty,\infty]$ with nonempty effect domain.

\smallskip

For $g:\cH^j\to(-\infty,\infty]$ and each $x\in\cH^j$, recall that the subdifferential of $g$ at $x$ is given by
\begin{equation*}\partial g(x)=\left\{z\in\cH^j:g(y)\geq g(x)+\la z, y-x\ra_{\cH^j},\,\forall y\in\cH^j\right\}.
\end{equation*}
The effective domain of $\partial g$ is defined to be
\begin{align*}
    \dom \partial g = \left\{x\in\cH^j:\ \partial g(x)\neq\emptyset\right\}.
\end{align*}

\smallskip

We now list some lemmas needed in our proofs. 
\begin{lemma}\label{lemma:itrconvex}
For a convex set $\cE\subset \cH^j$, if $y\in\mathsf{cl}\,\cE$ and $y'\in\mathsf{int}\,\cE$, then $\lambda y+(1-\lambda)y'\in\mathsf{int}\,\cE$ for all $\lambda \in[0,1)$.
\end{lemma}
\begin{lemma}\label{lemma:itrsubdif}
For $g\in\Gamma_0(\cH^j)$, it holds that $\itr\dom g\subset \mathsf{dom}\, \partial g\subset \dom g$.
\end{lemma}
\begin{lemma}\label{lemma:lineconv}
Let $g\in\Gamma_0(\cH^j)$, $x\in\cH^j$ and $y\in\mathsf{dom}\,g$. For every $\alpha\in(0,1)$, set $x_\alpha=(1-\alpha)x+\alpha y$. Then $\lim_{\alpha\to 0} g(x_\alpha)=g(x)$.
\end{lemma}

\begin{lemma}\label{lemma:supsubdif}
Let $g\in\Gamma_0(\C^j)$, $x\in\C^j$ and $y\in \C^j$. If $y\in\partial g(x)$, then $g^*(y)=\la x,y\ra_{\cH^j}-g(x)$.
\end{lemma}

\begin{lemma}\label{lemma:hyperplane}
For $g\in\Gamma_0(\C^j)$ and $x\in\C^j$, we have
\begin{align*}g^{**}(x)=\sup L_{a,\nu}(x)
\end{align*}
where the supremum is taken over
\begin{align}\label{e.sup_set}
    \{(a,\nu)\in \C^j\times\R: L_{a,\nu}\leq g\text{ on }\C^j\}.
\end{align}
\end{lemma}

For, Lemmas \ref{lemma:itrconvex}, \ref{lemma:itrsubdif}, and  \ref{lemma:lineconv}, we refer to \cite[Propositions~3.35, 16.21, and 9.14]{bauschke2011convex}. 
Here, let us prove Lemma~\ref{lemma:supsubdif} and Lemma~\ref{lemma:hyperplane}.

\begin{proof}[Proof of Lemma~\ref{lemma:supsubdif}]
By the standard extension, we have $g\in\Gamma_0(\cH^j)$. Since $y\in\partial g(x)$, it is classically known (c.f.\ \cite[Theorem 16.23]{bauschke2011convex}) that
\begin{align*}
    \sup_{z\in\cH^j}\left\{\la z, y \ra_{\cH^j}- g(z)\right\} = \la x, y \ra_{\cH^j} -g(x).
\end{align*}
By assumption, we know $x\in\dom\partial g$. Hence, Lemma~\ref{lemma:itrsubdif} implies $x\in \dom g$ and thus both sides above are finite. On the other hand, by the extension, we have $g(z)=\infty$ if $z\not\in \C^j$, which yields
\begin{align*}
    \sup_{z\in\cH^j}\left\{\la z, y \ra_{\cH^j}- g(z)\right\} = \sup_{z\in\C^j}\left\{\la z, y \ra_{\cH^j}- g(z)\right\} = g^*(y).
\end{align*}
The desired result follows from the above two displays.
\end{proof}

\begin{proof}[Proof of Lemma~\ref{lemma:hyperplane}]
For each $y\in\C^j$,
\begin{align*}
    L_{y,\, -g^*(y)}(x) = \la y,x\ra_{\cH^j} - g^*(y),\quad \forall x \in \C.
\end{align*}
is an affine function with slope $y\in\C^j$. By~\eqref{e.g^*_C^j}, we can see that $L_{y,\,-g^*(y)}\leq g$ on $\C^j$. In view of the definition of $g^{**}$, we have $g^{**}(x)\leq \sup L_{a,\nu}(x)$ for all $x\in\C^j$ where the $\sup$ is taken over the collection in~\eqref{e.sup_set}.

\smallskip

For the other direction, if $(\alpha,\nu)$ belongs to the set in \eqref{e.sup_set}, we have
\begin{align*}
    \la a, x\ra_{\cH^j} +\nu \leq g(x),\quad\forall x\in\C^j.
\end{align*}
Rearranging and taking supremum in $x\in\C^j$, we get $g^*(a)\leq -\nu$. This yields
\begin{align*}
    L_{a, \nu}(x)\leq \la a, x\ra_{\cH^j} -g^*(a)\leq g^{**}(x),
\end{align*}
which implies $ \sup L_{a,\nu}(x)\leq g^{**}(x)$.
\end{proof}

The proof of Proposition~\ref{p.fenchel-moreau} consists of two parts. The first part, summarized in the lemma below, concerns the case where $\dom g$ has a non-empty interior.

\begin{lemma}\label{l.biconj_interior}
If $\itr\dom g\neq\emptyset$, then $g^{**}=g$ if and only if $g$ is convex, lower semicontinuous and $(\cM^j)^*$-increasing.
\end{lemma}

The next subsection is devoted to its proof. The second part deals with the case where $\dom g$ has an empty interior. For this, we need a more careful analysis of the structure of the boundary of $\C^j$. This is done in the second subsection.

\subsection{Proof of Lemma~\ref{l.biconj_interior}}Let satisfy $\itr\dom g\neq \emptyset$.
It is clear that $g^{**}$ is convex, lower semicontinuous, and $(\C^j)^*$-increasing.

Henceforth, assuming that $g$ is convex, lower semicontinuous, and $(\C^j)^*$-increasing, we want to prove the converse. For convenience, we write $\Omega = \dom g$. The plan is to prove the identity $g=g^{**}$ first on $\itr\Omega$, then on $\cl\Omega$, and finally on the entire $\C$.

\subsubsection{Analysis on $\itr\Omega$}

Let $x\in \itr \Omega$. By Lemma~\ref{lemma:itrsubdif}, we know $\partial g(x)$ is not empty. For each $v\in (\C^j)^*$, there is $\eps>0$ small so that $x -\eps v \in \Omega$. For each $y\in \partial g(x)$, by the definition of subdifferentials and the monotonicity of $g$, we have
\begin{equation*}
    \la v,y\ra_{\cH^j}\geq \frac{1}{\eps}\left( g(x)-g(x-\eps v)\right)\geq 0,
\end{equation*}
which implies $\emptyset \neq \partial g(x)\subset \C^j$.
Invoking Lemma~\ref{lemma:supsubdif}, we can deduce
\begin{align*}
    g(x)\leq\sup_{y\in\C^j}\{\la y,x\ra_{\cH^j}-g^*(y)\}=g^{**}(x).
\end{align*}
On the other hand, from the definition of $g^{**}$,
it is easy to see that
\begin{align}\label{e.u>u**}
    g(x)\geq g^{**}(x),\quad\forall x\in\C^j.
\end{align}
Hence, we obtain
\begin{align*}g(x)=g^{**}(x),\quad\forall x\in\itr\Omega.
\end{align*}

\subsubsection{Analysis on $\mathsf{cl}\,\Omega$}
Let $x\in\cl\Omega$ and choose $y\in\itr\Omega$. Setting $x_\alpha= (1-\alpha)x+\alpha y$, by Lemma~\ref{lemma:itrconvex}, we have $x_\alpha\in\itr\Omega$ for every $\alpha \in (0,1]$. By the result on $\itr\Omega$, we have
\begin{align*}
    g(x_\alpha) = g^{**}(x_\alpha).
\end{align*}
Then, $x_\alpha$ belongs to $\dom g$ and $\dom g^{**}$. Applying Lemma~\ref{lemma:lineconv} and sending $\alpha \to 0$, we get
\begin{align}\label{e.u=u**_on_cl}
    g(x)=g^{**}(x),\quad\forall x\in \cl\Omega.
\end{align}

\subsubsection{Analysis on $\C^j$}

Due to~\eqref{e.u=u**_on_cl}, we only need to consider points outside $\cl\Omega$. Fixing any $x\in \C^j\setminus\mathsf{cl}\,\Omega$, we have $g(x)=\infty$. Since $f$ is not identically equal to $\infty$ and $(\C^j)^*$-increasing, we must have $0\in\Omega$.
By this, $x\notin\mathsf{cl}\,\Omega$ and the convexity of $\cl\Omega$, we must have
\begin{equation}\label{e.lambda'<1}
    \bar\lambda=\sup\{\lambda\in\R_+:\lambda x\in\mathsf{cl}\,\Omega\}<1.
\end{equation}
We set
\begin{align}\label{e.x'}
    \bar x=\bar\lambda x.
\end{align}
Then, we have that $\bar x\in\bd\Omega$ and $\lambda \bar x\notin \mathsf{cl}\,\Omega$ for all $\lambda>1$. 

\smallskip
We need to discuss two cases: $\bar x\in\Omega$ or not.

\smallskip
In the second case where $\bar x\notin \Omega$, we have $g(\bar x)=\infty$. Due to $\bar x\in\cl\Omega$ and~\eqref{e.u=u**_on_cl}, we have $g^{**}(\bar x)=\infty$. On the other hand, by~\eqref{e.u=u**_on_cl} and the fact that $0\in\Omega$, we have $g^{**}(0)=g(0)$ and thus $0\in \dom g^{**}$. The convexity of $g^{**}$ implies that
\begin{align*}
    \infty = g^{**}(\bar x)\leq \bar\lambda g^{**}(x)+ (1-\bar\lambda)g^{**}(0).
\end{align*}
Hence, we must have $g^{**}(x)=\infty$ and thus $g(x)=g^{**}(x)$ for such $x$.

We now consider the case where $\bar x\in\Omega$. For every $y\in\cH^j$, the outer normal cone to $\Omega$ at $y$ is defined by
\begin{equation}\label{e.ndef}
    \nn_\Omega (y)=\{z\in \cH^j:\la z, y'-y\ra_{\cH^j}\leq 0, \,\forall y'\in\Omega\}.
\end{equation}
We need the following result.

\begin{lemma}\label{lemma:outernormal}
Assume $\itr\Omega\neq \emptyset$. For every $y\in \Omega\setminus\mathsf{int}\,\Omega$ satisfying $\lambda y\notin \mathsf{cl}\,\Omega$ for all $\lambda>1$, there is $z\in \nn_\Omega(y)\cap \C^j$ such that $\la z, y\ra_{\cH^j} >0$.
\end{lemma}

By Lemma~\ref{lemma:outernormal} applied to $\bar x\in\Omega$, there is $z\in \C^j$ such that
\begin{gather}\label{e.znx'}
\la z, w-\bar x\ra_{\cH^j}\leq 0,\quad\forall w\in\Omega,\\
\label{e.zx'>0}\la z,\bar x\ra_{\cH^j}>0.
\end{gather}

\smallskip

The monotonicity of $g$ ensures that $g(x)\geq g(0)$ for all $x\in\C^j$. For each $\rho\geq 0$, define
\begin{equation*}
     \cL_\rho= L_{\rho z,\ g(0)-\rho\la z,\bar x\ra_{\cH^j}}.
\end{equation*}
Due to~\eqref{e.znx'}, we can see that
\begin{align*}
     \cL_\rho(w)&=\rho \la z, w-\bar x\ra_{\cH^j} + g(0)\leq   g(w),\quad\forall w\in\Omega.
\end{align*}
Since we know $f\big|_{\C^j\setminus\Omega}=\infty$, the inequality above gives
\begin{align}\label{e.cL_rho<f}
    \cL_\rho\leq g,\quad\forall \rho\geq 0.
\end{align}
Evaluating $\cL_\rho$ at $x$ and using~\eqref{e.x'},
we have
\begin{align*}
     \cL_\rho(x)&=\rho\la z,x-\bar x\ra_{\cH^j} + g(0)= \rho\left({\bar\lambda}^{-1}-1\right)\la z,\bar x\ra_{\cH^j} + g(0).
\end{align*}
By~\eqref{e.lambda'<1} and~\eqref{e.zx'>0}, we obtain
\begin{align*}
    \lim_{\rho\rightarrow\infty} \cL_\rho(x)=\infty.
\end{align*}
This along with~\eqref{e.cL_rho<f}, Lemma~\ref{lemma:hyperplane} and~\eqref{e.u>u**} implies
\begin{align*}
    g(x)=g^{**}(x)\quad\forall x \in \C^j\setminus\cl\Omega.
\end{align*}
In view of this and~\eqref{e.u=u**_on_cl}, we have completed the proof of Lemma~\ref{l.biconj_interior}. It remains to prove Lemma~\ref{lemma:outernormal}.

\begin{proof}[Proof of Lemma~\ref{lemma:outernormal}]

Fix $y$ satisfying the condition. Since it is possible that $y\not\in\itr\C^j$, we want to approximate $y$ by points in $\mathsf{bd}\,\Omega\cap \mathsf{int}\,\C^j$. For every open ball $B\subset \cH^j$ centered at $y$, there is some $\lambda>1$ such that $y'=\lambda y\in \C^j\cap (B\setminus \mathsf{cl}\,\Omega)$. Due to $\mathsf{int}\,\Omega\neq \emptyset$ and $y\in\Omega$, by Lemma~\ref{lemma:itrconvex}, there is some $y''\in B\cap \mathsf{int}\,\Omega\subset \mathsf{int}\,\C$. For $\rho\in[0,1]$, we set
\begin{equation*}
    y_\rho= \rho y'+(1-\rho)y''\in B.
\end{equation*}
Then, we take
\begin{align*}
    \rho_0=\sup\{\rho\in[0,1]:y_\rho\in\mathsf{int}\,\Omega\}.
\end{align*}
Since $y'\notin\mathsf{cl}\,\Omega$, we must have $\rho_0<1$. It can be seen that $y_{\rho_0}\in\mathsf{cl}\,\Omega\setminus\mathsf{int}\,\Omega$ and thus $y_{\rho_0}\in B\cap\mathsf{bd}\,\Omega$. Due to $y'\in\C^j$, $y''\in\mathsf{int}\,\C^j$ and Lemma~\ref{lemma:itrconvex}, we have $y_{\rho_0}\in \mathsf{int}\,\C^j$. In summary, we obtain $y_{\rho_0}\in B\cap\mathsf{bd}\,\Omega\cap \mathsf{int}\,\C^j$.

By this construction and varying the size of the open balls centered at $y$, we can find a sequence $(y_n)_{n=1}^\infty$ such that
\begin{gather}
    y_n\in \itr \C^j,\label{e.y_n_int}\\
    y_n\in\mathsf{bd}\,\Omega,\label{e.y_n_bd}\\
    \lim_{n\rightarrow \infty}y_n=y.\label{e.y_n_lim}
\end{gather}

\smallskip

Fix any $n$. By~\eqref{e.y_n_int}, there is $\delta >0$ such that
\begin{align}\label{e.y_n+B_in_C}
    y_n+ B(0,2\delta) \subset\C^j.
\end{align}
Here, for $a\in\cH^j,r>0$, we write $B(a,r)=\{z\in\cH^j:\ |z-a|<r\}$.
For each $\eps\in(0,\delta)$, due to~\eqref{e.y_n_bd}, we can also find $y_{n,\eps}$ such that
\begin{gather}
y_{n,\eps}\in\Omega\label{e.y_n,eps_itr},\\
    |y_{n,\eps}-y_n|<\eps.\label{e.y_n,eps-y_n}
\end{gather}
This and~\eqref{e.y_n+B_in_C} imply that
\begin{align*}
    y_{n,\eps}-a \in \C^j,\quad \forall \eps \in(0,\delta),\ a \in B(0,\delta).\end{align*}
Since $g$ is $(\C^j)^*$-increasing, this along with~\eqref{e.y_n,eps_itr} implies that
\begin{align*}
    y_{n,\eps}-a\in\Omega,\quad \forall \eps\in(0,\delta),\ a\in (\C^j)^*\cap B(0,\delta).
\end{align*}
Due to~\eqref{e.y_n_bd} and $\itr\Omega\neq \emptyset$, we have that $\nn_\Omega(y_n)$ contains some nonzero vector $z_n$ (see \cite[Proposition~6.45]{bauschke2011convex} together with \cite[Proposition~6.23~(iii)]{bauschke2011convex}). The definition of the outer normal cone in~\eqref{e.ndef} yields
\begin{align*}
    \la z_n,y_{n,\eps}-a-y_n\ra_{\cH^j}\leq 0,
\end{align*}
which along with~\eqref{e.y_n,eps-y_n} implies
\begin{equation*}
    \la z_n, a\ra_{\cH^j}\geq -|z_n|\eps.
\end{equation*}
Sending $\eps\rightarrow 0$ and varying $a\in(\C^j)^*\cap B(0,\delta)$, we conclude that
\begin{align}\label{e.z_n_in_n_cap_C}
    z_n\in\nn_\Omega(y_n)\cap \C^j,\quad \forall n.
\end{align}

\smallskip

Now for each $n$, we rescale $z_n$ to get $|z_n|=1$. 
By passing to a subsequence, we can assume that there is $z\in\C^j$ such that $z_n$ converges to $z$. By $z_n\in\nn_\Omega(y_n)$, we get
\begin{equation*}\la z_n,w-y_n\ra_{\cH^j}\leq 0,\quad \forall w\in\Omega. 
\end{equation*}
The convergence of $(z_n)_{n=1}^\infty$ along with~\eqref{e.y_n_lim} implies
\begin{align*}
    \lim_{n\rightarrow \infty}\la z_n,w-y_n\ra_{\cH^j}=\la z,w-y\ra_{\cH^j},\quad \forall w\in\Omega.
\end{align*}
The above two displays yield $z\in\nn_\Omega(y)\cap \C^j$.

\smallskip

Then, we show $\la z, y\ra_{\cH^j} >0$.
Fix some $x_0\in\itr\Omega$ and some $\eps>0$ such that $B(x_0,2\eps)\subset \Omega$. Let $y_n$ and $z_n$ be given as in the above. Due to $|z_n|=1$, we have
\begin{align*}
    x_0-\eps z_n\in\Omega\subset \C^j.
\end{align*}
Since it is easy to see that $\C^j\subset (\C^j)^*$, by \eqref{e.z_n_in_n_cap_C}, we have $z_n\in(\C^j)^*$, which along with the above display implies that
\begin{align*}
    \la x_0-\eps z_n,z_n\ra_{\cH^j} \geq 0
\end{align*}
and thus $\la x_0,z_n\ra_{\cH^j}\geq \eps$. Using $z_n\in \nn_\Omega(y_n)$, we obtain
\begin{align*}
    \la y_n,z_n\ra_{\cH^j} \geq \la x_0,z_n\ra_{\cH^j} \geq \eps.
\end{align*}
Passing to the limit, we conclude that $\la z,y\ra_{\cH^j}>0$, completing the proof.
\end{proof}

\subsection{Proof of Proposition~\ref{p.fenchel-moreau}}

Similar to the arguments at the beginning of the proof of Lemma~\ref{l.biconj_interior}, we only need to show the direction that $g^{**}=g$ if $g$ is convex, lower semicontinuous and $(\C^j)^*$-increasing. By Lemma~\ref{l.biconj_interior}, we only need to consider the case where $\Omega$ has an empty interior. Recall that we have set $\Omega = \dom g$. Throughout this subsection, we assume that $\Omega$ has an empty interior. We proceed in steps.

Step~1. 
Setting
\begin{align}\label{e.N=rank}
    N = \max\left\{\rank\left(x_{|j|}\right):x\in\Omega\right\},
\end{align}
we want to show $N<D$. We need the following lemma.
\begin{lemma}\label{l.full_rank}
If there is $x\in \cM^j$ such that $x_{|j|}$ is of full rank, then $\itr(\cM^j\cap(x-(\cM^j)^*))\neq \emptyset$.
\end{lemma}
\begin{proof}
Recall the partial order induced by $\S^D_+$ in \eqref{e.a>b_SD_+}.
Let $x$ satisfy the assumption. Then, there is some constant $a>0$ such that $x_{|j|}\geq a I_D$ where $I_D$ is the $D\times D$ identity matrix. Let us define $y_k=k\delta I_D$, $k=1,2,\dots,|j|$, for some $\delta>0$ to be chosen later. Then, it is clear that $y\in \cM^j$. We consider $B=\{z\in \cH^j: |z_k-y_k|\leq r,\,\forall k\}$ for some $r>0$ to be chosen later. Then, due to finite dimensionality, there is some $c>0$ such that, for every $z\in B$,
\begin{align*}
    -crI_D\leq z_k-y_k\leq crI_D,\quad\forall k=1,2,\dots,|j|.
\end{align*}
Using this, we can show that, for every $z\in B$,
\begin{align*}
    z_{k}-z_{k-1}\geq y_{k}-y_{k-1}-2crI_D=(\delta-2cr)I_D,\quad\forall k=1,2,\dots,|j|,
\end{align*}
where we set $z_0=y_0=0$. By choosing $r$ sufficiently small, the above is in $\S^D_+$, and we have $B\subset \cM^j$. On the other hand, we also have, for $i=1,2,\dots,|j|$,
\begin{align*}
    \sum_{k=i}^{|j|}(x_k-z_k) =  \sum_{k=i}^{|j|}(x_k-y_k+y_k-z_k)\geq \left(x_{|j|}-\sum_{k=1}^{|j|}y_k\right) + \sum_{k=i}^{|j|}(y_k-z_k)
    \\
    \geq \left(aI_D-\sum_{k=1}^{|j|}y_k\right)-|j|crI_D = \left(a-\frac{1}{2}(1+|j|)|j|\delta-|j|cr\right)I_D,
\end{align*}
which is in $\S^D_+$ if $\delta$ and $r$ are chosen sufficiently small. Hence, we have $x-z\in (\C^j)^*$ for all $z\in B$, which is equivalent to $B\subset x-(\cM^j)^*$. Since $B$ has a nonempty interior, the proof is complete.
\end{proof}

Since $g$ is $(\cM^j)^*$-increasing, we have
\begin{align*}
    \cM^j\cap (x-(\cM^j)^*)\subset \Omega,\quad\forall x\in \Omega.
\end{align*}
Hence, Lemma~\ref{l.full_rank} implies that if there is $x\in\Omega$ with $\rank(x_{|j|})=D$, then $\itr\Omega\neq \emptyset$. Therefore, under our assumption $\itr\Omega=\emptyset$, we must have that $x_{|j|}$ is of rank less than $D$ for every $x\in\Omega$. So, for $N$ defined in \eqref{e.N=rank}, we must have $N<D$.

Step~2.
We fix $\tilde x\in \Omega$ with rank $N$. By changing basis, we may assume $\tilde x=\diag(a,0_{D-N})$ where $a$ is a $N\times N$ diagonal matrix with positive entries and $0_{D-N}$ is $(D-N)\times(D-N)$ zero matrix. We set
\begin{gather*}
    \tilde \S^N_+=\{\diag(a,0_{D-N}):a\in\S^N_+\}\subset \S^D_+,
    \\
    \tilde \cM^j=\{x\in \cM^j:x_k\in \tilde\S^N_+,\,\forall k\}.
\end{gather*}
We want to show that
\begin{align}\label{e.Omega_in_tilde_cM^j}
    \Omega\subset \tilde \cM^j.
\end{align}
We argue by contradiction and assume that there is $y\in \Omega$ such that $y_{k'}\not\in \tilde\S^N_+$ for some $k'$. Let us define $\tilde y$ by $\tilde y_k = y_k$ for all $k\leq k'$ and $\tilde y_k=y_{k'}$ for all $k>k'$. We clearly have $\tilde y\in \cM^j\cap (y-(\cM^j)^*)$ which implies that $\tilde y\in\Omega$ (because $g$ is $(\C^j)^*$-increasing). By convexity of $\Omega$, we must have $z=\frac{1}{2}\tilde x+\frac{1}{2}\tilde y\in \Omega$. 

We argue that $z_{|j|}$ has rank at least $N+1$. Since $y_{k'}$ is positive semi-definite, we must have
\begin{align}\label{e.y_k'_ii>0}
    (y_{k'})_{ii}> 0
\end{align}
for some $i>N$. By reordering coordinates, we may assume $i=N+1$ in \eqref{e.y_k'_ii>0} and thus $(y_{k'})_{N+1,N+1}> 0$. Setting $\hat z_{|j|} = ((z_{|j|})_{m,n})_{1\leq m,n\leq |j|}$, it suffices to verify $v^\intercal z_{|j|}v>0$ for all $v\in \R^{N+1}\setminus\{0\}$. We define $\hat x_{|j|}$ and $\hat y_{|j|}$ analogously. If $v_n\neq0$ for all $n=1,2,\dots, N$, we have
\begin{align*}
    v^\intercal \hat z_{|j|}v\geq \frac{1}{2}v^\intercal \hat x_{|j|}v>0
\end{align*}
due to the fact that $\hat x_{|j|}=a$ is a diagonal matrix with positive entries.
If $v_n=0$ for all $n=1,2,\dots, N$, then we must have $v_{N+1}\neq 0$ and thus
\begin{align*}
    v^\intercal \hat z_{|j|}v\geq \frac{1}{2}v^\intercal \hat y_{|j|}v = \frac{1}{2}v^2_{N+1}(y_{|j|})_{N+1,N+1}>0.
\end{align*}
We conclude that $z_{|j|}$ has rank at least $N+1$ contradicting the definition of $N$. Therefore, we must have~\eqref{e.Omega_in_tilde_cM^j}.

Step~3.
We conclude by applying Lemma~\ref{l.biconj_interior} to $g$ restricted to $\tilde\C^j$, and treating $g$ on $\C^j\setminus\tilde\C^j$ using Lemma~\ref{lemma:hyperplane}.

In view of~\eqref{e.N=rank} and~\eqref{e.Omega_in_tilde_cM^j}, applying Lemma~\ref{l.full_rank} to $\tilde\C^j$, we have that $\Omega$ has nonempty interior relative to $\tilde\C^j$. Let $\tilde g$ be the restriction of $g$ to $\tilde\C^j$. Define
\begin{align*}
    \tilde g^{\tilde *}(y)=\sup_{x\in \tilde\C^j}\{\la x, y \ra_{\cH^j} - \tilde g(x)\},\quad\forall y  \in \tilde\cH^j
\end{align*}
where $\tilde \cH^j =\{x\in\cH^j: x_k\in \tilde\S^N,\ \forall k\}$ with $\tilde\S^N=\{\diag(a,0_{K-N}):a\in\S^N\}$. Since $g(x) = \infty$ for $x\not\in \tilde\C^j$ and $g=\tilde g$ on $\tilde\C^j$, we can see from the definition of $g^*$ in \eqref{e.g^*_C^j} that
\begin{align*}
    g^*(y) = \sup_{x\in\tilde \C^j}\{\la x,y\ra_{\cH^j}- g(x)\} = g^{\tilde *}(y),\quad\forall y\in\tilde\cH^j,
\end{align*}
which implies $g^{**}(x) \geq \tilde g^{\tilde *\tilde *}(x)$ for all $x\in \tilde\C^j$. Since Lemma~\ref{l.biconj_interior} implies that $\tilde g(x) = \tilde g^{\tilde * \tilde *}(x)$ for $x\in \tilde\C^j$, we can thus conclude that $g^{**}(x) \geq \tilde g(x) = g(x)$ for all $x\in\tilde\C^j$. This along with~\eqref{e.u>u**} yields
\begin{align}\label{e.u**=u_tilde_C}
    g^{**}(x) = g(x),\quad\forall x \in \tilde\C^j.
\end{align}

For $x\in \cM^j\setminus \tilde \cM^j$, arguing as above (the paragraph studying the rank of $z_{|j|}$), we can see that there is some $k$ and some $i>N$ such that $(x_{k})_{ii}>0$. Now, setting $y_k=\diag(0_N,I_{K-N})$ for every $k$, we have $y\in \cM^j$, $\la y,x\ra_{\cH^j}>0$ and $\la y,z\ra_{\cH^j}=0$ for all $z\in \tilde \cM^j$. We define $\cL_\rho = \rho\la y,\cdot\ra_{\cH^j}+g(0)$ for each $\rho>0$. Since $g(z)\geq g(0)$ for all $z\in\C^j$ due to the monotonicity of $g$, and since $\cL_\rho(z) = g(0)$ for all $z\in \tilde\C^j$, we have $g(z) \geq \cL_\rho(z)$ for all $z\in\tilde\C^j$. Due to $g =\infty$ on $\C^j\setminus\tilde\C^j$, we thus get 
\begin{align*}
    g(z) \geq \cL_\rho(z),\quad\forall z\in \C^j.
\end{align*}
Due to $\la y, x\ra_{\cH^j}>0$, we also have $\lim_{\rho\to\infty}\cL_\rho(x) =\infty = g(x)$. In view of Lemma~\ref{lemma:hyperplane}, this along with the above display implies that $g^{**}(x) = g(x)$ for all $x\in \C^j\setminus \tilde\C^j$, which together with~\eqref{e.u**=u_tilde_C} completes the proof of Proposition~\ref{p.fenchel-moreau}.

\bibliographystyle{abbrv}
\end{document}